\documentclass[article] {memoir}
\usepackage{amsmath, amsthm, amsfonts,amssymb}

\midsloppy 

\settypeblocksize{47\baselineskip}{360pt}{*} 
\setulmargins{1.5in}{*}{*} 
\setlrmargins{*}{*}{1} 
\setheadfoot{\onelineskip}{2\onelineskip} 
\setheaderspaces{*}{1.5\onelineskip}{*}  

\checkandfixthelayout

\theoremstyle{plain}
\newtheorem{proposition}{Proposition}
\newtheorem{theorem}[proposition]{Theorem}
\newtheorem*{theorem*}{Theorem}
\newtheorem*{proposition*}{Proposition}
\newtheorem{corollary}[proposition]{Corollary}

\newtheorem{lemma}[proposition]{Lemma}

\theoremstyle{definition}
\newtheorem{definition}{Definition} 
\newtheorem{exercise}{Exercise} 

\theoremstyle{remark}
\newtheorem{example}{Example}
\newtheorem*{remark}{Remark}

\newcommand{\defeq}{\stackrel{\mathrm{def}}{\, = \,}}

\newcommand{\bN}{{\mathbb{N}}}
\newcommand{\bZ}{{\mathbb{Z}}}

\newcommand{\bQ}{{\mathbb{Q}}}

\begin{document} 

\title{What are discrete valuation rings? What are Dedekind domains?}

\author{A Mathematical Essay by Wayne Aitken}

\author{A mathematical essay by Wayne Aitken\thanks{Copyright \copyright \ 2019 by Wayne Aitken. 
This work is made available under a Creative Commons Attribution 4.0 License.
Readers may copy and redistributed this work under the terms of this license.}}

\date{Fall 2019\thanks{Version of \today.}}

\maketitle

This essay introduces discrete valuation rings (DVRs), and shows that they have several
possible definitions that are equivalent (we call these ``characterizations'').
We will use discrete valuations and other ideas to see that Dedekind domains have several interesting equivalent characterizations as well. Along the way we will study properties of Dedekind domains and their fractional ideals.
For example, we will classify the discrete valuations of a Dedekind domain in terms of its nonzero prime ideals,
and use these valuations to gain a better understanding of fractional ideals.
This essay also introduces local methods to the study of Dedekind domains and integral domains more generally.
We will often work in integral domains, generally or with some constraints, rather than always assuming
we are working in a Dedekind domain.
A purpose of this approach is to open the way to the study of rings that are not quite
Dedekind; for example, Appendix B considers integral domains, such as orders in algebraic number fields,
that are like Dedekind domains except fail to be integrally
closed. Appendix E considers integral domains that a like Dedekind domains except they
fail to be Noetherian.


An intended audience includes readers studying number theory who are familiar with the ring of integers in
an algebraic number field, and who have proved that such rings are Dedekind domains.\footnote{For example, I
envision readers who have studied the relevant chapters of an introductory textbook at the level of
P.~Samuel, \emph{Algebraic Theory of Numbers}, I.~Steward and D.~Tall, \emph{Algebraic Number Theory and Fermat's Last Theorem}, or D.~Marcus, \emph{Number Fields}.}
For such a reader, results about Dedekind domains are mainly applicable in
the context of algebraic number fields.
Such a reader is likely to have proved that every PID is a Dedekind domain (in particular 
that~PIDs, even UFDs, are integrally closed), and perhaps the theorem on the unique factorization of ideals 
into prime ideals in a Dedekind domain. 
For this audience, this essay is not an introduction to Dedekind domains but is,
in a sense, a part two in
the study of Dedekind domains. Even so, I try to make this essay as self-contained as possible.
For example, we will include independent proofs, along the way, that every PID is a Dedekind domain
and that every fractional ideal factors uniquely as the product of powers of prime ideals.
Because it is self-contained in this way,
another possible audience consists of readers interested in the general subject
of commutative algebra and who want to 
explore Dedekind domains and related rings. Such a reader could likely get by with no prior exposure
to Dedekind domains as long as they have enough background to understand the 
concepts used in the definitions,
and is likely not to need extra motivation to pursue the abstract approach followed here.

As the name suggests,
the theory of Dedekind domains owes much to Dedekind's work in the 19th century on ideals 
in the context of algebraic
number fields. The modern version is largely due to Emmy Noether who in the 1920s studied and characterized
Dedekind domains as a general type of ring for which unique factorization of ideals
and other results, including some traditionally associated to integers in an algebraic number field, could be proved.
In a sense, by doing so she invented the field of commutative algebra.
Similarly, this essay does not deal with integers in a number field specifically, but with general integral domains
that satisfy various axiomatic properties.
So this essay is very much in the spirit of Noether's abstract and axiomatic approach.
Of course this essay draws on the work of several other mathematicians in commutative algebra who pioneered
the local approach starting with Wolfgang~Krull in the 1930s. I have also drawn inspiration and ideas from several other expository accounts by
various authors. This material is 
a cornerstone of modern commutative algebra, number theory, and algebraic geometry, and so is pretty standard,
but I have tried to give my own twist on the subject and I hope I have provides some novel viewpoints here and there.

I have attempted to give full and clear statements of the definitions and results,
with motivations provided where possible,
and give indications of any proof that is not straightforward.
However,  my philosophy is that,
at this level of mathematics, straightforward proofs are best 
worked out
by the reader. 
So although this is a leisurely account of the subject, some of the proofs
may be quite terse or missing altogether.
Whenever a proof is not given, this signals to the reader that they
should work out the proof, and that the proof is straightforward. Supplied
proofs are sometimes just sketches, but I have attempted to be detailed enough that the reader can 
supply the details without too much trouble.
Even when a proof is provided, I encourage the reader to attempt a proof first before looking
at the provided proof.  Often the reader's proof will make more sense because it reflects
their own viewpoint, and may even be more elegant.
In addition to this challenge to work out proofs, itself a very good exercise for the reader, I have provided
around~60 labeled exercises so a reader can deepen their knowledge -- these are usually less essential
to the main narrative, so can be skipped in a first reading.

\chapter{Required background}\label{ch1}

This document is written for readers with some basic familiarity with introductory abstract algebra including
some basic facts about groups and their homomorphisms, rings (at least commutative rings), integral domains, fields, polynomial
rings (in one variable), ideals, and at least some exposure to modules.
In this document all rings will be commutative with a unity element.
I assume familiarity with principal, prime and maximal ideals, and with the basics  
concerning PIDs (principal ideal domains:  integral domains whose
ideals are all principal).
For example, readers should be familiar with the fact that every maximal ideal $\mathfrak m$ in a 
commutative ring~$R$ is a prime ideal,
and that~$R/\mathfrak m$ is a field.
In Exercise~\ref{euclidean_exercise}, I also assume familiarity with Euclidean domains, but this
exercise is optional.
For us, a \emph{proper ideal} is any ideal that is not the whole ring.
I assume the result that every proper ideal
is contained in a maximal ideal (we often work in Noetherian domains
where this follows from the ascending chain
condition; for general commutative rings, we would have to use Zorn's lemma).
 I assume familiarity with the result that an ideal is the whole ring
if and only if it contains a unit. 
For us, a \emph{local ring} is a commutative ring with exactly one maximal ideal. I assume the reader knows,
or can verify, that in a local ring~$R$ the unit group~$R^\times$ is just~$R\smallsetminus \mathfrak m$
where $\mathfrak m$ is the maximal ideal, and conversely if the units~$R^\times$ of $R$
are such that $R\smallsetminus R^\times$ is an ideal then $R$ is a local ring
with maximal ideal $\mathfrak m = R\smallsetminus R^\times$.
The reader should also be familiar with the multiplication of ideals  (which we review and
extend in Sections~\ref{ch2.5} and~\ref{ch3}).

I assume that the reader is familiar with Noetherian rings.\footnote{See, for example,
 my short expository essay \emph{Noetherian modules and rings}.} 
 One definition is that a Noetherian ring is a commutative ring such that every ideal is finitely
 generated. So the reader should be familiar with the concept of finitely generated ideals and modules.
 We really only need the Noetherian concept for integral domains, and we use the term \emph{Noetherian domain}
 for an integral domain that is Noetherian.
In Section~\ref{ch3} we will make use of the following result.

\begin{proposition*}
Suppose $R$ is a Noetherian ring.
If $M$ is a finitely generated~$R$-module then 
all $R$-submodules of $M$ are also finitely generated $R$-modules.
\end{proposition*}

We also use the ascending chain condition for Noetherian rings, and the following:

\begin{proposition*}
Let $\mathcal I$ be a collection of ideals of
a Noetherian ring~$R$.
If~$\mathcal I$ is nonempty, then there is a maximal element $I\in \mathcal I$
in the sense that there is no~$I'\in \mathcal I$ with $I \subsetneq I'$.
\end{proposition*}
 
Some authors regard fields as Dedekind domains, some do not.
We will regard fields as Dedekind domains, and will give the following traditional definition:
\begin{definition}\label{dedekind_domain_def}
A \emph{Dedekind domain} is an integral domain~$R$ such that
\begin{enumerate}
\item
$R$ is Noetherian,
\item
$R$ is integrally closed (in the fraction field of~$R$), and
\item
every nonzero prime ideal of~$R$ is maximal.
\end{enumerate}
\end{definition}
As mentioned above, 
I expect many readers to already have some basic familiarity with Dedekind domains, including
the definition and the following results:
\begin{enumerate}
\item
The ring of integers in an algebraic number field
satisfies the above definition of Dedekind domain.
\item
Sometimes unique factorization of elements fails in such Dedekind domains.
\item
But still Dedekind domains generalize PIDs.
In particular, every PID is integrally closed and is in fact a Dedekind domain.
\item
Unique factorization is restored at the level of ideals:
every ideal in a Dedekind domain factors essentially uniquely as the product of prime
ideals.
\end{enumerate}
We don't really use (1) and (2) in this document, but they provide a primary motivation for 
considering Dedekind domains at all, at least for number theorists.
We will use~(3) and~(4)  in our study of DVRs, but will end up giving independent proofs
as we move on to Dedekind domains. So they are not strictly necessary as background from a logical point of view.
In fact we will prove (4) twice, once using local methods in the main body, and then
a second time in Appendix A,
where we consider a standard non-local proof likely similar to proofs that the reader may have seen.
This appendix is provided for the convenience of the reader, and to help the reader compare the two approaches.

Since the notion of integrally closed appears in the definition of Dedekind domain, 
we expect that the reader is 
 familiar with the idea of integral elements over a ring (in terms
of roots of monic polynomials with coefficients in a given ring), and 
the notion of integral closure.

Every integral domain~$R$ is a subfield of its field of fractions~$K$. In essence,~$K$
is the smallest field containing~$R$. We assume the reader is familiar with such \emph{fields of fractions},
which we also call \emph{fraction fields}.
This is a simple example of a localization $S^{-1} R$ of $R$.
Starting with Section~\ref{ch5}, I will assume the 
reader is familiar with the basics of localization, at least in 
 the setting of an integral domain.\footnote{See
 my previous expository essay on localization in integral domains.}
 Since we localize only in an integral domain~$R$, any localization $S^{-1} R$
 can be regarded as an intermediate ring 
 $$R \subseteq S^{-1} R \subseteq K.$$
Here $S$ is a multiplicative system of~$R$.
For us, a multiplicative system of $R$ is a subset closed under multiplication, containing $1$
but not~$0$. 
We assume the reader is also familiar with the localization of ideals of an integral domain.
This includes the relationship between ideals of~$R$ and ideals of the localization,
which is especially simple in the case of prime ideals.
We assume the reader is also familiar with the localization of modules,
at least for~$R$-submodules of the field of fractions~$K$. For such a module~$I$,
the localization~$S^{-1} I$ will also be an $R$-submodules of the field of fractions~$K$.
We also assume familiarity with the following:
 
\begin{proposition*}
Suppose $R$ is an integral closed integral domain and that $S$ is a multiplicative system.
Then the localization $S^{-1} R$ is integrally closed.
 
 Suppose $R$ is a Dedekind domain and that $S$ is a multiplicative system.
 Then the localization $S^{-1} R$ is a Dedekind domain.
 \end{proposition*}
 Actually, we really only require the first claim as background (which is a good exercise
 if a reader has not seen it), since we supply arguments for the rest.
 
Suppose $R$ is an integral domain, and $x \in K$ where $K$ is the field of fractions.
Then we can form the ring extension of~$R$ generated by $x$.
We write this as $R[x]$, and since $x \in K$ this will be a subring of $K$.
 We use this construction in Section~\ref{ch4}.
On the other hand $R[X]$ with a upper-case $X$ will denote the ring of polynomials
with coefficients in~$R$. The rings $R[x]$ and $R[X]$ are related: $y \in R[x]$ if and only if there is a polynomial
$f\in R[X]$ such that $y = f(x)$.
A key result is that $x$ is integral over~$R$ if and only if $R[x]$ is a finitely generated $R$-module.
We will supply an argument when needed that finitely-generated here implies integral.
The other direction is more straightforward: 
the reader should be able to show that if~$x$ is integral and satisfies a monic polynomial of degree $d$
then $R[x]$ is generated by~$1, x, \ldots, x^{d-1}$ (one way is to show by induction that
$1R+ xR + \ldots+ x^{d-1}R$ contains the power~$x^{d+k}$ for all $k\ge 0$).

When we get to the Chinese remainder theorem in Section~\ref{approximation_theorem_ch} we will need to work
with Cartesian products of rings $R_1 \times R_2$. The reader should know that the product of rings is a ring
under componentwise operations. (Although for our application we will only need to know that it is a group
under addition).

If $I$ is an ideal in an integral domain $R$ (or more generally, a commutative ring) then there is natural surjective ring homomorphism 
$R[X] \to (R/I)[X]$
between polynomial rings that sends $X$ to~$X$.
It essentially acts by replacing each coefficient with its equivalence class mod~$I$.
We use this homomorphism in Section~\ref{ch_gauss_lemma} to help prove Gauss's lemma in Dedekind domains.


\chapter{Discrete valuation rings} \label{ch2}

In this section we will consider several equivalent characterization of discrete valuation rings.
We start by thinking of such rings as  rings arising from discrete valuations.

\begin{definition}
A \emph{discrete valuation} of a field $K$ is a surjective homomorphism
$$
v\colon K^\times \to \bZ
$$
between the multiplicative group $K^\times$ of the field and the additive group $\bZ$
such that the following law holds for sums: for all $x, y \in K^\times$, if $x + y \ne 0$
then
$$
v(x + y) \ge \min\{v(x), v(y)\}.
$$
\end{definition}

\begin{remark}
Observe that valuations satisfy a multiplicative law $v(xy) = v(x) + v(y)$ and an additive 
law $v(x + y) \ge \min\{v(x), v(y)\}$. We can extend these laws to all 
elements~$x, y \in K$ by defining  $v(0) = \infty$ and then extending addition
and the order relation  to~$\bZ\cup \{ \infty\}$  in the obvious way. In particular, we can then remove the restriction
$x+y \ne 0$
in the above definition. 
\end{remark}

\begin{exercise}
Suppose $w\colon K^\times \to \bZ$ satisfies the above, except for the assumption of surjectivity.
Assume instead that $w$ has nontrivial image.
Show that~$v \defeq \frac{1}{e} w$ is a discrete valuation where $e$ is the smallest positive number in the image of $w$.
 Such a $w$ is called a \emph{unnormalized discrete valuation}, and sometimes
for emphasis what we call a discrete valuation is called a \emph{normalized discrete valution}.
\end{exercise}

\begin{exercise}
Let $F[X]$ be the ring of polynomials over a field~$F$. Can the degree map be extended into
a valuation of the field of fractions $F(X)$ of~$F[X]$?
What about $(-1)$ times the degree map? 
\end{exercise}

\begin{lemma}
If $v$ is a discrete valuation then $v(1) = 0$, $v(-1) = 0$, and $v(-x) = v(x)$ for all $x \in K^\times$.
\end{lemma}

In the following we use the convention $v(0) = \infty$ mentioned above.
 (We can easily rephrase the definitions and statements
to avoid this though, but it is a convenient convention). 

\begin{proposition} \label{prop1}
Let $v\colon K^\times \to \bZ$ be a discrete valuation.
Then
$$
\mathcal O_v\; \defeq \;\{ x \in K \mid v(x) \ge 0\}
$$
is a local integral domain, with maximal ideal
$$
\mathfrak p_v\; \defeq\; \{ x \in K \mid v(x) \ge 1\}
$$
and unit group
$$
\mathcal O^\times_v = \{ x \in K \mid v(x) = 0\}.
$$
Also $x\in \mathcal O_v$ or $x^{-1} \in \mathcal O_v$
for all $x \in K^\times$, so $K$ is the fraction field of $\mathcal O_v$.
\end{proposition}

If $v(x) \ne v(y)$ then the inequality $v(x + y) \ge \mathrm{min}(v(x), v(y))$
becomes an equality:

\begin{proposition}
Let $v\colon K^\times \to \bZ$ be a discrete valuation
which we extend to $0 \in K$ using the convention $v(0) = \infty$.
If $x, y \in K$ are such that $v(x) \ne v(y)$ then
$$
v(x + y) = \min\{v(x), v(y)\}.
$$
\end{proposition}

\begin{proof}
Suppose, say, that $v(y) > v(x)$.
We outline two arguments for the result.

In the first argument, observe that $x$ cannot be zero. 
By dividing $x$ and $y$ by~$x$,
we reduce to the case $x_1 = x/x = 1$ and $y_1 = y/x$.
Since $v(y_1) > v(x_1) = 0$, we have $y_1 \in \frak p_v$.
Thus $1 + y_1$ is a unit in $\mathcal O_v$, so $v(1 + y_1) = 0 = v(1)$.

For the second argument, suppose that $v(x + y) \ne v(x)$. Thus $v(x+y) > v(x)$.
Then
$$
v(x) = v((x+y) - y) \ge \min\{ v(x+y), v(y) \} > v(x),
$$
a contradiction.
\end{proof}

\begin{corollary}
Let $v\colon K^\times \to \bZ$ be a discrete valuation
which we extend to $0 \in K$ using the convention $v(0) = \infty$.
Suppose $x_1, \ldots, x_k \in K$ are such that~$v(x_1)$ is strictly less than
$v(x_i)$ for all $i \ne 1$. Then
$$
v(x_1 + \ldots + x_k) = v(x_1).
$$
In particular, $x_1 + \ldots + x_k$ is nonzero.
\end{corollary}

\begin{proof}
Let $x = x_1$ and $y = x_2 + \ldots + x_k$. Use the above proposition.
\end{proof}

\begin{definition}\label{def2}
A \emph{discrete valuation ring} (DVR) is any ring of the form
$$
\mathcal O_v = \{ x \in K \mid v(x) \ge 0\}
$$
where $K$ is a field and where
$v \colon K^\times \to \bZ$ is a discrete valuation.
\end{definition}

\begin{definition}
Let $v\colon K^\times \to \bZ$ be a discrete valuation. If $\pi \in K^\times$ is such that
$$v(\pi) = 1$$
then we call $\pi$ a \emph{uniformizer} of $v$.
Since discrete valuations are surjective, such uniformizers exist.
\end{definition}

The following shows the usefulness of a uniformizer in terms of unique factorization.

\begin{proposition}
Let $v\colon K^\times \to \bZ$ be a discrete valuation with uniformizer $\pi$.
Then every element $x$ of $K^\times$ can be written uniquely as
$$
x= u \pi^k
$$
where $k\in \bZ$ and where $u$ is a unit in~$\mathcal O_v$.
When $x$ is written in this form, $v(x) = k$.
\end{proposition}

\begin{remark}
This shows that every DVR is a simple type of UFD.
\end{remark}

Given a nonzero ideal $I$ of $\mathcal O_v$ we must have a smallest value $v(x)$ among
elements~$x \in I$. This can be used to show the following:

\begin{proposition}
Let $v\colon K^\times \to \bZ$ be a discrete valuation with uniformizer $\pi$.
The nonzero ideals of
the ring~$\mathcal O_v$
are $\pi^k \mathcal O_v$ where $k\ge 0$. These ideals are distinct and
$$
\pi^k \mathcal O_v = \left \{ x \in K \mid v(x) \ge k \right\}.
$$
An element of $K$ is a uniformizer if and only if it generates $\frak p_v$ as an ideal.
\end{proposition}

\begin{corollary}\label{cor6}
Every DVR is a PID with a unique nonzero prime ideal.
\end{corollary}

\begin{exercise}\label{euclidean_exercise}
Show that every DVR is a Euclidean domain.
\end{exercise}

Recall that an integral domain $R$ is integrally closed if, for any monic polynomial~$f \in R[X]$, every 
root of $f$ in the field of fractions of~$R$ is actually in~$R$.

\begin{proposition}\label{every_dvr_is_integrally_closed_prop}
Every DVR is integrally closed.
\end{proposition}

\begin{proof}
We can appeal to the theorem that every PID or even every UFD is integrally closed. However, there is a nice
proof that uses properties of valuations.
Let~$\mathcal O_v$ be a~DVR with valuation $v: K^\times \to \bZ$ 
where $K$ is the fraction field of~$\mathcal O_v$.
Let~$f$ be a monic polynomial in~$\mathcal O_v[X]$ with root $x \in K$. 
We wish to show $x \in \mathcal O_v$, so we assume otherwise:
suppose that $x \not\in \mathcal O_v$, so $v(x) = - k < 0$.
We can express the equation $f(x) = 0$ as follows:
$$
x^d = a_{d-1} x^{d-1} + \ldots + a_1 x + a_0
$$
where each $a_i \in \mathcal O_v$. 
The valuation of the left hand side is $-d k$, and the valuation of each  term on the right
is greater than or equal to $-(d-1) k$. So the right hand side must have 
strictly greater valuation than the left hand side, a contradiction.
Thus $v(x) \ge 0$, and $x \in \mathcal O_v$.
\end{proof}

\begin{exercise}\label{ex3}
Suppose $v_1, v_2 \colon K^\times \to \bZ$ are two valuations. Show that if $\mathcal O_{v_1}  = \mathcal O_{v_2}$
then $v_1 = v_2$.
Conclude that $v \mapsto \mathcal O_v$ is a bijection between the set of discrete valuations of a given field~$K$ and the set of discrete valuation
rings with field of fractions $K$.
\end{exercise}

We have established various properties of discrete valuation rings. Some of these properties are actually
sufficient for an integral domain to be a discrete valuation ring. We start with one such property:

\begin{proposition}\label{propA}
Suppose that $R$ is an integral domain and suppose $\pi$ is a nonzero element of~$R$ that is not a unit.
Suppose that every nonzero element of $R$ can be written as $u \pi^k$ with $k\ge 0$ and with $u \in R^\times$.
Then the fraction field~$K$ of $R$ has the property that every nonzero element can
be written uniquely as $u \pi^k$ with $k\in \bZ$ and with $u \in R^\times$.
The function $u \pi^k \mapsto k$ defines a discrete valuation~\text{$K^\times \to \bZ$},
and~$R$ is the associated discrete valuation ring with uniformizer~$\pi$.
\end{proposition}

\begin{corollary}\label{cor9}
Suppose that $R$ is an integral domain.
Then $R$ is a DVR if and only if 
there is a non-unit, nonzero $\pi \in R$ such that every nonzero element of $R$ can be written as $u \pi^k$
for some unit $u\in R^\times$ and some $k\in \bN$.
\end{corollary}

Building on the above results, we can prove a deeper theorem:

\begin{theorem}\label{thmB}
Suppose that $R$ is a local Noetherian domain whose maximal ideal~$\frak m$ is nonzero and principal. Then
$R$ is a DVR. Conversely, every DVR is a local Noetherian domain with a nonzero principal maximal ideal.
\end{theorem}

\begin{proof}
Write $\frak m$ as $\pi R$.
Observe that if $x \in R$ is not a unit then $x$ can be written as~$x_1 \pi $ for some $x_1 \in R$. 
If $x_1$ is not a unit
then we can write~$x =  x_1\pi =  x_2 \pi^2$ for some $x_2 \in R$. And so on.
There are two possible types of elements in~$R$: (1) those for which this process stops and so can be written
as $u \pi^k$ for some unit~$u$ and~$k \ge 0$, and (2)  those
for which the process does not stop, and so
we can write such an element~$x$ as~$x_k \pi^k$ for all~$k\in \bN$.
Call the second type of element ``exceptional''. Of course $0$ is exceptional. We wish to show that $0$ is the only exceptional
element.

First we show that if $x = x_1 \pi \ne 0$ then $xR \subsetneq x_1 R$. Of course $xR \subseteq x_1 R$ holds,
but suppose $xR = x_1 R$. Then $x_1 = x y$ for some $y\in R$. In other words,~\text{$x =x y \pi $}.
We write this as $x(1-y\pi) = 0$. However, $1 - \pi y$ is a unit. Thus $x = 0$, a contradiction.

So, let $x$ be exceptional and nonzero, and write
$$
x = x_1 \pi = x_2 \pi^2 = x_3 \pi^3 = \ldots
$$
with $x_i = x_{i+1} \pi$.
By the above claim,
$$
x R \subsetneq x_1 R \subsetneq x_2 R \subsetneq  x_3 R \subsetneq  \ldots.
$$
This contradicts the Noetherian assumption. Thus the only exceptional element is~$0$.
By Corollary~\ref{cor9} this means that $R$ is a DVR. (The converse is straightforward).
\end{proof}

\begin{remark}
Serre points out in his \emph{Local Fields} (Chapter I, \S 2) that we do not need to assume in advance that $R$ is
an integral domain: we can replace that assumption with the assumption that $R$ is a commutative ring
(with unity) such that $\pi$ is not a nilpotent element~($\pi^k \ne 0$ for all $k$).
He goes on to say that the proof is much simpler if $R$ is assumed to be an integral domain.
However, if one uses the proof given here, the general case is not much harder:
just add the observation (after showing that $0$ is the only exceptional element) that~$(u \pi^k)(v \pi^l) \ne 0$ if $u, v$ are units, 
so $R$ is an integral domain.
\end{remark}

\begin{exercise}
Check that the above proof can be modified, as mentioned in the remark, to prove the following:
\emph{Let $R$ be a local Noetherian ring whose maximal ideal is $\pi R$ where $\pi$
is not a nilpotent element ($\pi^k \ne 0$ for all $k\in \bN$). Then
$R$ is a discrete valuation ring. }
\end{exercise}

We can use Theorem~\ref{thmB} to strengthen Corollary~\ref{cor6}.

\begin{corollary} \label{cor11}
Let $R$ be an integral domain. Then the following are equivalent:
\begin{enumerate}
\item
$R$ is a DVR.
\item
$R$ is a PID with a unique nonzero prime ideal. 
\item
$R$ is a PID with a unique nonzero maximal ideal. 
\end{enumerate}
\end{corollary}

Recall that every PID is a Dedekind domain (but not every Dedekind domain is a~PID).
The above corollary actually can be generalized to all Dedekind domains
with a unique nonzero prime ideal.
If we are willing to accept the unique factorization of ideals result, the proof is simple.
Later in the document we will pursue another approach that leads to a proof
that does not rely on the unique factorization of ideals.

\begin{theorem}\label{thmC}
Let $R$ be an integral domain.
Then $R$ is a discrete valuation ring if and only if 
$R$ is a Dedekind domain with a unique nonzero prime ideal.
\end{theorem}

\begin{proof}
Let $R$ be a DVR.
We can appeal to the theorem that every PID is a Dedekind domain.
However, we can give a direct proof that uses
Proposition~\ref{every_dvr_is_integrally_closed_prop}
to show that~$R$ is integrally closed. The ring $R$ has only one nonzero prime ideal, and it is maximal. Finally, $R$ is Noetherian since $R$ is a PID.
Thus $R$ is a Dedekind domain.

For the other direction, 
suppose $R$ is a Dedekind domain with unique nonzero prime ideal~$\mathfrak p$.
By the unique factorization theorem for ideals in a Dedekind domain, we have that the nonzero ideals of $R$
are all uniquely of the form $\frak p^k$. If $\pi \in \frak p \smallsetminus \frak p^2$ then~$\pi R$
cannot be $\frak p^k$ with $k\ge 2$, so $\pi R = \frak p$.
By definition of Dedekind domain,~$R$ is a Noetherian domain. So we apply Theorem~\ref{thmB}.

(Another approach, once we admit fractional ideals, is to define the valuation of the fractional ideal $\frak p^k$ to be $k$, where $k \in \bZ$.
Next define the valuation of $\alpha\ne 0$ in the field of fractions $K$ to be the valuation
of the fractional ideal~$\alpha R$.)
\end{proof}

The proof of the above result suggests another charactarization. Call a ring \emph{strongly local} (to coin a new term) if
it is local and if there
is a largest proper nonmaximal ideal $\mathfrak n$ in the sense that all proper nonmaximal ideals are contained
in~$\mathfrak n$.
For example, a DVR is strongly local whose largest proper nonprime ideal is~$\mathfrak n = \frak p^2$.
A DVR is also Noetherian. These two properties characterize DVRs:

\begin{proposition}\label{propH}
Suppose $R$ is an integral domain.
Then $R$ is a discrete valuation ring if and only if 
$R$ is a strongly local Noetherian domain.
\end{proposition}

\begin{proof}
We discussed already one implication. For the other implication,
let $\frak m$ be the maximal ideal, and let $\mathfrak n$ 
be the largest proper nonmaximal ideal. Let~$\pi \in \frak m \smallsetminus \mathfrak n$.
Then $\frak m = \pi R$. Now apply Theorem~\ref{thmB}.
\end{proof}

Next we consider a fairly simple characterization of DVRs among Noetherian domains:

\begin{proposition}\label{propI}
Let $R$ be a Noetherian domain with fraction field~$K$. Assume~$R$ is not a field.
Then~$R$ is a DVR if and only if $x$ or $x^{-1}$ is in~$R$ for each $x \in K^\times$.
\end{proposition}

\begin{proof}
One direction is straightforward,
so assume $x$ or~$x^{-1}$ is in~$R$ for each $x \in K^\times$.
We restate this assumption as follows:
if $a, b \in R$ are nonzero then either~$a/b$ or~$b/a$ is in~$R$.
In other words, $a \in b R$ or $b \in a R$. So $aR \subseteq bR$ or $bR \subseteq aR$.
Thus the set of principal ideals in $R$ are totally ordered by inclusion.

Since $R$ is Noetherian, there is a maximal element $\pi R$ among the set of proper nonzero principal
ideals of~$R$. Since the set of principal ideals in $R$ are totally ordered, $\pi R$ is the maximum proper
principal ideal. So for each non-unit~$a \in R$ we have $a R \subseteq \pi R$. Thus $a \in \pi R$.
We conclude that $\pi R$ is the set of non-units.
Thus~$\pi R$ is the unique maximal ideal.
Now apply Theorem~\ref{thmB}.
\end{proof}

%

\begin{exercise}\label{valuation_ring_ex}
A \emph{valuation ring} is an integral domain~$R$ such that for all $x \ne 0$ in the field of fractions~$K$
either $x$ or $x^{-1}$ is in~$R$. The above proposition shows that a Noetherian valuation ring is a DVR.
Let us consider valuation rings that are not necessarily Noetherian. 
(1) Show that the collection of principal ideals in such a valuation ring~$R$ is totally ordered by inclusion.
(Hint: see the proof of the above proposition).
(2) Show that any finitely generated ideal in such a valuation ring $R$ is principal. 
(3) Show that every valuation ring is local. 
\end{exercise}

\begin{exercise}
Suppose $R$ is a commutative ring such that the collection of principal ideals is
totally ordered by inclusion. Show that the collection of \emph{all} ideals is also totally ordered by inclusion.
Based on the previous exercise, conclude that the
collection of ideals of a valuation ring are totally ordered by inclusion.
\end{exercise}

\begin{exercise}
Suppose $R$ is an integral domain such that the collection of principal ideals is
totally ordered by inclusion. Show that $R$ is a valuation ring.
\end{exercise}

\begin{exercise}\label{valuation_ring_ex2}
Let $R$ be a valuation ring with nonzero maximal ideal $\mathfrak m$.
Show that either $\mathfrak m = \mathfrak m^2$ or
$\mathfrak m$ is principal.
Hint: let $\pi \in \mathfrak m \smallsetminus \mathfrak m^2$ (if such exists), and use Exercise~\ref{valuation_ring_ex}
to argue that $a R \subseteq \pi R$ for all non-units $a\in R$.
\end{exercise}

Now we assume the reader is comfortable with the concept of the product of ideals. (If not the reader
can skip ahead to the next section where this concept is reviewed
in some generality.) This concept leads naturally to the notion of divisibility:

\begin{definition}
Let $I$ and $J$ be two ideals of an integral domain.
We say that $I$ \emph{divides} $J$ if there is an ideal $I'$
such that $I I' = J$. In this case we write $I \mid J$.
\end{definition}

\begin{proposition}
Let $I$ and $J$ be ideals of an integral domain. If $I \mid J$ then~$J \subseteq I$.
\end{proposition}

Throughout this document, we will be interested in the situation where $J \subseteq I$
guarantees that $I \mid J$. Clearly this happens in DVRs given the explicit description of ideals.
We now consider two results related to divisibility in the case of principal ideals in local integral domains.

\begin{lemma}
Let $R$ be a local integral domain with maximal ideal~$\mathfrak m$. 
Suppose $I$ and~$J$ are ideals with $I J = a R$
where $a\in R$. Then there are elements $b \in I$ and~$c \in J$ such that $b c = a$.
\end{lemma}

\begin{proof}
We jump to the case $a\ne 0$, and we
write $a$ as a finite sum: $a = \sum b_i c_i$ with each  $b_i \in I$ and~\text{$c_i \in J$}.
Since each $b_i c_i \in aR$, we can write $b_i c_i = a u_i$ with $u_i\in R$.
Observe that $1 = \sum u_i$, so is not possible for each $u_i$ to be in~$\mathfrak m$. Thus $u_i$ is a unit
for at least one value of~$i$.
For such $i$ we have $(u_i^{-1} b_i) c_i = a$ as desired.
\end{proof}

\begin{proposition} \label{divisibility_principal_prop}
Suppose $I$ is an ideal in a local integral domain. If $I$ divides a
nonzero principal ideal then $I$ is itself principal.
\end{proposition}

\begin{proof}
We have $I J = a R$ for some ideal $I$ and some nonzero $a \in R$. By the previous lemma
we also have $b c = a$ where $b \in I$ and $c \in J$. Clearly $b R \subseteq I$. We claim
that in fact $I = b R$.

Let $d \in I$. Observe that $d c \in aR$, so $dc = ar$ for some $r\in R$.
Thus
$$
a d = (b c) d = (d c) b = ar b.
$$
Since $a$ is nonzero, $d = r b \in b R$ as desired.
\end{proof}

\begin{theorem}\label{divisibility_DVR_thm}
Suppose that $R$ is an integral domain.
Then $R$ is a~DVR if and only if $R$ has a unique nonzero maximal ideal
and has the property that  $I \mid J$ for all ideals~$I$ and $J$
with $J \subseteq I$.
\end{theorem}

\begin{proof}
One direction is straightforward given our description of ideals of a DVR.
So suppose $R$ has a unique maximal ideal, and
that  $I \mid J$ for all ideals $I$ and $J$
with $J \subseteq I$. Let $I$ be a nonzero ideal, and let $a \in I$ be nonzero.
Since $a R \subseteq I$ we have that $I$ divides $a R$. So by the previous proposition
we have that $I$ is principal. Thus $I$ is a PID, which means it is a DVR
by Theorem~\ref{thmB}.
\end{proof}

In summary, we have shown that we can characterize DVRs in myriad ways.
Let~$R$ is an integral domain with fraction field~$K$. Then any of the following
provides a necessary and sufficient condition for $R$ to be a DVR.
\begin{enumerate}
\item 
$R$ is $\mathcal O_v$ for a discrete valuation $v \colon K^\times \to \bZ$. (Definition~\ref{def2})
\item
There is a non-unit, nonzero $\pi \in R$ such that every nonzero element of $R$ can be written as $u \pi^k$
for some unit $u\in R^\times$ and some $k\in \bN$. (Corollary~\ref{cor9})
\item
$R$ is local and Noetherian, with a principal, nonzero maximal ideal. (Theorem~\ref{thmB})
\item
$R$ is a PID with a unique nonzero prime ideal. (Corollary~\ref{cor11})
\item
$R$ is a PID with a unique nonzero maximal ideal. (Corollary~\ref{cor11})
\item
$R$ is a Dedekind domain with a unique nonzero prime ideal. (Theorem~\ref{thmC})
\item
$R$ is a strongly local Noetherian domain. (Proposition~\ref{propH})
\item
$R$ is Noetherian, not a field, and $x$ or $x^{-1}$ in $R$ for all $x$ in the field of fractions of~$R$. (Proposition~\ref{propI})
\item
$R$ has a unique nonzero maximal ideal and has the property that  $I \mid J$ for all ideals~$I$ and $J$
with $J \subseteq I$. (Theorem~\ref{divisibility_DVR_thm})
\item
$R$ is local and Noetherian with invertible maximal ideal. 
(Definitions and proof below, see especially Theorem~\ref{theoremI} or the remarks after Proposition~\ref{propInv})
\end{enumerate}
The first nine summarize the previous material. The last of these will be proved after we review the theory of fractional ideals.\footnote{There is also an 11th characterization.  Exercise~\ref{dimension_ex} gives the following necessary and sufficient condition:
$R$ is local and Noetherian with maximal ideal~$\mathfrak m$ having the property that
$\mathfrak m/\mathfrak m^2$ is a vector space of dimension 1 over the field $R/\mathfrak m$.
This exercise can be skipped in a first reading.}

If we begin with a Noetherian domain~$R$ with a  unique nonzero prime ideal $\mathfrak p$, then $R$ is a DVR
if and only if one, and hence all, of the following hold:

\begin{enumerate}
\item
$\mathfrak p$ is principal (Theorem~\ref{thmB})
\item
$R$ is a PID. (Corollary~\ref{cor11})
\item
$R$ is a Dedekind domain (i.e., $R$ is integrally closed). (Theorem~\ref{thmC})
\item
$R$ is a strongly local. (Proposition~\ref{propH})
\item
$x$ or $x^{-1}$ in $R$ for all $x$ in the field of fractions of~$R$. (Proposition~\ref{propI})
\item
$R$ has the property that  $I \mid J$ for all ideals~$I$ and $J$
with $J \subseteq I$. (Theorem~\ref{divisibility_DVR_thm})
\item
$\frak p$ is invertible. (Definitions and proof below, see especially Theorem~\ref{theoremI})
\end{enumerate}
(Exercise~\ref{dimension_ex} adds a eighth condition:
$\mathfrak p/\mathfrak p^2$ is a vector space of dimension 1 over the field $R/\mathfrak p$.)

Finally, we observe that DVRs are maximal proper subrings of their fraction fields.

\begin{proposition}\label{prop15}
Suppose $R$ is a DVR with field of fractions $K$. If $R \subsetneq R' \subseteq K$ and if $R'$ is a subring of~$K$,
then $R' = K$.
\end{proposition}

\chapter{Submodules of the fraction field} \label{ch2.5}

Our next goals are to (1) make good on the promise to
show that any local Noetherian domain with invertible maximal ideal is
a DVR which includes explaining the ideal of invertible, and 
(2) give another proof of Theorem~\ref{thmC} that does not rely on the pre-acceptance of the
unique factorization theorem for Dedekind domains. We start by
explaining the concept inverse for ideals.
This concept will help us meet these two immediate goals, but will 
also be useful for better understanding Dedekind domains and similar rings.
(Much of this section is likely review, but the reader should verify for themselves
any unfamiliar result.)

Let $R$ be an integral domain with fraction field $K$. 
The collection of nonzero ideals of~$R$ forms a commutative monoid under the product
operation with~\text{$I=R$} providing the identity element. When $R$ is a Dedekind
domain this monoid can be expanded into a group by adding fractional ideals
to the monoid. Such fractional ideals will provide (multiplicative) inverses for nonzero ideals.
We will define fractional ideals in the next section, but for now we
mention that they are a kind of~$R$-submodule of~$K$.
So in order to prepare the way for fractional ideals, we will first discuss~$R$-submodules of~$K$
in general.

We assume the reader has at least a basic familiarity with modules in general, but we will
focus on $R$-submodules of~$K$. 
Recall that $K$ is an $R$-module. In fact, any ring containing $R$ is an $R$-module.
Every ideal of $R$ is an $R$-submodule of~$K$, and
an  $R$-submodule $I$ of $K$ is an ideal of~$R$ if and only $I$ is contained in~$R$.
(Note we will use letters such as $I$ and $J$ for $R$-submodules of~$K$ since we are thinking 
of such submodules as a straightforward generalization of ideals.)

Recall that if $I_1$ and $I_2$ are $R$-submodules of~$K$ (or submodules of any fixed module)
then we can define the sum
$$
I_1 + I_2 \; \defeq \; \left\{ x_1 + x_2 \mid x_1 \in I_1, x_2 \in I_2 \right\}.
$$
The sum is an $R$-submodule of~$K$. This sum is commutative and associative.
The zero submodule is the identity. So we get an additive monoid of $R$-submodules of~$K$.

The intersection $I_1 \cap I_2$ of two $R$-submodules of~$K$ is itself an $R$-submodule.
This operation is commutative and associative, and $K$ is the identity. So we get a 
commutative monoid of $R$-submodules of~$K$ under intersection.

The operations of addition and intersection are meaningful for the collection of submodules of any fixed
$R$-module $M$, and do not make special use of our case where $M = K$. However we will be especially
concerned with an operation, the product of submodules, that uses the fact that $M=K$
has a product. The product of $R$-submodules $I_1, I_2$ of $K$ is defined as follows:
$$
I_1  I_2 \; \defeq \; \left\{ \sum_{i} x_i y_i \; \middle| \;
\text{where each $x_i \in I_1$ and where each $y_i \in I_2$} \right\}.
$$
Here the sums are finite sums, and a sum with zero terms is defined here
to be~$0$.
This operation results in an $R$-submodule of~$K$. This operation is commutative and associative,
and the ideal $I = R$ is the identity. Thus we get a commutative monoid.
Note also that $I_1 I_2$ is the minimum $R$-submodule of $K$ (under inclusion) containing all
products $x y$ where $x\in I_1$ and~$y \in I_2$. So we can think of it as the submodule generated
by such products. (Of course, we can similarly think of $I_1 + I_2$ as the submodule generated by sums
of the form~$x+y$ with~$x\in I_1, y\in I_2$.)

The operations of sum, intersection, and product are monotonic in the sense that if $I_1 \subseteq I_1'$
then 
$$
I_1 + I_2 \subseteq I_1'+I_2, \qquad I_1 \cap I_2 \subseteq I_1'\cap I_2, \qquad I_1  I_2 \subseteq I_1'I_2.
$$
We also have a distributive law
$$
I ( J_1 + J_2) = I J_1 + I J_2.
$$

Given $x \in K$ then 
$$
x R\; \defeq\; \{ x r \mid r\in R \}.
$$
We sometimes write $Rx$ for $xR$ (since $R$ is an integral domain, it is commutative).
Observe that $x R$ is an  $R$-submodule of $K$; it is called a \emph{principal submodule}.
It is the minimum among submodules of $K$ (under inclusion) containing $x$, and so is sometimes
called the \emph{$R$-submodule of $K$ generated by~$x$}.
We extend this notation a bit: given $x \in K$ and an $R$-submodule  $I$ of $K$ then
$$
x I\; \defeq \; \{ x y  \mid y\in I \}.
$$
This results in an $R$-submodule of~$K$.
We have identities and properties such as 
$$
x(y I) = (xy)I, \quad x I = (x R) I, \quad (xR)(yR) = (xy)R \quad
\text{and}\quad (I \subseteq J \implies x I \subseteq x J).
$$
The identity $(xR)(yR) = (xy)R$ implies that the
collection of principal submodules is closed under multiplication.
Obviously $R = 1 R$ is principal, so 
the collection of principal submodules forms a submonoid of the collection of all
$R$-submodules of~$K$.
The collection of nonzero principal submodules also forms a submonoid of the collection of all
$R$-submodules of~$K$.

Given $x_1, \ldots, x_k \in K$, there is a minimum $R$-submodule of~$K$ containing these elements.
It is
$$
x_1 R + \ldots + x_k R.
$$
(Here minimum is with respect to inclusion).
Given an infinite subset $U$ of $K$, we can also form
a minimum $R$-submodule of~$K$ containing $U$: just take the intersection of
all $R$-submodules that contain~$U$. Its elements
are all the finite $R$-linear combinations of elements of $U$. We will not
need the infinite case in this document.

\begin{definition}
Suppose $I$ is an $R$-submodule of $K$. We say that $I$ is \emph{invertible}
if there is an $R$-submodule $J$ of $K$ such that 
$$
I J = R.
$$
In this case $J$ is called the \emph{inverse} of $I$.
\end{definition}

\begin{proposition} \label{easy_invert_prop}
Let $R$ be an integral domain with field of fractions~$K$.
If $x \in K^{\times}$ then $x R$ is invertible with inverse $x^{-1}R$.
So if $R$ is a PID them every nonzero ideal is invertible.

If $I$ and $J$ are $R$-submodules of $K$ such that $I J = x R$
for some $x \in K^{\times}$, then~$I$ and $J$ are invertible.
More generally, $I$ and $J$ are both invertible if and only if $IJ$ is invertible.
\end{proposition}

Since a DVR is a PID, its maximal ideal~$\frak m$ is invertible (along with all nonzero ideals). This is one direction
of our promised result.
We are also ready to prove the more substantial converse
 that a local Noetherian domain with invertible maximal ideal is a DVR.
All we need here is (1) the notion of invertibility, and (2) our previous result on strongly local Noetherian domains.

\begin{theorem}\label{theoremI}
Let $R$ be an integral domain.
Then $R$ is a discrete valuation ring if and only if 
$R$ is a local Noetherian domain whose maximal ideal is invertible. 
\end{theorem}

\begin{proof}
If $R$ is a DVR, then it is a PID, and all nonzero principal ideals are invertible.

Now suppose $R$ is a local Noetherian domain with maximal ideal~$\mathfrak m$ and suppose there is an $R$-submodule~$\mathfrak m^{-1}$ of~$K$
with $\mathfrak m \mathfrak m^{-1} = R$.
To show that $R$ is a DVR, it is enough,
by Proposition~\ref{propH}, 
to show that $R$ is strongly local.
Let $I$ be a proper nonmaximal ideal.
From $I \subseteq \mathfrak m$ we have
$\mathfrak m^{-1} I \subseteq R$.
However, $\mathfrak m^{-1} I =R$ cannot hold since~$I \ne \frak m$. So~$\mathfrak m^{-1} I \subseteq \frak m$.
Hence~$I \subseteq \frak m^2$. 
We conclude that $R$ is strongly local since every proper nonmaximal ideal is contained in~$\frak m^2$.
\end{proof}

\begin{remark}
See the remark after Proposition~\ref{propInv} for another short argument.
\end{remark}

If an $R$-submodule of $R$ is invertible, then its inverse has the following explicit description:

\begin{proposition}\label{inverse_formula_prop}
Suppose $IJ = R$ where $I$ and $J$ are 
$R$-submodules of~$K$. In other words, suppose $I$ and $J$ are inverses. 
Then
$$
J = \left\{x \in K \mid xI \subseteq R \right\}.
$$
\end{proposition}

\begin{proof}
The direction $J \subseteq  \left\{x \in K \mid xI \subseteq R \right\}$ is straightforward.

Suppose that~\text{$x I \subseteq R$}. In other words, $(x R) I \subseteq R$. So
$(xR) I J \subseteq R J$.
Thus~$x R \subseteq J$, and so $x \in J$.
\end{proof}

Proposition~\ref{inverse_formula_prop} implies
that $\left\{x \in K \mid xI \subseteq R \right\}$ is the unique candidate for the inverse of $I$.
It also suggests a necessary criterion for invertibility. Observe that in the above proposition
neither $I$ nor $J$ can be the zero module. This forces $J$ to contain some nonzero $x$.
In other words, there is an $x \in K^{\times}$ such that $x I \subseteq R$.

\begin{example}
As long as $R$ is not all of $K$, the module
$I = K$ is not invertible. Otherwise, there is an $x \in K^{\times}$ with $x K \subseteq R$. 
But clearly $x K = K$, a contradiction.
\end{example}

In spite of the fact that $I$ may not be invertible, we will still 
label~$\left\{x \in K \mid xI \subseteq R \right\}$ as $I^{-1}$
with the caveat that although it is the only possible inverse, it may fail to be the inverse
simply because no inverse of $I$ exists.

\begin{definition}
Let $R$ be an integral domain and let $K$ be its fraction field.
If~$I$ is an $R$-submodule of $K$ then 
$$I^{-1} \;\defeq\; \left\{ x \in K \mid x I \subseteq R \right\}.$$
\end{definition}

\begin{proposition}
Let $R$ be an integral domain, and let $K$ be its fraction field.
If~$x \in K^\times$ then 
$$
(x R)^{-1} = x^{-1} R.
$$
\end{proposition}

\begin{proof}
Recall that $(xR)^{-1}$ is the inverse of $x R$ if an inverse exists.
But $xR$ is in fact invertible with inverse $x^{-1} R$.
\end{proof}

\begin{proposition}
Let $R$ be an integral domain and let $K$ be its fraction field.
If~$I$ is an $R$-submodule of $K$ then the following hold:
\begin{itemize}
\item
$I^{-1}$ is an $R$-submodule of~$K$.
\item
$I I^{-1}$ is an ideal of $R$.
\item
In fact, $I I^{-1}$ is the maximum (for inclusion) among ideals of $R$
of the form~$IJ$ where~$J$ is an $R$-submodule of~$K$.
\end{itemize}
If $I$ and $J$ are $R$-submodules of $K$ then the following holds:
\begin{itemize}
\item
 $I J \subseteq R$ then $J \subseteq I^{-1}$
\item
If $I \subseteq J$ then $J^{-1} \subseteq I^{-1}$.
\end{itemize}
\end{proposition}

\begin{proof}
These are straightforward. (It is efficient to prove the fourth claim before the third claim, since
the fourth claim implies the third.)
\end{proof}

\chapter{Fractional ideals} \label{ch3}

Let $R$ be an integral domain with fraction field $K$. 
The collection of nonzero ideals of $R$ forms a commutative monoid under the product operation.
A fundamental discovery in algebraic number theory and commutative algebra is that for a
certain class of widely used integral domains (namely Dedekind domains)
this monoid can be expanded to include some nonzero $R$-submodules of $K$ such
that the result is an Abelian group. However, we want to be careful about what
to add. For example, we do not want to add $I=K$ to the monoid since, as we have seen,~$I=K$ is not
invertible (except for the trivial situation where $R=K$). In the last section we introduced
the definition
$$I^{-1} \;\defeq\; \left\{ x \in K \mid x I \subseteq R \right\}$$
where $I^{-1}$ will be the inverse, if it exists, of a given $R$-submodule~$I$ of $K$.
If, however,~$I^{-1} = \{0\}$ then, of course, $I^{-1}$ cannot be the inverse of $I$. So
the existence of a nonzero $x \in K$ with $x I \subseteq R$ is a necessary condition
for invertibility. 

So in order to allow invertibility, we will restrict our attention to nonzero $I$ which satisfy this condition.
This motivates the definition of \emph{fractional ideal}:

\begin{definition}
Let $R$ be an integral domain with fraction field~$K$. A \emph{fractional ideal} of $R$
is a nonzero $R$-submodule $I$ of~$K$ such that $x I \subseteq R$ for some $x \in K^{\times}$.
In other words, $x I$ is an ideal of~$R$ for some nonzero $x \in K$.  In this context,
a regular ideal $I$ of~$R$ is sometimes called an \emph{integral ideal}.
\end{definition}

Here are a few useful equivalent characterizations of fractional ideals:

\begin{proposition}
Let $I$ be a nonzero $R$-submodule of~$K$. Then the following are equivalent:
\begin{itemize}
\item
$I^{-1}$ is not $\{0\}$.
\item
$I$ is a fractional ideal: there is a nonzero $x \in K$ such that $x I \subseteq R$.
\item
There is a nonzero $d \in R$ such that $d I \subseteq R$.

\item
There is a nonzero $d \in R$ such that $I \subseteq d^{-1} R$.

\item
There is a nonzero $x \in K$ such that $I \subseteq x R$.

\end{itemize}
\end{proposition}

\begin{remark}
Let $I$ be a fractional ideal. Then any nonzero $d\in R$ such that $I \subseteq d^{-1} R$
can be called a ``common denominator''. So a fractional ideal is just an $R$-submodule of $K$
with a common denominator.
\end{remark}

\begin{proposition}
Every nonzero ideal of an integral domain is a fractional ideal.
If $x \in K^{\times}$ then $xR$ is a fractional ideal.
Every nonzero $R$-submodule of a fractional ideal is a fractional ideal.
\end{proposition}

\begin{remark}
If $x \in K^{\times}$ then we call $x R$ a
\emph{principal fractional ideal}.
\end{remark}

\begin{exercise}
Let $R$ be a PID.
Show that every fractional ideal is principal.
\end{exercise}

\begin{proposition}
If $I_1$ and $I_2$ are fractional ideals then
so are the sum $I_1 + I_2$, the product $I_1 I_2$, and the intersection $I_1 \cap I_2$.
If $I$ is a fractional ideal, then $I^{-1}$ is a fractional ideal.
\end{proposition}

\begin{proof}
Let $a_1, a_2 \in R$ be nonzero elements such that $a_1 I_1$ and~$a_2 I_2$ are ideals of~$R$.
Then $a_1 a_2 (I_1 + I_2)$ is an ideal, as is $a_1 a_2 (I_1  I_2)$.
Note $I_1 \cap I_2$ is nonzero: multiply a nonzero element of $I_1 \cap R$ with a nonzero 
element of $I_2 \cap R$.
Since $I_1 \cap I_2$ is an~$R$-submodule of a fractional ideal, it must be a fractional ideal.

For the last claim, note that $I$ and $I^{-1}$ are both nonzero since $I$ is a fractional ideal.
If~$y \in I$ is nonzero, then $y I^{-1} \subseteq R$. (One can also argue that $y R \subseteq I$ 
so~$I^{-1} \subseteq y^{-1} R$).
\end{proof}

\begin{corollary}
The collection of fractional ideals of an integral domain~$R$ forms a commmutative
monoid under the product operation. This monoid contains, as  submonoids, (1) the monoid of nonzero
ideals, (2) the monoid of principal fractional ideals.
\end{corollary}

If $R$ is a DVR, then the monoid of fractional ideals is easily described.

\begin{proposition}\label{prop24}
Let $R$ be a DVR with uniformizer $\pi$. Then the fractional ideals of $R$
are all principal of the form $\pi^k R$ with $k\in \bZ$. These are distinct, and
form a group under multiplication.
The map $k \to \pi^k R$ is an isomorphism between the additive group $\bZ$
and the multiplicative group of fractional ideals of $R$.
\end{proposition}

\begin{exercise}
Let $R$ be a DVR. Show that the only nonzero $R$-submodule of $K$
that is not a fractional ideal is $K$ itself.
\end{exercise}

We will see that for Dedekind domains every fractional ideal is invertible, so, as in the
special case of a DVR, this monoid is actually a group.
However, for general
integral domains the question of invertibility is trickier. In fact, we have another necessary condition
for invertibility. 

\begin{proposition}\label{prop_inv_fg}
If $I$ is an invertible $R$-submodule of $K$ then $I$ is finitely generated.
\end{proposition}

\begin{proof}
If $I J = R$ then there are finite sequences $x_1, \ldots, x_k \in I$ and $y_1, \ldots, y_k \in J$ of elements 
such that
$$
1 = \sum x_i y_i.
$$
If $x \in I$ then
$$
x = \sum x_i (x y_i)
$$
which is in $x_1 R + \ldots + x_nR$.
Thus $I = x_1 R + \ldots + x_nR$.
\end{proof}

We now have two necessary conditions for invertibility of nonzero $R$-submodules~$I$ of~$K$:  (1) $x I \subseteq R$ for some $x \in K^\times$, and (2) $I$ is finitely generated.
However, the second clearly implies the first.

\begin{proposition}
Suppose $I$ is a finitely generated $R$-submodule of $K$. Then $I$ is a fractional ideal.
\end{proposition}

The proceeding two propositions  suggests that, for the purposes of invertibility, 
we focus on finitely generated fractional ideals.
However, if $R$ is not Noetherian, this excludes even some nonzero integral ideals.
So if we want all nonzero integral ideals to be invertible, we should focus on Noetherian
domains.\footnote{In a non-Noetherian
ring the best you can do is for all finitely generated fractional ideals to be 
invertible. Integral domains where every finitely generated
fractional ideal is invertible are called \emph{Pr\"ufer domains}.  See Appendix E for more information.}

The collection of finitely generated fractional ideals has closure properties. 
This is summarized by the next proposition (to see this, express
any finitely generated fractional ideal as $x_1 R_1 + \ldots + x_2 R_2$ and use
distributive laws in the case of $I_1 I_2$):

\begin{proposition}
Let $I_1$ and $I_2$ be finitely generated fractional ideals of~$R$. Then~$I_1 + I_2$
and~$I_1 I_2$ are also finitely generated fractional ideals. In particular, the collection of 
finitely generated fractional ideals forms
a commutative monoid under products (with identity $R$).
\end{proposition}

As mentioned above, if we want every nonzero integral ideal to have a chance of being
invertible, we should work in a Noetherian domain.
In this case all fractional ideals are automatically finitely generated:

\begin{proposition}
If $R$ is a Noetherian domain then all fractional ideals are finitely generated.
In fact, a nonzero $R$-submodule $I$ of the fraction field~$K$ is finitely generated
if and only if $I$ is a fractional ideal.
\end{proposition}

\begin{proof}
Recall that if $M$ is a finitely generated module over a Noetherian ring then 
all its submodules are finitely generated.
In the case of $R$-submodules $I$ of $K$, we have already established that (1)  if $I$ is a fractional ideal then
it is a submodule of a principal fractional ideal,
and (2)  that if $I$ is finitely generated and nonzero then it is a fractional ideal.
\end{proof}


\chapter{Invertibility criteria and results}\label{ch4}

In this section we will begin in earnest our study of invertible fractional ideals, and see how
invertibility is connected with the integrally closed condition. 
Recall that if~$I$ is a fractional ideal then $I^{-1}$ is defined
as the fractional ideal $\{x \in K \mid x I \subseteq R \}$, with the caveat
that~$I^{-1}$ might not be an actual inverse. In general we can only expect $I I^{-1} \subseteq R$.
But if~$I$ is invertible, then~$I^{-1}$ will
be the true inverse in the sense that $I I^{-1} = R$.
In this section we will be interested in determining when $I^{-1} $
is the true inverse of $I$. 

We have established an important case where we have invertibility: every principal
fractional ideal is invertible. For local integral domains we have the following
tidy result:

\begin{proposition}\label{propInv}
Let $R$ be a local integral domain.
Then a fractional ideal $I$ of $R$ is invertible if and only if $I$ is principal.
\end{proposition}

\begin{proof}
One direction is clear, so suppose $I$ is invertible: $I I^{-1} = R$.
Let $a \in R$ be nonzero such that $J_1 = a I$ and $J_2 = a I^{-1}$ are ideals.
Since $J_1 J_2 = a^2 R$, we can apply
Proposition~\ref{divisibility_principal_prop} 
to conclude that $J_1$ and $J_2$ are principal. Thus $I$ is principal.
\end{proof}

\begin{remark}
This can be used to give another proof of Theorem~\ref{theoremI}
since we have established that a local Noetherian domain with a principal nonzero maximal ideal
is a DVR.
\end{remark}

Now we will see how the integrally closed condition can give us necessary
conditions for invertibility. Recall that $R$ is integrally closed
if, for each monic $f\in R[X]$, every root of $f$ in $K$ is actually in~$R$.
To relate this to invertibility of fractional ideals,
we begin with what seems like an unrelated question:
Given a fractional ideal $I$ for $R$, can we find a larger subring of~$K$
such that $I$ is a module for that ring as well (with scalar multiplication coming from the multiplication of~$K$)?
This would require that for every $x$ in the larger ring,~$x I \subseteq I$ still holds.
This motivate the following definition:

\begin{definition}
Let $R$ be an integral domain and let $K$ be its fraction field.
If~$I$ is a fractional ideal, then
$$
\mathcal{R}(I) \; \defeq \; \{ x \in K \mid x I \subseteq I\}.
$$
\end{definition}

This set $\mathcal{R}(I)$ turns out to be the sort of ring we want:

\begin{proposition}
Let $R$ be an integral domain and let $K$ be its fraction field.
If $I$ is a fractional ideal of $R$, then the following hold:
\begin{itemize}
\item $\mathcal{R}(I)$ is a subring of~$K$ containg~$R$: so  $R \subseteq \mathcal{R}(I) \subseteq K$.
\item $I$ is a fractional ideal of~$\mathcal R(I)$, where scalar multiplication is induced
 by the product of~$K$.
In fact, $\mathcal R (I)$ is the maximum subring $R'$ of~$K$ (under inclusion) such that $I$ is a fractional
ideal of~$R'$.
\item $\mathcal{R}(I) \subseteq (I I^{-1})^{-1}$.
\item $\mathcal{R}(I)$ is a fractional ideal of~$R$.
\end{itemize}
\end{proposition}

\begin{proof}
Verifying these properties are mostly straightforward.
For example, to show~$\mathcal{R}(I) \subseteq (I I^{-1})^{-1}$, observe that
$x I \subseteq I$ implies $x I I^{-1} \subseteq I I^{-1} \subseteq R$.
Note that $(I I^{-1})^{-1}$ is a fractional ideal of~$R$, and recall
that any nonzero $R$-submodule of a fractional ideal is a fractional ideal.
\end{proof}

When $R$ is Noetherian we see a connection between $\mathcal R(I)$ and integral elements.

\begin{proposition}
Let $R$ be a Noetherian domain.
Then every element of $\mathcal{R}(I)$
is integral over~$R$.
\end{proposition}

\begin{proof}
Let~$x \in \mathcal{R}(I)$. Since $\mathcal{R}(I)$ is a ring,
the ring $R[x]$ is a subring of $\mathcal{R}(I)$.
Also observe that $R[x]$ is an $R$-submodule of~$\mathcal R(I)$, so $R[x]$ is also a fractional ideal.
In particular it is finitely generated as an $R$-module since $R$ is Noetherian. 
Fix a finite generating set of $R[x]$. Each generator is of the form $f(x)$
for some polynomial~$f \in R[X]$. Fix such a polynomial for each generator, and let
$d$ be the largest degree among these polynomials.
Observe that since $x^{d+1}$ can be written as an $R$-linear
combination of the generators, we can find a monic polynomial $g$ of~$R[X]$ of degree $d+1$
such that~$g(x) = 0$. Thus $x$ is integral over~$R$.
\end{proof}

\begin{corollary}\label{corF}
Let $R$ be a Noetherian domain that is integrally closed (in its field of fractions).
If $I$ is a fractional ideal of $R$ then  
$$\mathcal{R}(I) = R.$$
\end{corollary}

This leads to the following consequence in the case that $R$ is an integrally closed
Noetherian domain.

\begin{proposition}\label{subset_cancel_prop}
Let $R$ be an integrally closed Noetherian domain. If $I$ and $J$ are fractional ideals such
that $I J \subseteq I$ in~$R$, then $J$ is an ideal of~$R$.
\end{proposition}


We can use these ideas to come up with criteria for invertibility, at least in the case of maximal ideals.
The first criterion is useful for general integral domain~$R$.

\begin{proposition}\label{prop35}
Let $R$ be an integral domain and let $\frak m$ be a nonzero maximal ideal of~$R$.  
Then  $\frak m$ is invertible if and only if $\mathfrak m^{-1}$ is not contained in~$\mathcal{R}(\mathfrak m)$.
\end{proposition}

\begin{proof}
Since $R \subseteq \frak m^{-1}$, we have that $\frak m \subseteq \frak m \frak m^{-1} \subseteq R$.
Since $\mathfrak m$ is maximal, either~$\frak m \frak m^{-1} = \frak m$ or $\frak m \frak m^{-1} = R$.

We now prove the contrapositive version of the claim. 
Suppose $\mathfrak m$ is not invertible. Then $\frak m \frak m^{-1} = \frak m$.
Thus $\frak m^{-1} \subseteq \mathcal{R}(\frak m)$.
Conversely, if $\frak m^{-1} \subseteq \mathcal{R}(\frak m)$ then~$\frak m \frak m^{-1} \subseteq \frak m$,
so $\mathfrak m$ is not invertible.
\end{proof}

The next criterion is useful for integrally closed Noetherian domains. 
It follows from the previous proposition using the equality $\mathcal R(\mathfrak m) = R$.
(We do not need to exclude the trivial case where~$\mathfrak m = \{0\}$, since in 
this case~$\mathfrak m^{-1} = K = R$
and the result holds.)

\begin{proposition} 
Let $R$ be an integrally closed Noetherian domain
and let $\frak m$ be a maximal ideal of~$R$.
Then~$\mathfrak m$ is invertible if and only if $\mathfrak m^{-1}$ is 
not an integral ideal.
\end{proposition}

\begin{remark}
Since $R \subseteq \mathfrak m^{-1}$, we can rephrase the above criterion as
giving a condition for $\mathfrak m^{-1}$ to be~$R$.
This condition is $\mathfrak m^{-1} = R$ if and only if
$\mathfrak m$ is not invertible.
\end{remark}

In practice, we sometimes use the following criterion that follows immediately from the above.

\begin{corollary} \label{new_criterion_cor}
Let $R$ be an integrally closed Noetherian domain
and let $\mathfrak m$ be a maximal ideal of $R$.
If there is a fractional ideal $I$ of~$R$ that is not an integral ideal  
and if~$I \mathfrak m \subseteq R$,
then $\mathfrak m$ is invertible.
\end{corollary}

\begin{proof}
From $I \mathfrak m \subseteq R$ we have $I \subseteq \mathfrak m^{-1}$.
So $\mathfrak m^{-1}$ cannot be an integral ideal of~$R$ (it cannot be contained in~$R$).
Now use the above proposition.
\end{proof}

We will use this criterion to show that a Dedekind domain that is not a field must have at least one invertible maximal ideal.
But first we need a lemma.

\begin{lemma} \label{lemma38}
Let $R$ be a Noetherian domain with fraction field~$K$.
If $R$ is not a field then
there is a nonzero prime ideal $\mathfrak p$ of $R$ and an $x \in K^\times \smallsetminus R$
such that~$x \mathfrak p \subseteq R$.
\end{lemma}

\begin{proof}
Let $\mathcal S$ be the collection of all nonzero ideals $I$ for which there is an
$x \in K^\times \smallsetminus R$ such that $x I \subseteq R$.  Any nonzero proper principal
ideal is in $\mathcal S$,  so $\mathcal S$ is not empty. By the Noetherian property there
is a maximal element $\mathfrak p$ in $\mathcal S$.
Observe that $\mathfrak p$ is a proper ideal since the identity ideal $R$ is not in~$\mathcal S$.
Fix $x \in K^\times \smallsetminus R$ where~$x \mathfrak p \subseteq R$
 
Suppose $a b \in \mathfrak p$ but $b\not\in \mathfrak p$ where $a, b \in R$.
Since~$x \mathfrak p \subseteq R$, we 
have~$ax (\mathfrak p + b R) \subseteq R$. By maximality of $\mathfrak p$, we have $ax \in R$.
Thus $x (\mathfrak p + aR) \subseteq R$. By maximality of $\mathfrak p$ again we have
$\mathfrak p + aR = \mathfrak p$. Thus $a \in \mathfrak p$. Hence $\mathfrak p$ is prime.
\end{proof}

\begin{theorem} \label{thm39}
Let $R$ be a Dedekind domain that is not a field. Then $R$ has an invertible
prime ideal.
\end{theorem}

\begin{proof}
By Lemma~\ref{lemma38}, there is a nonzero prime ideal $\mathfrak p$ and
a fractional principal ideal $x R$ that is not an integral ideal
such that $(x R) \mathfrak p \subseteq R$. 
Since $R$ is a Dedekind domain, the prime ideal $\mathfrak p$ is maximal.
By Corollary~\ref{new_criterion_cor},
$\mathfrak p$ is invertible.
\end{proof}

As promised, we now prove Theorem~\ref{thmC} without 
using the unique factorization theorem for ideals. We need to reprove the following:

\begin{theorem}\label{thmC2}
A Dedekind domain $R$ with a unique nonzero prime ideal is a DVR.
\end{theorem}

\begin{proof}
Let $\mathfrak p$ be the nonzero prime ideal of~$R$.
By Theorem~\ref{thm39}, $\mathfrak p$ is invertible. 
Thus by Theorem~\ref{theoremI} (or Proposition~\ref{propInv} plus Theorem~\ref{thmB})
$R$ is a DVR.
\end{proof}

\begin{exercise}
Let $R$ be a Dedekind domain with a unique prime ideal~$\mathfrak p$. 
Let $x$ be as in Lemma~\ref{lemma38}. Show that $x^{-1} \in R$ and that $\mathfrak p = x^{-1} R$.
Use this to give another proof of Theorem~\ref{thmC2} by using Theorem~\ref{thmB}
to conclude that $R$ is a DVR.

Hint: if $x \mathfrak p \subseteq \mathfrak p$ then $x \in \mathcal R(\mathfrak p)$, 
which cannot happen. So what is $x \mathfrak p$?
\end{exercise}

The next two exercises concern the rings $\mathcal R(I)$.

\begin{exercise}\label{ex8_new}
Show that if $\mathcal R(I) = R$ for all fractional ideals of an integral domain~$R$ then $R$
is integrally closed. 

Hint: suppose $x\in K\smallsetminus R$ is integral over~$R$.
Show that~$I = R[x]$ is an~$R$-submodule of~$K$. Show that $I$ is a finitely generated $R$-module,
hence is a fractional ideal. Show that
since $I$ is a ring we have $I^2 \subseteq I$. Conclude that 
$$R \subsetneq I \subseteq \mathcal R(I).$$

Note: an integral domain $R$ such that $\mathcal R(I) = R$
for all fractional ideals $I$ is said to be \emph{completely integrally closed}.
This exercise shows that a completely integrally closed domain is indeed
integrally closed. Corollary~\ref{corF} shows that for Noetherian domains
integrally closed implies completely integrally closed.
\end{exercise}

\begin{exercise}\label{ex12}
The \emph{normalization} of an integral domain $R$ is defined to be the set of all elements of
its fraction field $K$ that
are integral over~$R$. Suppose $R$ is a Noetherian domain. Show that $x\in K$ is in the normalization
of~$R$ if and only if~$x\in \mathcal R (I)$ for some fractional ideal $I$ (Hint: see previous exercise).
Conclude that the normalization is the union of the rings~$\mathcal R (I)$.

If $I, J$ are fractional ideals, 
show that $\mathcal R(I)$ is contained in $\mathcal R(IJ)$. 
Show then that if $x, y\in K$ are in the normalization of $R$ then $x, y \in \mathcal R(I)$
for some fractional ideal $I$.
Conclude that the normalization
is a subring of~$K$.
\end{exercise}

The next three exercises concern cancellation in special cases.
Inverses are very handy for cancellation, but unfortunately 
we cannot hope to have invertibility for a general fractional ideal except in Dedekind domains.
There are situations where we, nevertheless, have cancellation even for non-invertible fractional ideals.
The next two exercises illustrate some special cases. (See Appendix E for other situations where we have cancellation.)

\begin{exercise} \label{exA}
Let $R$ be a local integral domain, and suppose
$$
I J = J = R J
$$
where $I$ is a nonzero ideal and where $J$ is a finitely generated fractional ideal.
Show that we can cancel $J$ to get~$I = R$.

Hint: otherwise, note that $\frak m J = J$ where $\frak m$ is the maximal ideal.
Take a minimal generating set of $J$ as an $R$-module, and get a contradiction by making it smaller.
(Recall that $1 - a$ is a unit if $a \in \mathfrak m$).
(This is related to Nakayama's lemma in commutative algebra. See Exercise~\ref{ex8}.)
\end{exercise}

\begin{exercise}\label{exx} 
Extend the above to any integral domain $R$. In other words, suppose
$$
I J = J
$$
where $I$ is a nonzero ideal and where and $J$ is a finitely generated fractional ideal. 
Show that $I = R$.

Hint (using localization, see Section~\ref{ch5} below):
Suppose otherwise that $I$ is contained in a maximal ideal~$\frak m$,
so that $\frak m J = J$. Now localize, and use the previous
exercise to derive a contradiction.

Hint (using linear algebra over the fraction field of~$R$): Set up a system of equations, and 
identify a singular matrix. From the resulting determinant,
show that~$1 \in I$. 
\end{exercise}

\begin{exercise}\label{ex15}
Use Proposition~\ref{subset_cancel_prop} and the previous exercise
to prove the following cancellation law when $R$ is an integrally closed Noetherian domain.
If $I$ and $J$ are fractional ideals such that $I J = J$, then $I = R$. 

Conclude further that
if $I_1 J = I_2 J$ where $I_1, I_2,$ and $J$ are fractional ideals one of which is invertible,
then $I_1 = I_2$.
\end{exercise}

The following  four exercises build on each other
 to culminate in another condition 
that characterizes discrete valuation rings.

\begin{exercise} \label{ex8}
Let $R$ be a local commutative ring, let $I$ be a ideal of~$R$,
and let $M$ be a finitely generated~$R$-module. 
Generalize Exercise~\ref{exA} and prove the following
version of Nakayama's lemma: if $I M = M$ then $M=0$ or $I = R$.
(Here $I M$ is defined as the submodule of $M$ given by finite sums
of elements of the form $a m$ where $a\in I$ and $m \in M$).
\end{exercise}

\begin{exercise} \label{nakayama_ex}
Let $R$ be a local commutative ring and let $I$ be a proper ideal of~$R$.
Let~$M$ be an $R$-module and let $N$ be a submodule of~$M$.
Assume that either (1)~$M$ is finitely generated as an $R$-module,
or at least that (2) the quotient $M/N$ is finitely generated as an~$R$-module.
Use the previous exercise to prove the following
version of Nakayama's lemma:
If $M = N + I M$ then~$M = N$.
(Hint: consider the quotient module $M/N$).
\end{exercise}

\begin{exercise} \label{ex18}
Let $R$ be a local commutative ring with maximal ideal~$\mathfrak m$. 
Let $k$ be the field $R/\mathfrak m$, called the \emph{residue field}.

(1) Show that the scalar multiplication law
$$
R/\mathfrak m \times \mathfrak m/\mathfrak m^2\to \mathfrak m/\mathfrak m^2,
\qquad [a]\cdot [b] \defeq [ab]\quad\text{with $a\in R$ and $b \in \mathfrak m$}
$$ 
is well-defined and
makes the Abelian group $\mathfrak m/\mathfrak m^2$ into a $k$-vector space.

(2) Suppose $\mathfrak m$ is finitely generated.
Use the previous exercise to show that if~$[a_1], \ldots, [a_n] \in \mathfrak m/\mathfrak m^2$
spans the $k$-vector space $\mathfrak m/\mathfrak m^2$ where $a_1, \ldots, a_n \in \mathfrak m$,
then~$a_1, \ldots, a_n$ generate the ideal~$\mathfrak m$.
Hint: use the previous exercise with
$$M = I = \mathfrak m, \qquad N = a_1 R + \ldots + a_n R.$$
\end{exercise}

\begin{exercise} \label{dimension_ex}
Let $R$ be a local Noetherian domain with maximal ideal~$\mathfrak m$, 
and residue field  $k = R/\mathfrak m$.
Show that $R$ is a DVR if and only if 
the $k$-vector space~$\mathfrak m/\mathfrak m^2$ has dimension 1.
\end{exercise}

\begin{exercise}\label{exxx}
Define an \emph{irreducible} ideal in an integral domain to be a nonzero proper ideal that is not  equal to $I J$
for nonzero proper ideals $I$ and $J$.  
Show that in a Noetherian domain every nonzero proper ideal of $R$ factors as the product
of irreducible ideals. 

Hint: use ascending chains and use Exercise~\ref{exx}
to show that if $I = J_1 J_2$ for nonzero proper ideals $J_1, J_2$ then $I \subsetneq J_1$
and $I \subsetneq J_2$.
\end{exercise}


\chapter{Localizing fractional ideals}\label{ch5}

In this section, and the remaining sections, we assume the reader is familiar with localization, at
least in the context of integral domains.\footnote{See,
for example, my
expository essay on localization in integral domains.}
In this section we review this theory and expand the theory to 
include fractional ideals.
(Although much of this section is likely review, the reader should verify for themselves
any unfamiliar result.)

Localization is a process of forming a new ring $S^{-1} R$ from a given
commutative ring and multiplicative system~$S$.
Although this can be done for any commutative ring, the prototypical setting and the most
accessible situation is when we localize with
integral domains.
In this document, when we localize we will always assume $R$ is an integral domain, and that
$S$ is a subset of $R$ closed under multiplication that contains $1$ but does not
contain~$0$. In other words, $S$ is a multiplicative submonoid of $R \smallsetminus \{ 0\}$.
The nice thing about this situation is that the ring
$S^{-1} R$ can be identified with the subring of the fraction field $K$ of~$R$ consisting
of elements of the form $r/s$ where $r\in R$ and $s\in S$.
An important example is where $S = R \smallsetminus \mathfrak p$ where~$\mathfrak p$ is a prime
ideal of~$R$. In this case $S^{-1} R$ is written $R_{\mathfrak p}$.
In this case $R_{\mathfrak p}$ is a local ring.

We can localize modules as well. Given an $R$-module $M$, localization produces
an $S^{-1} R$-module called $S^{-1} M$.
 We will limit ourselves to the nice case where $I$ is an~$R$-submodule
of~$K$. 
If $I$ is an $R$-submodule of $K$, then $S^{-1} I$ can be identified with the set elements of the form $x / s$
with~$x\in I$ and~$s \in S$. Here $x/s$ is just notation for~$x s^{-1}$. 
The nice thing about this case is that $S^{-1} I$ is again a subset of~$K$.
Note that the localization of ideals is a special case of this type of localization.
We assume the reader is familiar with localizing such modules (if not, it is a reasonable
exercise to check the details).
For example, we take it as established
from earlier work
(or leave it to the reader to check) that $S^{-1} I$ is an $S^{-1} R$ submodule of $K$.
Note that~$y\in S^{-1} I$ if and only if it is of the form $x (r/s)$ where~$x\in I, r\in R, s\in S$.
So we sometimes write~$I  (S^{-1} R)$ for $S^{-1} I$. This is especially
common when~$S^{-1} R$ is $R_{\frak p}$ for some prime ideal~$\mathfrak p$.
In this case we often write~$I R_{\frak p}$ for $S^{-1} I$.

\begin{proposition}
If $I$ is a fractional ideal of~$R$, then $S^{-1} I$ is a fractional ideal of~$S^{-1} R$.
\end{proposition}

\begin{proof}
We have $x I \subseteq R$ for some $x \in K^{\times}$. Observe $x (S^{-1} I) \subseteq S^{-1} R$.
\end{proof}

We take the next two propositions as established from previous work (or we leave
the proofs to the reader):

\begin{proposition}
Let $I_1, I_2$ be $R$-submodule of $K$. 
Then, as $S^{-1} R$-modules,
$$S^{-1} (I_1 + I_2) = (S^{-1} I_1) +  (S^{-1} I_2), $$
$$S^{-1} (I_1 I_2) = (S^{-1} I_1) (S^{-1} I_2),$$
and
$$S^{-1} (I_1 \cap I_2) = (S^{-1} I_1) \cap  (S^{-1} I_2).$$
\end{proposition}

The correspondence is also well-behaved with respect to principle ideals:

\begin{proposition}
If $x \in K$ then
$$
S^{-1} ( xR) = x (S^{-1} R).
$$
More generally, if $U$ is a set of elements of~$K$, and if $I$ is the $R$-submodule generated by~$U$
then $S^{-1} I$ is the $R$-submodule generated by $U$ in $S^{-1}R$.
\end{proposition}

\begin{remark}
The module generated by a $x_1, \ldots, x_k \in K$
is just $x_1 R + \ldots + x_k R$. We have also mentioned the $R$-submodule generated by an infinite set $U$,
and the above proposition does apply to this case.
However, we do not need the case of infinite $U$ in this document.
\end{remark}

\begin{corollary}\label{cor41_new}
Suppose $I$ is a principal fractional ideal of $R$, then $S^{-1} I$ is a principal
fractional ideal of~$S^{-1} R$. 
Suppose $I$ is a finitely generated fractional ideal of~$R$, then $S^{-1} I$ is a finitely
generated fractional ideal of~$S^{-1} R$. 
\end{corollary}

We can also derive results about inverses. Warning: we assume $I$ is finitely generated here.

\begin{proposition}\label{prop49}
Suppose $I$ is a finitely generated fractional ideal of~$R$. Consider~$I^{-1}$ as a fractional ideal of~$R$
and consider $(S^{-1} I)^{-1}$ as a fractional ideal of~$S^{-1}R$. Then
$$
(S^{-1} I)^{-1} = S^{-1} I^{-1}.
$$
\end{proposition}

\begin{proof}
The inclusion $S^{-1} I^{-1} \subseteq (S^{-1} I)^{-1}$ is straightforward.
For the other inclusion, let $x_1, \ldots, x_k$ be generators of $I$ as an $R$-module.
Suppose $x \in (S^{-1} I)^{-1}$. Then $x x_i = r_i / s_i$ for some $r_i \in R$ and $s_i \in S$.
So $(s_1 \cdots s_k) x \in I^{-1}$.
\end{proof}

We take is as established from earlier work that every ideal of $S^{-1} R$ is of
the form $S^{-1} I$ for some ideal $I$ of~$R$.\footnote{In fact, given an ideal $J$ of $S^{-1} R$,
the ideal $I=(J \cap R)$ will work, and will be the maximum such~$I$.}
We can extend this to fractional ideals.

\begin{proposition}
If $J$ is a fractional ideal of~$S^{-1} R$, then there is a fractional ideal~$I$ of $R$ such that
$J = S^{-1} I$.
\end{proposition}

\begin{proof}
Recall that $x J$ is an ideal of $S^{-1}R$ for some $x \in K^\times$.
Let $I'$ be an ideal of~$R$ with $S^{-1} I' = x J$.
Now consider $I = x^{-1} I'$.
\end{proof}

\begin{corollary} \label{noeth_cor}
Suppose $R$ is an integral domain with multiplicative system~$S$. If~$R$ is a PID then
so is~$S^{-1} R$. If $R$ is Noetherian, then so is $S^{-1}R$.
\end{corollary}

\begin{proof}
This follows from the above proposition and Corollary~\ref{cor41_new}.
\end{proof}

\begin{corollary}
Suppose $R$ is an integral domain with multiplicative system~$S$. If every 
fractional ideal of $R$ is invertible, then every fractional ideal of $S^{-1}R$
is invertible.
\end{corollary}

In a similar vein, we have the following:

\begin{proposition}\label{int_closed_prop}
Let $R$ be an integral domain with field of fractions~$K$.
Let~$S$ be a multiplicative system of~$R$. If~$R$
is integrally closed in $K$, then $S^{-1} R$ is integrally closed in~$K$.
\end{proposition}

\begin{remark}
The usual proof involves manipulating polynomials, and I assume the reader
has seen it. (If not, it is a good exercise; See for example,
Exercise~\ref{ex38}).
We sketch another
argument that highlights the techniques given in this document. It requires the
extra assumption that $R$ is Noetherian, which holds in the situations we are most interested in.

Start by showing, 
for any finitely generated fractional ideal $I$,
the identity (where the right-hand side is in $S^{-1} R$):
$$S^{-1} \mathcal R(I) = \mathcal R(S^{-1} I).$$
We have $\mathcal R(I) = R$ (Corollary~\ref{corF}), so~$S^{-1} R= \mathcal R(S^{-1} I)$.
This holds for all fractional ideals $I$ (assuming $R$ is Noetherian), and all fractional ideals of
$S^{-1} R$ are of the form $S^{-1}I$. By Exercise~\ref{ex8_new}, 
the integral domain $S^{-1} R$ is integrally closed.
\end{remark}

We can express some of the above results in terms of homomorphisms of monoids.

\begin{proposition}
Let $R$ be an integral domain and
let $S$ be a multiplicative system.
The map 
$$I \mapsto S^{-1} I$$ is a monoid homomorphism from 
the multiplicative monoid of nonzero $R$-submodules of $K$
to the multiplicative monoid of nonzero $S^{-1} R$-submodules of $K$.

This map restricts to a surjective monoid homomorphism from 
the multiplicative monoid of fractional ideals of~$R$ to 
the multiplicative monoid of fractional ideals of~$S^{-1} R$.
If the domain of this map is a group, then so is the image,
and the map is a group homomorphism.

This map further restricts to a surjective monoid homomorphism from 
the multiplicative monoid of nonzero integral ideals of~$R$ to 
the multiplicative monoid of nonzero integral ideals of~$S^{-1} R$.
\end{proposition}

We can say something about the kernel:

\begin{proposition}
Let $R$ be an integral domain with fraction field $K$
and let $S$ be a multiplicative system.
Suppose $I$ is an $R$-submodule of~$K$ that maps to the identity under
the above described homomorphism: in other words, suppose
$S^{-1} I = S^{-1} R$. Then $I$ must intersect~$S$.
Conversely, if $I$ is an integral ideal that intersects~$S$, then
it maps to the identity:~$S^{-1} I = S^{-1} R$.
\end{proposition}

\begin{proof}
If $S^{-1} I = S^{-1} R$ then $1 = x/s$ for some $x \in I$ and $s \in S$.
\end{proof}

Finally, we remind the reader of the correspondence of prime ideals. We take this
as established from earlier work:

\begin{proposition}\label{prime_correspondence_prop}
Let $R$ be an integral domain, and let $S$ be a multiplicative system of~$R$.
The rule $\frak p \mapsto S^{-1} \frak p$ defines an inclusion preserving bijection 
$$
\left\{ \text{Prime ideals of $R$ disjoint from~$S$} \right\} \; \to \; \left\{ \text{Prime ideals of $S^{-1} R$} \right\}.
$$
The inverse sends a prime ideal $\frak p$ of $S^{-1} R$ to $\frak p \cap R$, and is also inclusion preserving.
\end{proposition}

This correspondence, together with earlier results, leads to the following theorem:

\begin{theorem}\label{localize_dd_thm}
Let $R$ be an integral domain, and let $S$ be a multiplicative system of~$R$.
If $R$ is a Dedekind domain, then so is $S^{-1} R$.
\end{theorem}

\begin{proof}
We know that $S^{-1} R$ must be an integrally closed Noetherian domain
by Proposition~\ref{int_closed_prop} (and the remark following it) and Corollary~\ref{noeth_cor}.
So we just need to show that every nonzero prime ideal $\mathfrak q$ of $S^{-1} R$
is maximal. 

Suppose that $\mathfrak q$ is a nonzero prime ideal of $S^{-1} R$, and let $\mathfrak m$ be a maximal
ideal of $S^{-1} R$ containing it.
By the above proposition, there are prime ideal $\mathfrak p_1, \mathfrak p_2$ of~$R$
such that $\mathfrak q = S^{-1} \mathfrak p_1$ and $\mathfrak m = S^{-1} \mathfrak p_2$.
Since the correspondence is inclusion perserving in both directions, 
$\mathfrak p_1\subseteq \mathfrak p_2$.
Obviously $\mathfrak p_1, \mathfrak p_2$ are not zero, so they must be equal since every
nonzero prime ideal of $R$ is maximal. Thus their images $\mathfrak q, \mathfrak m$
are equal, and so $\mathfrak q$ is maximal.
\end{proof}


\chapter{Some local to global results}\label{ch7}

Now we investigate the relationship between properties of an integral domain $R$
and the corresponding properties of the localizations $R_{\mathfrak m}$.
This will allow us to prove results about various types of integral domains, including
Dedekind domains, in a unified and elegant manner
by reducing to the easier local situation.

\begin{proposition}\label{propJ}
Let $R$ be an integral domain with field of fractions~$K$.
If $I$ is an~$R$-submodule of~$K$ then
$$
I = \bigcap_{\frak m\in \mathcal M} I R_{\frak m}
$$
where $\mathcal M$ is the set of maximal ideals of~$R$.
\end{proposition}

\begin{proof}
One direction is straightforward.

Suppose $x \in \bigcap I R_{\frak m}$. Let $J_x$ be defined as follows:
$$
J_x \;\defeq\; \left\{ y \in R \mid y x \in I\right\}.
$$
Observe that $J_x$ is an ideal of~$R$.
If $J_x \ne R$ then let $\frak m$ be a maximal
ideal containing~$J_x$. However, $x \in I R_{\frak m}$ so is of the form $a/s$ with $a \in I$ and $s\not\in \frak m$.
Thus~$s \in J_x$, a contradiction. So $J_x$ contains $1$.
\end{proof}

\begin{corollary}\label{intersection_cor}
If $R$ is an integral domain then
$$
R = \bigcap_{\frak m\in \mathcal M} R_{\frak m}
$$
where $\mathcal M$ is the set of maximal ideals of~$R$.
\end{corollary}

The above proposition yields a very useful test for inclusion and equality:

\begin{corollary}\label{equality_test_cor}
Let $R$ be an integral domain and let $I$ and $J$ be $R$-submodules of the fraction field of~$R$.
Then $I \subseteq J$ if and only if $I R_{\mathfrak m}  \subseteq J R_{\mathfrak m}$
for all maximal ideals~$\mathfrak m$. 
Similarly, $I = J$ if and only if $I R_{\mathfrak m}  = J R_{\mathfrak m}$
for all maximal ideals~$\mathfrak m$. 
\end{corollary}

Now we illustrate the power of the local approach by investigating various properties of integral
domains and their fractional ideals:

\begin{proposition}\label{prop56}
Suppose that $R$ is an integral domain with fraction field $K$.
Then~$R$ is integrally closed in $K$ if and only if $R_{\mathfrak m}$ is integrally closed
in $K$ for all maximal ideals~$\mathfrak m$.
\end{proposition}

\begin{proof}
Proposition~\ref{int_closed_prop} yields one direction. Suppose that 
$R_{\mathfrak m}$ is integrally closed for each maximal ${\mathfrak m}$.
Let $f\in R[X]$ be a monic polynomial, and let $x \in K$ be a root.
In particular, $x \in R_{\mathfrak m}$ for each $\mathfrak m$ since $R_{\mathfrak m}$ is integrally closed.
So $x \in \cap R_{\mathfrak m}$. Thus~$x \in R$ by Corollary~\ref{intersection_cor}.
\end{proof}

\begin{proposition}\label{prop57}
Let $R$ be an integral domain.
Then $R$ has the property that every nonzero prime ideal is maximal if and only
if $R_{\mathfrak m}$ has that property for 
all maximal ideals~$\mathfrak m$.
\end{proposition}

\begin{proof}
Suppose $R$ has the property in question. By the correspondence
of primes (Proposition~\ref{prime_correspondence_prop}), $R_\mathfrak m$ must
have the property in question as well.
Conversely, suppose $R_\mathfrak m$ has the property in question for all maximal ideals~$\mathfrak m$.
Let $\mathfrak p$ be any nonzero prime ideal of $R$ and let $\mathfrak m$ be a maximal ideal containing
$\mathfrak p$.  By assumption and the prime correspondence 
$\mathfrak p R_\mathfrak m = \mathfrak m R_\mathfrak m$, so
by the prime correspondence (Proposition~\ref{prime_correspondence_prop}) $\mathfrak p = \mathfrak m$.
\end{proof}

Now we are ready to prove one of the most important characterizations of Dedekind domains:

\begin{theorem}\label{thm58}
Let $R$ be a Noetherian domain. Then $R$ is 
a Dedekind domain if and only if $R_{\frak m}$  is a 
discrete valuation ring for all nonzero maximal ideals $\frak m$ of~$R$.
\end{theorem}

\begin{proof}
If the zero ideal is maximal, then $R$ is a field, which is considered a Dedekind domain.
So the claim is trivially true in this case. We now assume that the zero ideal is not maximal.

Suppose $R$ is a Dedekind domain. Then $R_{\mathfrak m}$
is a Dedekind domain for each nonzero maximal ideal $\mathfrak m$
by Theorem~\ref{localize_dd_thm}.
 Observe that each such~$R_{\mathfrak m}$ has a unique nonzero prime ideal, namely $\mathfrak m R_{\mathfrak m}$.
So each such $\mathfrak m R_{\mathfrak m}$ is a~DVR by Theorem~\ref{thmC}.
 
Suppose that $R_{\mathfrak m}$ is a DVR for each nonzero maximal ideal~$\mathfrak m$.
Then $R$ is Noetherian by assumption, every nonzero prime ideal of $R$ is maximal 
by~Proposition~\ref{prop57}, and $R$ is integrally closed by Proposition~\ref{prop56}
\end{proof}

\begin{remark}
An \emph{almost Dedekind domain} is defined to be an integral domain $R$ such
that~$R_{\frak m}$ is a DVR for all nonzero maximal ideals~$\mathfrak m$.
The above theorem implies that every Noetherian almost Dedekind
domain is a true Dedekind domain. However, there exists non-Noetherian
almost Dedekind domains.  See Appendix~E for more information.
\end{remark}

Every PID is a Dedekind domain, a fact we have taken to be established background knowledge. 
In a sense, we do not need to take it as established since it follows as an easy corollary to the above theorem.

\begin{corollary}
Every PID is a Dedekind domain.
\end{corollary}

\begin{proof}
Let $R$ be a PID and let $\mathfrak m$ be a nonzero maximal ideal of $R$.
Then $R_{\mathfrak m}$ is a~PID by Corollary~\ref{noeth_cor}.
In fact, $R_{\mathfrak m}$ is a DVR by Theorem~\ref{thmB}.
So $R$ is a Dedekind domain by the above theorem.
\end{proof}

\begin{theorem}
Let $I$ be a finitely generated fractional ideal of an integral domain~$R$.
Then $I$ is invertible as a fractional ideal of $R$ if and only if
$I R_{\frak m}$ is invertible as a fractional ideal of $R_{\frak m}$ for all maximal ideals $\frak m$ of~$R$.
\end{theorem}

\begin{proof}
One direction is straightforward.
For the other direction, 
apply Corollary~\ref{equality_test_cor} to show $I I^{-1} = R$.
\end{proof}

\begin{corollary}\label{cor41}
If $R$ is a Dedekind domain, then every fractional ideal is invertible, and so the monoid
of fractional ideals under products forms an Abelian group. This group is generated by
the nonzero ideals of $R$.
\end{corollary}

\begin{proof}
Use Theorem~\ref{thm58} and the fact that every fractional ideal of a DVR is 
invertible. 
Finally, note that every fractional ideal $J$ in~$R$ has the property that $a J$ is an ideal for some nonzero $a \in R$.
Thus $(aR) J = I$ for some nonzero ideal~$I$. Hence~$J = I (aR)^{-1}$.
\end{proof}

\begin{remark}
Later we will see that this group is generated by the nonzero prime ideals of~$R$.
\end{remark}

\begin{exercise}\label{fg_invert_ex}
Show that every finitely generated ideal in an almost Dedekind domain is invertible. 
(See the remarks after Theorem~\ref{thm58} for the definition of almost Dedekind domain.)
\end{exercise}

\begin{corollary}
Let $I$ be a finitely generated fractional ideal of an integral domain~$R$.
Then $I$ is invertible if and only if
$I R_{\frak m}$ is a principal fractional ideal of $R_{\frak m}$ for all maximal ideals $\frak m$ of~$R$.
\end{corollary}

\begin{proof}
Use Proposition~\ref{propInv}.
\end{proof}

We can strengthen Corollary~\ref{cor41}, giving us one of the major 
results about Dedekind domains. (A result of Emmy Noether).

\begin{theorem}\label{thm64}
Let $R$ be an integral domain. Then $R$ is a Dedekind domain
if and only if every nonzero ideal of~$R$ is invertible.
\end{theorem}

\begin{proof}
Corollary~\ref{cor41} gives one direction, so suppose that every 
nonzero ideal of $R$ is invertible. By Proposition~\ref{prop_inv_fg},
every nonzero ideal of~$R$ is finitely generated, so~$R$ must
be a Noetherian domain.
By Theorem~\ref{thm58} it is now enough to show
$R_{\frak m}$  is a~DVR for any nonzero maximal ideals $\frak m$ of~$R$.
Let $\mathfrak m$ be a nonzero maximal ideal of~$R$.
Since $\mathfrak m$ is invertible in~$R$, the 
maximal ideal $\mathfrak m R_{\frak m}$ is invertible in $R_{\frak m}$.
By Theorem~\ref{theoremI}, this implies that  $R_{\frak m}$  is a~DVR.
\end{proof}

\begin{exercise}
Let $R$ be an integral domain and let $u$ be in fraction field of~$R$. 
Show that $u$ is a unit in~$R$ if and only if it is a unit in $R_{\frak m}$
for each maximal ideal~$\mathfrak m$ of~$R$.
\end{exercise}


\chapter{Discrete valuations of Dedekind domains}\label{dedekind_vals_ch}

Suppose $R$ is a Dedekind domain with fraction field~$K$ and let $\mathfrak p$ be a nonzero prime ideal of~$R$.
Then $R_{\mathfrak p}$ is a DVR. Let $v_{\mathfrak p}\colon K^\times \to \bZ$
be the discrete valuation associated with~$R_{\mathfrak p}$. We consider this valuation for elements of~$R$:

\begin{proposition}
Suppose $R$ is a Dedekind domain with discrete valuation $v_{\mathfrak p}$ where~$\mathfrak p$
is a nonzero prime ideal. (As usual, let $v_{\mathfrak p}(0) = \infty$).
If $a \in R$ then $(1)$ $v_{\mathfrak p}(a) \ge 0$, and
$(2)$ $v_{\mathfrak p}(a) > 0$
if and only if~$a \in \mathfrak p$.
\end{proposition}

\begin{proof}
Observe that $R \subseteq R_{\mathfrak p}$, and that $(\mathfrak p R_{\mathfrak p})\cap R = \mathfrak p$.
\end{proof}

\begin{corollary}
Suppose $R$ is a Dedekind domain with fraction field~$K$. 
Then the map $\mathfrak p \mapsto v_{\mathfrak p}$
is an injective map from the set of nonzero prime ideals of $R$ to the set of discrete valuations of~$K$.
\end{corollary}

Our next goal is to strengthen this corollary by identifying the image of the 
map~$\mathfrak p \mapsto v_{\mathfrak p}$. In other words,
we want to identify which valuations are of the form~$v_{\mathfrak p}$.
By the above proposition, only valuations that are nonnegative on~$R$ can
be of the form~$v_{\mathfrak p}$, so we have at least this restriction on the image.

\begin{lemma}
Let $R$ be an integral domain with fraction field~$K$.
If
$v \colon K^\times \to \bZ$ is a valuation with $v(r) \ge 0$ for all $r \in R$
then
$$
\mathfrak p = \left\{ a \in R \mid v(a) > 0 \right\}
$$
is a nonzero prime ideal of $R$.
\end{lemma}

\begin{proof}
The properties of $v$ imply that $\mathfrak p$ is a prime ideal of~$R$.
If $\mathfrak p$ is the zero ideal, then $v$ would be identically zero on $K^\times$ 
since $K$ is the field of fractions of~$R$. This contradicts the surjectivity of 
$v \colon K^\times \to \bZ$.
\end{proof}

\begin{lemma}
Suppose $R$ is a Dedekind domain with fraction field~$K$. 
Suppose that~$v \colon K^\times \to \bZ$ 
is a valuation with~$v(r) \ge 0$ for all $r \in R$.
Let~$\mathfrak p$ be the associated nonzero prime ideal~$\left\{ a \in R \mid v(a) > 0 \right\}$.
Then 
$$
v = v_{\mathfrak p}.
$$
\end{lemma}

\begin{proof}
Observe that $v(r/s) \ge 0$ for all $r/s \in R_{\mathfrak p}$ where $r\in R$ and $s\in R\smallsetminus\mathfrak p$.
In particular, $R_{\mathfrak p} \subseteq \mathcal O_v$.
By Proposition~\ref{prop15} we have $R_{\mathfrak p} = \mathcal O_v$
since~$R_{\mathfrak p}$ is a DVR.
Now use Exercise~\ref{ex3}.
\end{proof}

This lemma allows us to conclude the following:

\begin{theorem}
Suppose $R$ is a Dedekind domain with fraction field~$K$. Then the map $\mathfrak p \mapsto v_{\mathfrak p}$
is a bijection from the set of nonzero prime ideals of $R$ to the set of discrete valuations of~$K$
that are nonnegative on~$R$.
\end{theorem}

For any prime $p$ in $\bZ$, write $v_p$ for $v_{\mathfrak p}$ with $\mathfrak p = p \bZ$.

\begin{corollary}
The map $p \mapsto v_p$
is a bijection from the set of (positive) primes of~$\bZ$ to the set of discrete valuations of~$\bQ$.
\end{corollary}

\begin{proof}
Observe that every discrete valuation must be nonnegative on~$\bZ$.
\end{proof}

One of the main theorems for Dedekind domains, perhaps the main theorem,
is that ideals factor uniquely into prime ideals. I felt free to assumed this result as background
for the proof of Theorem~\ref{thmC} in Section~\ref{ch2}, since this document is in some sense a part two in the theory of 
Dedekind domains. 
We later saw an alternate proof of Theorem~\ref{thmC} (see Theorem~\ref{thmC2}) which
did not use the unique factorization theorem. 
So
from a logical perspective the results
we have proved up to now do not depend on this unique factorization theorem. Now
that we know that the fractional ideals of a Dedekind domains form a group,
we can give a brief proof of at least the existence of a prime factorization, with uniqueness
coming later.
This will be important in showing that, for example, $v_{\mathfrak p}(a) = 0$ for all but
a finite number of valuations of a Dedekind domain. 

\begin{theorem}\label{thm75}
Let $R$ be a Dedekind domain. 
Every nonzero ideal can be written as the product of nonzero prime ideals.
 (We adopt the convention that the empty product is the identity ideal~$R$). 
\end{theorem}

\begin{proof}
The following proof is based on the fact that the set of fractional ideals of~$R$ forms a groups
under the product operation. So we freely use properties of groups. 

Suppose otherwise that there is a nonzero 
ideal that is not such a product. By the Noetherian property there is a maximal such ideal $I$.
Note that $I$ must be a proper ideal.
Let $\mathfrak p$ be a maximal ideal containing~$I$. Then $I \ne I \frak p^{-1}$ since $\mathfrak p \ne R$, and
$$
I \subsetneq I \frak p^{-1} \subseteq \frak p \frak p^{-1} = R.
$$
So
$$
I \frak p^{-1} = \frak p_1 \cdots \frak p_k
$$
for some $k \ge 0$ and nonzero prime ideals $\frak p_i$. Thus
$$
I  = \frak p_1 \cdots \frak p_k \frak p.
$$
\end{proof}

\begin{exercise}
Use properties of prime ideals to
show that the decomposition of the above theorem is unique up to order of the prime factors.
(We will see another justification for uniqueness later using valuations).
\end{exercise}

\begin{exercise}
Prove Theorem~\ref{thm75} using the following chain argument.
Start with any nonzero ideal $I$. Show that if $\mathfrak p$ is a prime ideal 
containing $I$ then $I = \mathfrak p I'$ for some ideal $I'$. Show that $I \subsetneq I'$.
Iterate this process to find an ascending chain and use the ascending chain condition.
\end{exercise}

\begin{corollary}
Let $R$ be a Dedekind domain. Every fractional ideal $I$ of $R$ can written in the form
$$
I = \frak p_1^{n_1} \cdots \frak p_k^{n_k}
$$
where $k \ge 0$, where $\frak p_1, \ldots, \frak p_k$ are distinct prime ideals of~$R$,
and where $n_i \in \bZ$. (Here we defined $\frak p_i^0$ to be~$R$, so we can freely drop
or add any factor with $n_i = 0$). In particular, the prime ideals generate the group 
of fractional ideals.
\end{corollary}

Next we will define discrete valutions of fractional ideals.
First we start with fractional ideals of $R_{\mathfrak p}$ where $R$ is a Dedekind
domain and where~$\mathfrak p$ is a nonzero prime ideal.
Recall that the map 
$$k \mapsto \left( \mathfrak p R_{\mathfrak p}\right)^k$$
defines an isomorphism between the additive group~$\bZ$
and the multiplicative group~$\mathcal I (R_{\mathfrak p})$ of fractional
ideals of $R_{\mathfrak p}$
(See Proposition~\ref{prop24}).
We define the valuation isomorphism to be the inverse of this isomorphism.
We call this map $v_{\mathfrak p}$ (when there is no risk of confusion with the
valuation $v_{\mathfrak p}$ on elements):
$$
v_{\mathfrak p}\colon \mathcal I (R_{\mathfrak p}) \to \bZ,\qquad 
\left( \mathfrak p R_{\mathfrak p}\right)^k \mapsto k.
$$
Now recall that localization map $I \mapsto I R_{\mathfrak p}$  is a surjective group homomorphism
from the group $\mathcal I (R)$ of fractional ideals of $R$ to $\mathcal I (R_{\mathfrak p})$.
The valuation on $\mathcal I (R)$ is defined as the composition
$$
\mathcal I (R) \to \mathcal I (R_{\mathfrak p}) \to \bZ.
$$
We also call this map $v_{\mathfrak p}$. Warning: we are now using the symbol $v_{\mathfrak p}$
for three functions:
$$
K^{\times} \to \bZ, 
\qquad 
\mathcal I (R_{\mathfrak p}) \to \bZ, 
\qquad
\mathcal I (R)  \to \bZ
$$
where $K$ is the fraction field of~$R$. Context will dictate which is meant in any
given use. We now summarize the definition of $v_{\mathfrak p} \colon \mathcal I (R)  \to \bZ$:

\begin{definition}
Let $R$ be a Dedekind domain and let $\mathcal I (R)$ be the group of fractional ideals of~$R$.
Let $\mathfrak p$ be a nonzero prime ideal.
Then we define the valuation map $$v_{\mathfrak p} \colon \mathcal I (R) \to \bZ$$
as follows: if $I \in\mathcal I (R)$ then $v_{\mathfrak p}(I)$
is the unique $k\in \bZ$ such that $I R_{\mathfrak p} = \left( p R_{\mathfrak p}\right)^k$.
\end{definition}

From the definition, and the discussion proceeding it, we get the following:

\begin{proposition}
Let $R$ be a Dedekind domain  and let $\mathcal I (R)$ be the group of fractional ideals of~$R$.
Let $\mathfrak p$ be a nonzero prime ideal.
Then the valuation map~$\mathcal I (R) \to \bZ$ is a surjective homomorphism
from the multiplicative group $\mathcal I (R)$ to the additive group~$\bZ$.
\end{proposition}

Next we show that this new valuation is compatible with the valuation of elements.

\begin{proposition}
Let $R$ be a Dedekind domain with fractions field~$K$ and
let $\mathfrak p$ be a nonzero prime ideal. Then for all $x \in K^{\times}$,
$$v_{\mathfrak p}(x R) = v_{\mathfrak p}(x).$$
\end{proposition}

\begin{proof}
Write $x$ as $u \pi^k$ where $\pi \in R_{\mathfrak p}$
is a uniformizer and $u$ is a unit of $R_{\mathfrak p}$.
Under localization, $x R$ maps to $x R_{\mathfrak p}$, which
is $\pi^k R_{\mathfrak p} = \left( \mathfrak p R_{\mathfrak p}\right)^k$.
Thus~$v_{\mathfrak p}(x R)  = k$. But also
$v_{\mathfrak p}(x) = v_{\mathfrak p}(u \pi^k) = k$.
\end{proof}

There is another characterization of this valuation map:

\begin{proposition}\label{prop73}
Let $R$ be a Dedekind domain with fraction field~$K$ and let $\mathfrak p$ be a nonzero prime ideal.
If $I$ is a fractional ideal of~$R$ then
$$
v_{\mathfrak p}(I)  = \min_{x \in I} \left\{ v_{\mathfrak p}(x)\right\}.
$$
\end{proposition}

\begin{proof}
If $k = v_{\mathfrak p}(I)$ then
$I R_{\mathfrak p} = \left( \mathfrak p R_{\mathfrak p}\right)^k$.
But~$\left( \mathfrak p R_{\mathfrak p}\right)^k$ is just the set of elements $x$ 
of~$K^\times$ with $v_{\mathfrak p}(x) \ge k$.
So if $x\in I$ then $x \in I R_{\mathfrak p} = \left( \mathfrak p R_{\mathfrak p}\right)^k$ and 
so~$v_{\mathfrak p}(x) \ge k$.

We just need to show that our minimum is $k$ by showing that 
there is an element~$x \in I$ with $v_{\mathfrak p}(x) = k$. To that end,
start with an element $x/s$ where $x \in I$ and $s \in R \smallsetminus \mathfrak p$
such that $x/s \in I R_{\mathfrak p} = \left( p R_{\mathfrak p}\right)^k$
but not in~$\left( p R_{\mathfrak p}\right)^{k+1}$.
When we localize, $s$ becomes a unit in $R_{\mathfrak p}$
so
$$
k = v_{\mathfrak p}(x/s)  = v_{\mathfrak p}(x) - v_{\mathfrak p}(s) = v_{\mathfrak p}(x) - 0 = v_{\mathfrak p}(x).
$$
\end{proof}

\begin{remark}
In the above we can actually take the minimum of $v_{\mathfrak p}(x)$ for $x$ in a generating set of~$I$.
\end{remark}

\begin{lemma}
Let $R$ be a Dedekind domain and let $\mathfrak p, \mathfrak q$ be distinct nonzero prime ideals of~$R$.
Then $\mathfrak q R_{\mathfrak p} = R_{\mathfrak p}$. 
Thus~$v_{\mathfrak p}(\mathfrak q) = 0$.
\end{lemma}

\begin{proposition}\label{unique_prop}
Let $I$ be a fractional ideal of a Dedekind domain $R$.
Suppose
$$
I = \frak p_1^{n_1} \cdots \frak p_k^{n_k}
$$
where $k \ge 0$, where $\frak p_1, \ldots, \frak p_k$ are distinct nonzero prime ideals of~$R$,
and where each $n_i \in \bZ$.
Then~\text{$n_i = v_{\frak p_i} (I)$}.
If a nonzero prime ideal $\mathfrak p$ is not any of these $\frak p_i$ then~$v_{\frak p} (I) = 0$.
\end{proposition}

\begin{proof}
Consider the localization homomorphism $I \mapsto I R_{\mathfrak p}$.
\end{proof}

\begin{remark}
This gives the uniqueness of the prime factorization of a fractional ideal.
So between this result and existence of prime factorizations (Theorem~\ref{thm75}) we 
can conclude
that the group of fractional ideals $\mathcal I$ of a Dedekind domain $R$ is
a the free Abelian group generated by nonzero prime ideals. 

Let $\mathcal P$ be the subgroup of principal fractional ideals. 
The quotient group $\mathcal I / \mathcal P$ is
called the \emph{class group of $R$} and is one of the most important
invariants of $R$. In some sense it is a measure on how much an Dedekind domain fails 
to be a PID.
\end{remark}

We can summarize our results as follows: 

\begin{theorem}
Let $I$ be a fractional ideal of a Dedekind domain $R$.
Then~$v_{\frak p} (I)$ is $0$ for all but a finite number of nonzero prime ideals $\mathfrak p$.
Suppose $\frak p_1, \ldots, \frak p_k$ are distinct prime ideals of~$R$ containing at least all
nonzero prime ideals $\mathfrak p$ with $v_{\frak p} (I) \ne 0$.
Then
$$
I = \frak p_1^{v_{\frak p_i}(I)} \cdots \frak  p_k^{v_{\frak p_k}(I)}.
$$
\end{theorem}

\begin{corollary}
Let $I, J$ be fractional ideals of a Dedekind domain $R$.
Then $I = J$ if and only if $v_{\mathfrak p} (I) = v_{\mathfrak p}(J)$
for all nonzero prime ideals~$\mathfrak p$ of~$R$.
\end{corollary}

\begin{lemma}
Let $R$ be a Dedekind domain and let $\mathfrak p$ be a nonzero prime ideal of~$R$.
If~$k \le l$ then $\left( p R_{\mathfrak p}\right)^l\subseteq \left( p R_{\mathfrak p}\right)^k$ and so
$$
\left( p R_{\mathfrak p}\right)^k + \left( p R_{\mathfrak p}\right)^l = \left( p R_{\mathfrak p}\right)^k
\quad\text{and}\quad
\left( p R_{\mathfrak p}\right)^k \cap \left( p R_{\mathfrak p}\right)^l = \left( p R_{\mathfrak p}\right)^l.
$$
\end{lemma}

\begin{proposition}\label{prop79}
Let $R$ be a Dedekind domain and let $\mathfrak p$ be a nonzero prime ideal of~$R$.
If $I, J$ are fractional ideals of $R$, then
$$
v_{\mathfrak p}(I J ) = v_{\mathfrak p}(I )  + v_{\mathfrak p}(J ),
$$
$$
v_{\mathfrak p}(I +J ) = \min (v_{\mathfrak p}(I ), v_{\mathfrak p}(J ) ),
$$
$$
v_{\mathfrak p}(I \cap J ) = \max (v_{\mathfrak p}(I ), v_{\mathfrak p}(J ) ).
$$
\end{proposition}

\begin{proof}
The first is just a statement of the homomorphism property.
For the other two, use the above lemma.
\end{proof}

Also, I should mentioned the following straightforward consequence of the above results:

\begin{proposition}
Let $R$ be a Dedekind domain and let $I$ be a fractional ideal of~$R$.
Then $I$ is an integral ideal if and only if $v_{\mathfrak p}(I) \ge 0$
for all nonzero prime ideals $\mathfrak p$.
\end{proposition}

\begin{exercise}
Let $I_1, I_2, J$ be fractional ideals of a Dedekind domain.
Show that 
$$ (I_1 + I_2) \cap J = I_1 \cap J + I_2 \cap J$$
and
$$ (I_1 \cap I_2) + J = (I_1 + J) \cap (I_2 + J).$$
Which inclusions are valid for general integral domains?
\end{exercise}

\begin{exercise}
Let $I_1, I_2, J$ be fractional ideals of a Dedekind domain.
Show that 
$$ (I_1 \cap I_2) J = I_1 J \cap I_2 J.$$
Which inclusion is valid for general integral domains?
\end{exercise}


\chapter{Divisibility for nonzero ideals}

We now discuss divisibility for ideals in a Dedekind domain.
Recall that for ideal $I \mid J$ means that there is an ideal $I'$ such that~$J = I I'$.
In this case we also say that $J$ is a \emph{multiple} of~$I$.
We define \emph{common divisors} and \emph{common multiples} in the usual way.

\begin{proposition}
Divisibility in an integral domain $R$ satisfies the following:
\begin{enumerate}
\item
Let $I_1, I_2, I_3$ be ideals of $R$. If $I_1 \mid I_2$ and $I_2 \mid I_3$
then $I_1 \mid I_3$.  

\item
Let $I, J$ be ideals of $R$. If $I \mid J$ then $J \subseteq I$.

\item
Let $I, J$ be  ideals of~$R$. If $I \mid J$ and $J \mid I$ then $I  = J$.

\item
Let $I$ be an ideal of~$R$. Then $I \mid I$ and $R \mid I$.

\end{enumerate}
\end{proposition}

\begin{proof}
This is straightforward. Note: the second claim gives the third.
\end{proof}

Using factorization into prime ideals in a Dedekind domain, we get the following: 

\begin{proposition}\label{prop82}
Let $I, J$ be nonzero ideals of a Dedekind domain $R$. Then $I \mid J$
if and only if $v_{\mathfrak p} (I) \le v_{\mathfrak p}(J)$ for all nonzero prime ideals~$\mathfrak p$ of~$R$.
\end{proposition}

\begin{theorem}\label{thm89}
Let $I, J$ be nonzero ideals of a Dedekind domain $R$. Then
$$\text{$I \mid J$\quad
if and only if \quad $J \subseteq I$.}
$$
\end{theorem}

\begin{proof}
Suppose $J \subseteq I$.
Let $I'$ be the fractional ideal $I^{-1} J$. Observe that $II' = J$.
Observe that $I'$ is actually an integral ideal since
$$I' = I^{-1} J \subseteq  I^{-1} I \subseteq R.$$
\end{proof}

\begin{remark}
The above translates into a popular phrase for Dedekind domains:
$$\text{``To contain is to divide''.}$$ It is also interesting to note that
this property characterizes 
Dedekind domains. (See Section~\ref{ch12}).
\end{remark}

\begin{corollary}
Let $I, J$ be nonzero ideals of a Dedekind domain $R$. Then $I \cap J$
is the least common multiple of~$I$ and $J$.
\end{corollary}

\begin{remark} 
Here ``least'' means the minimum according to the divisibility relation, not the inclusion relation
which reverses the partial order.
\end{remark}

Another consequence of Theorem~\ref{thm89} is that this notion of divisibility 
is compatible with divisibility of elements:

\begin{proposition}
Suppose $R$ is a Dedekind domain and that $a, b \in R$ are nonzero. Then~$a\mid b$
if and only if $a R \mid b R$.
\end{proposition}

Note also the following.

\begin{proposition}
Let $R$ be a Dedekind domain. If $a \in R$ is nonzero, and $I$ is a nonzero ideal,
then $a \in I$ if and only if $I \mid aR$.
\end{proposition}

\begin{proposition}
If  $I, J$ are nonzero ideals of a Dedekind domain then $I+J$ is the greatest common divisor of $I$
and $J$.
\end{proposition}

\begin{proof}
By Proposition~\ref{prop79}, for each prime ideal~$\mathfrak p$, 
$$
v_{\mathfrak p}(I +J ) = \min (v_{\mathfrak p}(I ), v_{\mathfrak p}(J ) ).
$$
Now use Proposition~\ref{prop82}. (Or you can argue from Theorem~\ref{thm89}).
\end{proof}

\begin{remark} 
Here ``greatest'' means the maximum according to the divisibility relation, not the inclusion relation.
\end{remark}

\begin{definition}
Let $I$ and $J$ be nonzero ideals of an Dedekind domain~$R$. Then we say that $I$ and $J$
are \emph{relatively prime} if the only common divisor of $I$ and $J$ is~$R$. In other words, $I+J = R$,
or equivalently $\min (v_{\frak p}(I), v_{\frak p}(J)) = 0$ for all nonzero prime ideals~$\mathfrak p$ of $R$.
\end{definition}

\begin{remark}
This is equivalent, of course, to having prime factorizations that share no common prime ideals.
\end{remark}

Prime ideals behave as primes with the divisibility relation.

\begin{proposition}
Let $\mathfrak p$ be a proper nonzero ideal of a Dedekind domain $R$. Then~$\mathfrak p$
is a prime ideal if and only if the following holds:
For all ideals $I, J$ of $R$, if $\mathfrak p | I J$ then $\mathfrak p | I$
or $\mathfrak p \mid J$.
\end{proposition}

Finally, we can characterize powers of primes using valuations:

\begin{proposition} \label{prop90}
Let $\frak p$ be a nonzero prime in a Dedekind domain, and let $n \in \bN$.
Then 
$$
\frak p^n = \{ a \in R \mid v_{\frak p} (a) \ge n\}.
$$
\end{proposition}

\begin{proof}
If $a \in \frak p^n$ then $a R \subseteq \frak p^n$, so $\frak p^n \mid a R$.
Thus $n = v_{\frak p} (\frak p^n) \le v_{\frak p} (a R) = v_{\frak p}(a)$.

Conversely, if $v_{\frak p} (a) \ge n$ then 
$v_{\frak p} (\frak p^n) \le v_{\frak p} (a R)$.
For any other nonzero prime ideal $\frak q$, we have $v_{\frak q} (\frak p^n) = n v_{\frak q} (\frak p) = 0$,
so $v_{\frak q} (\frak p^n) \le v_{\frak q} (a R)$.
Thus $\frak p^n \mid a R$, and so~$a \in \frak p^n$.
\end{proof}

\begin{exercise}\label{ex28}
Let $I, J$ be ideals of a Dedekind domain.
Show that 
$$ (I + J) (I \cap J)=I J.$$
Which inclusion is valid for general integral domains?

Note this says that the least common multiple of two ideals times the greatest common divisor
is just the product. This generalizes a basic identity of integers.
\end{exercise}

\begin{exercise}
Let $I$ and $J$ be nonzero ideals in a Dedekind domain~$R$.
Show that~$I J \subseteq I \cap J$ with equality if and only if $I+J=R$.
Can you generalize at least part of this to ideals in an integral domain?
(Hint: see the proof of the Chinese remainder theorem below.) 
\end{exercise}

\begin{exercise}
Let $I$ be a nonzero ideal of a Dedekind domain $R$. Show that $I$ is a~$k\ge 0$
power if and only if $v_{\mathfrak p}(I)$ is a multiple of $k$ for all 
nonzero prime ideals~$\mathfrak p$. 
Does this hold for fractional ideals? Does this hold for elements? What if $R$ is a~PID?
What if $R$ is $\bZ$.
\end{exercise}

\begin{exercise}
Show that if $I^n = J^n$ for fractional ideals $I, J$ in a Dedekind domain, then $I = J$.
Show that if $I^{n} \mid J^m$  where $m \le n$ then $I \mid J$. (Here $m, n$ are positive integers.)
\end{exercise}

\begin{exercise}
Suppose $I_1 I_2 = J^k$ for nonzero ideals $I_1, I_2, J$ in a Dedekind domain, and where $k\ge 1$.
Show that if $I_1$ and $I_2$ are relatively prime then 
$I_1 = J_1^k$ and~$I_2 = J_2^k$ for unique nonzero ideals $J_1, J_2$. Show also that $J_1 J_2 = J$.
\end{exercise}


\chapter{Approximation Theorems}\label{approximation_theorem_ch}

Given elements $a$ and $b$ in an integral domain with ideal $I$, we write
$$
a \equiv b \mod I
$$
to mean $a-b \in I$. 
If $\mathfrak p$ is a nonzero prime ideal in a Dedekind domain, then
we think of $a \equiv b \mod \mathfrak p^k$
as indicating that $a$ approximates $b$ from the point of view of $\mathfrak p$ (or the valuation $v_{\mathfrak p}$).
The larger the exponent $k$, the better the approximation. We can also state this
approximation as $v_{\mathfrak p} (a - b) \ge k$.

We will be simultaneously approximating with respect to multiple nonzero prime ideals.
There is a sense in which different nonzero prime ideals are independent.
More generally, if $I + J = R$ where $R$ is an integral domain, then the ideals $I$ and $J$ are independent
in some sense. This is expressed by the
Chinese remainder theorem (a straightforward generalization of the Chinese remainder theorem
of elementary number theory).

\begin{theorem} [Chinese remainder theorem]
Suppose $R$ is a commutative ring and suppose $I, J$ are ideals of $R$ where $I+J = R$. Then
the natural homomorphism 
$$
R / IJ \to (R/I) \times (R/J), \qquad [a] \mapsto ([a], [a])
$$
is a ring isomorphism.
\end{theorem}

\begin{proof}
Start with the natural homomorphism 
$$
R \to (R/I) \times (R/J), \qquad a \mapsto ([a], [a]).
$$
Since $I + J = R$, we can find $e_2 \in I$ and $e_1 \in J$ such that $e_2 + e_1 = 1$.
Observe that~$e_1$ maps to $([1], [0])$ and $e_2$ maps to $([0], [1])$.
So, given $([a], [b]) \in (R/I) \times (R/J)$, we can find an element of $R$ mapping to it,
namely $a e_1 + b e_2$. So surjectivity is established.

Then kernel of $R \to (R/I) \times (R/J)$ is $I \cap J$. Clearly $IJ \subseteq I \cap J$.
Suppose that~$x \in I \cap J$. Let $e_1, e_2$ be as before. Then
$$
x = (e_2 + e_1) x = e_2 x + e_1 x \in I J.
$$
Thus $IJ = I \cap J$, and the kernel of $R \to (R/I) \times (R/J)$ is $IJ$.
\end{proof}

\begin{remark}
Above we used the Cartesian product of two rings. This is a ring where addition
and multiplication are defined componentwise.
(Actually for the results in this section, we really only need to consider the product as a group under addition).
\end{remark}

\begin{remark}
We could have written this isomorphism with $I\cap J$ instead of $I J$:
$$R / I\cap J \to (R/I) \times (R/J), \qquad [a] \mapsto ([a], [a]).$$
Often all we need is the surjection:
$$R \to (R/I) \times (R/J), \qquad a\mapsto ([a], [a]).$$
\end{remark}

This allows us to solve systems of congruences involving distinct prime ideals.

\begin{proposition} 
Let $R$ be a Dedekind domain, let~$\mathfrak p_1, \ldots, \mathfrak p_k$
be distinct nonzero prime ideals of~$R$, and let $n_1, \ldots, n_k$ be nonnegative integers.
Then the natural homomorphism 
$$
R \to (R/\mathfrak {p}_1^{n_1}) \times \cdots \times (R/{\mathfrak p}_k^{n_k})
$$
is a surjection. In other words, given $b_1, \ldots, b_k \in R$ we can find an $a \in R$ such that
$$
a \equiv b_i \mod \mathfrak p_i^{n_i}
$$
for all $i$.
\end{proposition}

\begin{proof}
We use induction and the Chinese remainder theorem to first get an isomorphism
involving $R/I$ where $I = \mathfrak {p}_1^{n_1} \cdots {\mathfrak p}_k^{n_k}$.
From the isomorphism produce the surjection using $R\to R/I$.
\end{proof}

\begin{theorem} [Approximation theorem]
Let $R$ be a Dedekind domain with fraction field~$K$, let~$\mathfrak p_1, \ldots, \mathfrak p_k$
be distinct nonzero prime ideals of~$R$, let $x_1, \ldots, x_k \in K$,
and let~$n_1, \ldots, n_k$ be integers.
Then there is an $x \in K$ such that 
$$v_{\mathfrak p_i}(x - x_i) \ge n_i$$
for each $\mathfrak p_i$, and such that 
$v_{\mathfrak p}(x) \ge 0$
for any other nonzero prime ideal $\mathfrak p$.
\end{theorem}

\begin{proof}
Observe that it is enough to prove this for $n_i \ge 0$, so we assume 
nonnegative~$n_i$.
If~$x_i \in R$ then we just apply the previous proposition
with~\text{$b_i = x_i$}.

In the general case let $d\in R$ be a common denominator for the $x_i$, and
write each $x_i$ as~$b_i/d$ with $b_i \in R$.
We then wish to find an $a \in R$ such that 
$$v_{\mathfrak p_i}(a - b_i) \ge n_i + v_{\mathfrak p_i}(d).$$
We also want $v_{\mathfrak p}(a) \ge v_{\mathfrak p}(d)$ for any other nonzero
prime ideal with $v_{\mathfrak p}(d) > 0$. We can find such an $a\in R$
by appealing to the first case. Now consider $x = a/d$.
\end{proof}

\begin{theorem} [Second Approximation theorem]
Let $R$ be a Dedekind domain with fraction field~$K$, let~$\mathfrak p_1, \ldots, \mathfrak p_k$
be nonzero prime ideals of~$R$, and let $n_1, \ldots, n_k$ be integers.
Then there is an element $x\in K$ such that
$v_{{\mathfrak p}_i} (x) = n_i$ for each~${\mathfrak p}_i$
and such that $v_{{\mathfrak p}} (x) \ge 0$ for any other nonzero prime ideal~${\mathfrak p}$.
\end{theorem}

\begin{proof}
For each such prime $\mathfrak p_i$,
let $x_i \in \mathfrak {p}_i^{n_i}\smallsetminus \mathfrak {p}_i^{n_i+i}$.
Use the approximation theorem to find an $x \in K$ such that
$v_{\mathfrak p_i}(x - x_i) \ge n_i + 1$
for each $\mathfrak p_i$, and such that 
$v_{\mathfrak p}(x) \ge 0$
for any other nonzero prime ideal $\mathfrak p$.

Observe that $v_{{\mathfrak p}_i}(x_i) = n_i$ (Proposition~\ref{prop90}).
Since $x = x_i + (x - x_i)$,
$$v_{{\mathfrak p}_i}(x) = 
\min \{ v_{{\mathfrak p}_i}(x_i), v_{{\mathfrak p}_i}(x-x_i) \}
=
v_{{\mathfrak p}_i}(x_i) = n_i.
$$
\end{proof}

\begin{theorem}
Let $R$ be a Dedekind domain with a finite number of prime ideals. Then $R$ is a PID.
\end{theorem}

\begin{proof}
Let $I$ be a nonzero ideal of $R$.
Use the second approximation theorem to find an element $a\in R$
such that $v_{\frak p}(a)$ is equal to $v_{\frak p}(I)$
for all nonzero prime ideals~$\frak p$ of~$R$. Thus $v_{\frak p}(a R) =  v_{\frak p}(a)  = v_{\frak p}(I)$ for all 
such $\mathfrak p$. So~$a R = I$.
\end{proof}

\begin{theorem}
Let $R$ be a Dedekind domain. Then every nonzero ideal $I$ of $R$ can be generated by one or two elements.
\end{theorem}

\begin{proof}
Let $I$ be a nonzero ideal of $R$.
By the second approximation theorem there is an element $a\in R$
such that $v_{\frak p}(a) = v_{\frak p}(aR)$ is equal to $v_{\frak p}(I)$
for all nonzero prime ideals~$\frak p$ of~$R$ dividing~$I$. Thus $I \mid a R$.

By the second approximation theorem a second time  there is an element $b\in R$
such that $v_{\frak p}(b) = v_{\frak p}(bR)$ is equal to $v_{\frak p}(I)$
for all nonzero prime ideals~$\frak p$ of~$R$ that 
divide $a R$
(including those dividing~$I$).

Observe that $a R + bR$ has the same valuation as $I$ for all valuations $v_{\frak p}$
associated to nonzero prime ideals of $R$.
\end{proof}

\begin{exercise}
Show that every \emph{fractional} ideal of a Dedekind domain can be generated by one or two elements.
\end{exercise}

\begin{exercise}
Show that for every nonzero ideal $I$ in a Dedekind domain, there is a 
nonzero principal ideal relatively prime to $I$.
\end{exercise}

\begin{exercise} 
Let $\mathfrak p$ be a nonzero prime ideal and let $I$ be a fractional ideal in
a Dedekind domain~$R$. Then the quotient $I / \mathfrak p I$ can be thought of as an $R$-module, as
can~$R /\mathfrak p$.
Show that $R /\mathfrak p$ is isomorphic as an $R$-module to~$I / \mathfrak p I$. 
Furthermore,
show that there is an isomorphism of the form $[r] \mapsto [r x]$ for each $x \in I$ not in~$\mathfrak p I$.

Hint: start with the composition $R \to I \to I / \mathfrak p I$ and show that the kernel is a 
proper ideal containing~$\mathfrak p$. For surjectivity, observe that 
any $R$-submodule of~$I / \mathfrak p I$ corresponds to a fractional
ideal containing 
$\mathfrak p I$ and contained in~$I$. Now use unique factorization of fractional ideals
(or multiply by a nonzero $d$ such that $d I$ is an integral ideal).
\end{exercise}


\chapter{Gauss's lemma in Dedekind domains}\label{ch_gauss_lemma}

There are other important results for general Dedekind domains including (1) Gauss's lemma
for polynomials with coefficients in a Dedekind domain, (2) results about modules, especially finitely generated modules,
over a Dedekind domain, and (3) results concerning the relationship between 
two Dedekind domains $R_1 \subseteq R_2$,
especially when the fraction field of $R_2$ is a finite extension of the fraction field of~$R_1$.
We won't treat (2) and (3) in this document, but we will touch on (1) in this section.
We will start with the case of a discrete valuation ring.

Let $R$ be a DVR with maximal ideal~$\mathfrak m$,  fraction field~$K$, and valuation $v$. 
It is possible to extend $v$, in a natural way, to a function (also called $v$)
$$v\colon K[X]\to \bZ \cup \{\infty\}$$
which we call the \emph{valuation map on polynomials}. We define $v(0)$ to be~$\infty$.
For any nonzero $f \in K[X]$, let $v(f)$ be the minimum of $v(a)$ among coefficients of $f$.
So, of course if $f = a$ is a constant polynomial, then $v(f)$ (using the valuation map on polynomials) agrees with $v(a)$
(using the original valuation map).

\begin{proposition}
Let $R$ be a DVR with fraction field~$K$. Let $v$ be the valuation map on polynomials
and let $f, g \in K[X]$. Then
$$
v(f + g) \ge \min \{v(f), v(g)\}.
$$
\end{proposition}

\begin{proposition}
Let $R$ be a DVR with fraction field~$K$. Let $v$ be the valuation map on polynomials
and let $f \in K[X]$. Then $f \in R[X]$ if and only if $v(f) \ge 0$.
\end{proposition}

Polynomials $f$ is such that $v(f) =0$ are called \emph{primitive polynomials}.
These include all monic polynomials in $R[X]$.
Recall that there is a natural surjective ring homomorphism
$$
R[X] \to (R/\mathfrak m)[X]
$$
that acts by replacing each coefficient with its equivalence class.
Primitive polynomials are exactly the polynomials in $R[X]$ that have nonzero image.

\begin{proposition}
Let $R$ be a DVR with maximal ideal~$\mathfrak m$. Let $v$ be the valuation map on polynomials
and let $f \in R[X]$. Then $v(f) = 0$ if and only if the image of~$f$ in $(R/\mathfrak m) [X]$ is nonzero.
\end{proposition}

We wish to show that the valuation map on polynomials is multiplicative. We start with an easy case:

\begin{lemma}
Let $R$ be a DVR with fraction field~$K$. Let $v$ be the valuation map on polynomials, let~$a \in K^\times$,
and let $f \in K[X]$. Then $v(a f) = v(a) + v(f)$.
\end{lemma}

\begin{proposition} [Gauss's lemma for DVRs]\label{gauss_lemma_DVR_thm}
Let $R$ be a DVR with fraction field~$K$. Let $v$ be the valuation map on polynomials and
let $f, g \in K[X]$. Then $$v(f g) = v(f) + v(g).$$
\end{proposition}

\begin{proof}
The zero case is straightforward, so we assume $f$ and $g$ are nonzero.
If~$\pi$ is a uniformizer, then write~$f = \pi^k f_0$ and~$g = \pi^l g_0$ 
where~$k$ is~$v(f)$, where~$l$ is~$v(g)$,
and where~$f_0, g_0 \in K[X]$.
Observe that~$v(f_0) = v(g_0) = 0$. In particular,~$f_0$ and~$g_0$ are in~$R[X]$. Also, the images of $f_0$ and $g_0$
in $(R/\mathfrak m) [X]$ are nonzero where~$\mathfrak m$ is the maximal ideal of~$R$.
This means that the product~$f_0 g_0$ has nonzero image as well since $R [X] \to (R/\mathfrak m) [X]$ is a homomorphism
and $(R/\mathfrak m) [X]$ is an integral domain.
Thus $v(f_0 g_0) = 0$. Since~$f g = \pi^{k+l} f_0 g_0$, we have $v(fg) = k + l + 0 = v(f) + v(g)$.
\end{proof}

We can extend the valuation map on polynomials further to a valuation~$v\colon K(X)^\times \to \bZ$
where $K(X)$ is the fraction field of~$K[X]$. 
We call this the \emph{the valuation of $K(X)$ induced
by the valuation $v$ of~$K$}. We use the same symbol $v$ for the valuation of $K(X)$ using context to distinguish
the various meanings of~$v$.
This valuation is define as follow for $f, g \in K[X]$ both nonzero:
$$
v(f/g) \; \defeq \; v(f) - v(g).
$$
Here we are ignoring $0$, but we set $v(0) = \infty$ if needed.

\begin{lemma}
The above function is well-defined, and extends the valuation map on polynomials.
\end{lemma}

\begin{proposition}
Let $R$ be a DVR with fraction field~$K$. Let $v$ be the induced valuation on~$K(X)$
and let $f, g \in K(X)$. Then
$$
v(f + g) \ge \min \{v(f), v(g)\}.
$$
\end{proposition}

\begin{proof}
This is straightforward when $f$ and $g$ are written as fractions with a common denominator
\end{proof}

\begin{proposition}
Let $R$ be a DVR with fraction field~$K$. Let $v$ be the induced valuation on~$K(X)$
and let $f, g \in K(X)$. Then
$$
v(f g) = v(f) + v(g).
$$
\end{proposition}

\begin{proposition}
Let $R$ be a DVR with fraction field~$K$. Then the induced valuation is a 
valuation map $K(X)^\times \to \bZ$.
The valuation ring consists of elements of the form $r f/g$ where $f, g \in K[X]$ are primitive
polynomials and $r\in R$. The maximal ideal consists of elements of this form $r f/g$ where $r$
is in the maximal ideal of~$R$. 
If $\pi$ is a uniformizer of $R$ and if $f, g$ are primitive polynomials, then $v(\pi^k f/g) = k$,
so $\pi$ is a uniformizer for the valuation ring associated to $K(X)^\times \to \bZ$.
\end{proposition}

\bigskip

Now we shift to a general Dedekind domain $R$. For every nonzero prime $\mathfrak p$ of~$R$
we have the valuation $v_{\mathfrak p}$ of $R_{\mathfrak p}$.
We extend $v_{\mathfrak p}$ to $K(X)^\times$ as above. (We also have the extension of $v$ to fractional ideals
of $R$ which we will also need).

\begin{proposition}
Let $f \in K(X)^\times$ where $K$ is the fraction field of a Dedekind domain~$R$. Then
$v_{\mathfrak p} (f) \ne 0$ for only a finite number of nonzero prime ideals~$\mathfrak p$ of~$R$.
\end{proposition}

It turns out that these extended valuations are closely connected to the concept of the \emph{content}
of a polynomial.

\begin{definition}
Let $R$ be an integral domain with fraction field~$K$.
If $$f = a_n X^n + \ldots + a_1 X + a_0 \in K[X]$$ 
then the \emph{content} of $f$ is defined as the following fractional ideal:
$$
\text{content} (f) \; \defeq \; a_n R + \ldots + a_1 R + a_0 R.
$$
\end{definition}

\begin{proposition}
Let $R$ be an integral domain with fraction field~$K$, then
$$
\mathrm{content} (fg ) \subseteq \mathrm{content} (f) \, \mathrm{content} (g).
$$
for any polynomials $f, g \in K[X]$.
\end{proposition}

Now we consider the case where $R$ is a Dedekind domain.

\begin{proposition}
Let $R$ be a Dedekind domain with fraction field~$K$.
If $f \in K[X]$  is nonzero then
$$
\mathrm{content} (f) \; = \; \prod_{\mathfrak p} \mathfrak p^{v_{\mathfrak p} (f)}
$$
where $\mathfrak p$ varies among any given finite set of nonzero prime ideals of $R$ 
containing at least all primes $\mathfrak p$ with $v_{\mathfrak p} (f) \ne 0$.
\end{proposition}

\begin{proof} 
By Corollary~\ref{equality_test_cor} it is enough to show equality locally for each nonzero prime 
ideal~${\mathfrak p}$.  If the coefficients of $f$ are $a_0, \ldots, a_n \in K$ then 
$$
\mathrm{content} (f) \, R_{\mathfrak p}=
a_0 R_{\mathfrak p} + \ldots + a_n R_{\mathfrak p}.
$$
Since $R_{\mathfrak p}$ is a DVR, we can use the classification of fractional ideals in DVRs to simplify
(where we temporarily define $\mathfrak p^{v_{\mathfrak p} (0)} R_{\mathfrak p} = 0$)
$$
\mathrm{content} (f) \, R_{\mathfrak p}= 
\mathfrak p^{v_{\mathfrak p} (a_0)} R_{\mathfrak p} + \ldots + \mathfrak p^{v_{\mathfrak p} (a_n)}  R_{\mathfrak p} 
=  \mathfrak p^{v_{\mathfrak p} (a_i)}  R_{\mathfrak p}
$$
where $v_{\mathfrak p} (a_i)$ is the minimum of $\{ v_{\mathfrak p} (a_1), \ldots, v_{\mathfrak p} (a_n)\}$.
By definition, $v_{\mathfrak p} (a_i)$ is just~$v_{\mathfrak p} (f) $.
Thus
$$
\mathrm{content} (f) \, R_{\mathfrak p}= \mathfrak p^{v_{\mathfrak p} (f)}  R_{\mathfrak p}.
$$
This is the desired local equality. The result now follows from Corollary~\ref{equality_test_cor}. 
\end{proof}

The above proposition allows us to extend the definition of content to elements of~$K(X)^\times$
(when $R$ is a Dedekind domain):

\begin{definition}
Let $f \in K(X)^\times$ where $K$ is the fraction field of a Dedekind domain~$R$. 
Then the \emph{content} of $f$ is defined as the following fractional ideal:
$$
\text{content} (f) \; \defeq \; \prod_{\mathfrak p} \mathfrak p^{v_{\mathfrak p} (f)}
$$
where $\mathfrak p$ varies among any given finite set of nonzero prime ideals of $R$ 
containing at least all primes $\mathfrak p$ with $v_{\mathfrak p} (f) \ne 0$.
\end{definition}

\begin{proposition}
Let $f \in K(X)^\times$ where $K$ is the fraction field of a Dedekind domain~$R$.
For every nonzero prime ideal $\mathfrak p$
$$
v_{\mathfrak p} (f) = v_{\mathfrak p} (\mathrm{content} (f) ).
$$
\end{proposition}

\begin{theorem}[Gauss's lemma for Dedekind domains]
Let $f, g \in K(X)^\times$ where~$K$ is the fraction field of a Dedekind domain~$R$. Then 
$$
\mathrm{content} (fg) = \mathrm{content} (f) \; \mathrm{content} (g)
$$
\end{theorem}

\begin{proof} 
By Corollary~\ref{equality_test_cor} it is enough to show equality in $R_{\mathfrak p}$ for each nonzero prime 
ideal~${\mathfrak p}$. This amounts to showing
$$
{\mathfrak p}^{v_{\mathfrak p} (fg)} R_{\mathfrak p} = 
\left( \mathfrak p^{v_{\mathfrak p} (f)}  R_{\mathfrak p} \right) \left( p^{v_{\mathfrak p} (g)} R_{\mathfrak p} \right)
$$
However $v_{\mathfrak p} (fg) = v_{\mathfrak p} (f) + v_{\mathfrak p} (g)$ by Theorem~\ref{gauss_lemma_DVR_thm}.
\end{proof}

\begin{remark}
Gauss's lemma gives a trick for finding a generating set for the product of two fractional ideals $IJ$.
Suppose $I = a_1 R+ \ldots + a_m R$, then let $f$ be the following ``basis polynomial'':
$$f \; \defeq \; a_1 + a_2 X   + \ldots + a_m X^{m-1}.$$
Observe that $I = \mathrm{content} (f)$. Let $J = \mathrm{content} (g)$ for
a similarly constructed polynomial. Then $I J = \mathrm{content} (f g)$.
Thus the coefficients of the polynomial $fg$ provide a generating set for $IJ$.

Observe, that if $I$ has $m$ generators and $J$ has $n$ generators, then we get $m + n - 1$
generators for $I J$ this way. Compare this to the $m n$ generators you get if instead
you use $a_i b_j$ as your generating set where $(a_i)$ are~$m$ generators for $I$ and $(b_i)$ are~$n$ generators for $J$.

For example, in algebraic number theory it is known that $R= \bZ[\sqrt{-5}\, ]$ is a Dedekind domain.
Since $\left(3 X + (1+ 2\sqrt{-5})\right) \left(3 X + (1- 2\sqrt{-5})\right)  = 9 X^2 + 6 x + 21$ we conclude that
$$
\left(3 R + (1+ 2\sqrt{-5}) R\right) \, 
\left(3 R + (1- 2\sqrt{-5}) R\right) = 9 R + 6 R + 21 R = 3 R
$$
which gives a factorization of the ideal $3 R$.
\end{remark}

\begin{exercise}
Suppose that $R$ is an integral domain with fraction field~$K$.
\begin{enumerate}

\item
Show that if $f = a$ is a nonzero constant polynomial,
then $\mathrm{content} (f) = a R$. 

\item
Let~$f \in K[X]$ be nonzero. Show that $f \in R[X]$ if and only if the content of $f$
is an integral ideal. 

\item
A \emph{primitive polynomial} is defined to be a polynomial in~$K[X]$ with content 
equal to the identity ideal~$R$. Show that every monic polynomial in $R[X]$ is primitive.
\end{enumerate}
\end{exercise}

\begin{exercise}
Suppose that $R$ is a Dedekind domain and that~$f, g, h \in K[X]$ are monic.
Show that if $f = g h$ and $f \in R[X]$ then $g, h \in R[X]$ as well.
\end{exercise}

\begin{exercise}
Show that if $R$ is a PID with fraction field~$K$, and if $f \in K[X]$ is nonzero, then $\mathrm{content} (f) = a R$ if and only
$f = a f_0$ where $f_0$ is a primitive polynomial. Make and justify a similar statement for $f \in K(X)^\times$.

In this case we sometimes say that $a$ is the content. 
The content considered as an element in $K^\times$ is only defined up to 
multiplication by a unit of~$R$.
\end{exercise}

\begin{exercise}
Suppose $f, g \in R[X]$ are nonzero with nonzero sum $f+ g$, where $R$ is a Dedekind
domain. Show that 
$$
 \mathrm{content} (f) + \mathrm{content} (g) 
 \mid
 \mathrm{content} (f + g)
$$
(where the sum is the sum of ideals).
\end{exercise}

\bigskip

If you weaken the assumption that $R$ is a Dedekind domain to just that $R$ is 
integrally closed, you still get some interesting results about polynomials.
In the remaining exercises in this section we consider integral domains that are not always 
Dedekind domains, but are at least integrally closed. 

\begin{exercise} \label{ex38}
Let $R$ be an integral domain that is integrally closed in its fraction field~$K$.
Show that if $r$ is a root of a nonzero $f \in R[X]$, and if $a \in R$ is the leading
coefficient of~$f$, then $a r \in R$. Hint: multiply $f$ by $a^{d-1}$ where $d$ is the
degree of~$f$ in order to form a monic polynomial with root $ar$.
\end{exercise}

\begin{exercise} \label{ex32}
Let $R$ be an integral domain that is integrally closed in its fraction field~$K$.
Show that if $r\in K$ is a root of a nonzero polynomial $f \in R[X]$ then
we have~$f(X)= q(X) (X-r)$ where $q$ is in~$R[X]$.
Hint: use (strong) induction on the degree of~$f$ and use the previous exercise 
in the context of the division algorithm in~$K[X]$. Pay special attention to the leading coefficient of~$f$.
\end{exercise}

\begin{exercise}\label{ex33}
The above exercise leads to an interesting characterization of the property of being integrally closed.
Let $R$ be an integral domain with fraction field~$K$.
Show that \emph{$R$ is integrally closed if and only if, 
for each nonzero $f \in R[X]$ and each root $r \in K$ of $f$, we can factor $f$ as 
 $q(X) (X-r)$ where $q$ is in~$R[X]$}.
 
Hint: consider the special case where $f$ is monic and look at the two terms of highest power of
$f(x)$ and the product~$q(X) (X-r)$.
\end{exercise}

\begin{exercise}
Use Gauss's lemma to give a simpler proof of 
the result of Exercise~\ref{ex32} when the hypothesis is replaced by the hypothesis 
that $R$ is a Dedekind domain. 
\end{exercise}

\begin{exercise} (Assumes some field theory background).
Let $R$ be an integral domain that is integrally closed in its fraction field~$K$.
Show that if $f \in R[X]$ is a nonzero polynomial that factors as $f = gh$
where $g, h \in K[X]$ and where $h$ is monic, then~$g \in R[X]$. 

Hint: Extend to the splitting field of~$h$,
and use Exercise~\ref{ex33} repeatedly $d$ times where~$d = \deg h$.
\end{exercise}

\begin{exercise}
Use the previous exercise to show the following
when $R$ is an integral domain that is integrally closed in its fraction field~$K$.
\emph{If $f \in R[X]$ is a monic polynomial that factors as $f = gh$
where $g, h \in K[X]$ are monic, then~$g, h \in R[X]$.}

Does the converse hold? In other words, does this property imply that $R$ is integrally closed?
\end{exercise}


\chapter{Divisibility domains and cancellation domains} \label{ch12}

Now we explore integral domains that behave like Dedekind domains in certain ways.
If an integral domain behaves like a Dedekind domain by having the property
that~$J \subseteq I \implies I \mid J$, we will call it a \emph{divisibility domain}.
If an integral domain behaves like a Dedekind domain by having a cancellation law
for nonzero ideals, we will call it a \emph{cancellation domain}.
The terms \emph{divisibility domain} and \emph{cancellation domain} should be
regarded as temporary classifications since we will show such rings
can be characterized in terms of other types of rings (see Theorem~\ref{thm98} 
and Theorem~\ref{cancellation_almost_dedekind_thm}).

\begin{definition}
A \emph{divisibility domain} is an integral domain with the property that if~$J \subseteq I$,
where $I$ and $J$ are nonzero ideals, then $I \mid J$.
\end{definition}

\begin{example}
As we have seen, every Dedekind domain is a divisibility domain.
\end{example}

\begin{example}
By Theorem~\ref{divisibility_DVR_thm}, any divisibility domain with a unique
nonzero maximal ideal is a DVR.
\end{example}

\begin{lemma}
Suppose $R$ is a divisibility domain. Then every nonzero ideal is invertible.
\end{lemma}

\begin{proof}
Let $I$ be a nonzero ideal and let $a \in I$ be a nonzero element.
Then $aR \subseteq I$. Hence $I \mid aR$.
Thus $I J = aR$ for some nonzero ideal $J$ of~$R$.
Observe that $a^{-1} J$ is an inverse for $I$.
\end{proof}

\begin{theorem}\label{thm98}
Let $R$ be an integral domain. Then $R$ is a 
divisibility domain if and only if it is a Dedekind domain.
\end{theorem}

\begin{proof}
We have already seen that every Dedekind domain is a divisibility domain
(Theorem~\ref{thm89}).
By the previous lemma and Theorem~\ref{thm64}, every divisibility domain
is a Dedekind domain since every nonzero ideal is invertible.
\end{proof}

Now we consider integral domains that are similar to Dedekind domains
by possessing a cancellation law for nonzero ideals.

\begin{definition}\label{cancellation_domain_def}
A \emph{cancellation domain} is an integral domain with the property that if~$I_1 J = I_2 J$,
where $I_1, I_2$ and $J$ are nonzero ideals, then $I_1 = I_2$.
\end{definition}

\begin{lemma} \label{cancel_lemma1}
Suppose  $I_1, I_2$, and $J$ are 
fractional ideals of a cancellation domain.
If $I_1 J = I_2 J$ then  $I_1 = I_2$.
\end{lemma}

\begin{lemma} \label{cancel_lemma2}
Suppose  $I_1, I_2$, and $J$ are 
fractional ideals of a cancellation domain.
If $I_1 J \subseteq I_2 J$ then  $I_1 \subseteq I_2$.
\end{lemma}

\begin{proof}
If $I_1 J \subseteq I_2 J$ then $(I_1 + I_2) J = I_1 J + I_2 J = I_2 J$.
Since $R$ is a cancellation domain, $I_1 + I_2 = I_2$. Thus $I_1 \subseteq I_2$.
\end{proof}

Here is a simple, but tricky lemma:

\begin{lemma}
Suppose $x \in K^\times$ where $K$ is the fraction field of a cancellation domain~$R$.
Then $xR \subseteq  x^2 R + R$.
\end{lemma}

\begin{proof}
Let $I_1 = x R$ and $I_2 = x^2 R + R$. Let $J = x R + R$. Then $I_1 J \subseteq I_2 J$.
Thus~$I_1 \subseteq I_2$.
\end{proof}

\begin{lemma}\label{lemma107}
Suppose $x \in K^\times$ where $K$ is the fraction field of a cancellation domain~$R$.
Then $a x^2 + x + b = 0$ for some $a, b \in R$.
\end{lemma}

\begin{proof}
This is a consequence of the previous lemma.
\end{proof}

\begin{lemma}
Every cancellation domain $R$  is integrally closed.
\end{lemma}

\begin{proof}
Suppose $x$ is in the fraction field $K$ of~$R$ and is integral over $R$.
Then the~$R$-submodule $I = R[x]$ of~$K$
must be a finitely generated as an $R$-module. Observe that $I I = I$, so by cancellation~$I = R$.
Thus $x \in R$ as desired.
\end{proof}

\begin{lemma} \label{cancel_lemma6}
Let $R$ be a cancellation domain~$R$ with fraction field~$K$.
Let  $\mathfrak p$ be a nonzero prime ideal of~$R$.
Then for every $x \in K^\times$ either $x \in R_{\mathfrak p}$ or $x^{-1} \in R_{\mathfrak p}$.
\end{lemma}

\begin{proof}
By Lemma~\ref{lemma107} we have that $a x^2 + x + b = 0$ for some $a, b \in R$.
This gives us~$(a x)^2 + (a x) + ab = 0$. Since $R$ is integrally closed, $a x \in R$.
Observe that~$(1 + a x ) x = -b$ so $(1+ax) x \in R$. Now $R_{\mathfrak p}$ is a local ring
so either $ax$ or $1+ax$ is a unit in~$R_{\mathfrak p}$. 
If $ax$ is a unit, then $1/x \in R_{\mathfrak p}$. If $1 + ax$ is a unit, then
$x \in R_{\mathfrak p}$.
\end{proof}

\begin{proposition}
Let $R$ be a Noetherian cancellation domain~$R$.
Then $R_{\mathfrak p}$ is a~DVR for every nonzero prime ideal of~$R$.
\end{proposition}

\begin{proof}
Since $R_{\mathfrak p}$ must be Noetherian, it must be a DVR
by the previous lemma and Proposition~\ref{propI}.
\end{proof}

\begin{theorem} \label{thm111}
Let $R$ be an integral domain. Then $R$ is a Dedekind domain if
and only if it is a Noetherian cancellation domain.
\end{theorem}

\begin{proof}
If $R$ is a Dedekind domain, it is Noetherian by definition and
every nonzero ideal is invertible (Theorem~\ref{thm64}), which makes it a cancellation domain.

If $R$ is a Noetherian cancellation domain, then 
$R_{\mathfrak p}$ is a DVR for every nonzero maximal ideal $\mathfrak p$ of~$R$ by the above proposition,
which makes it a Dedekind Domain (see Theorem~\ref{thm58}).
\end{proof}

In Appendix E we will consider cancellation domains that are not necessarily Noetherian.\footnote{We
will see that an integral domain is a cancellation domain if and only if it is an almost Dedekind domain.}
The following exercises applies to such general cancellation domains, and will be used in Appendix E.

\begin{exercise}\label{cancel_ex1}
Show that if $R$ is a cancellation domain and if $a\in R$ is not a unit then
$$
\bigcap_{k=1}^\infty a^k R = \{0\}.
$$
Hint: call this intersection $I$ and show $(a R) I = I$.
\end{exercise}

\begin{exercise} \label{cancel_ex2a}
Suppose that $R$ is an integral domain with maximal ideal~$\mathfrak m$.
Derive the identity
$$\left( {\mathfrak m}^k R_{\mathfrak m} \right) \cap R = {\mathfrak m}^k.$$

Hint: For each $a \in \left( {\mathfrak m}^k R_{\mathfrak m} \right) \cap R$, form the ideal 
$I_a = \{r \in R \mid ra \in  {\mathfrak m}^k\}$. Show that $I_a$
contains the ideal $ {\mathfrak m}^k + sR$ for some $s \in R\smallsetminus \mathfrak m$.
Conclude that $I_a  = R$.
\end{exercise}

\begin{exercise} \label{cancel_ex2}
Let $\mathfrak m$ be a nonzero maximal ideal in a cancellation domain~$R$.
Show that $\mathfrak m^2 \subsetneq \mathfrak m$.
Use the previous exercise to show that this strict inclusion continues to hold in~$R_{\mathfrak m}$:
$$
{\mathfrak m}^2 R_{\mathfrak m} \subsetneq {\mathfrak m} R_{\mathfrak m}.
$$
\end{exercise}

\begin{exercise} \label{cancel_ex3}
Let $\mathfrak m$ be a nonzero maximal ideal in a cancellation domain~$R$.
Show that $R_{\mathfrak m}$ is a valuation ring whose
maximal ideal is principal. Hint: see Exercise~\ref{valuation_ring_ex}
for the definition of valuation ring, and Exercise~\ref{valuation_ring_ex2}
to help show the maximal ideal is principal. Also, use the previous exercise.
\end{exercise}


\chapter{Characterizing Dedekind domains}

Here we summarize results, mainly from Sections \ref{ch7} and~\ref{ch12}, concerning what
properties are necessary and sufficient for an integral domain to be a Dedekind domain.

Before doing so, we derive a few more such characteristic properties. These are properties that have been
established or can easily be established for Dedekind domains, but, as we will see, any Integral domain with
these properties must be a Dedekind domain.

\begin{theorem}\label{thm101}
Suppose that $R$ is an integral domain such that every nonzero proper ideal is the product
of maximal ideals. Then $R$ is a Dedekind domain.
\end{theorem}

\begin{proof}
Let $\mathfrak m$ be a nonzero maximal ideal of~$R$ and let $a \in \mathfrak m$
be nonzero.
Then by assumption $a R$ is the product
of maximal ideal: $a R = \mathfrak m_1 \cdots \mathfrak m_k$. 
By Proposition~\ref{easy_invert_prop},
each~$\mathfrak m_i$ is invertible. Next observe that 
$\mathfrak m_1 \cdots \mathfrak m_k \subseteq \mathfrak m$.
This implies that~$\mathfrak m = \mathfrak m_i$ for some~$i$.
Thus $\mathfrak m$ is invertible.

Since every nonzero maximal ideal is invertible, any product of such ideals is invertible.
By assumption every nonzero proper ideal is the product of (nonzero) maximal ideals.
Thus every nonzero ideal is invertible and,
by Theorem~\ref{thm64}, $R$ is a Dedekind domain.
\end{proof}

\begin{theorem}\label{thm_principal_product}
Suppose that $R$ is an integral domain. Then $R$ is a Dedekind domain
if and only if for every nonzero ideal $I$ there is a nonzero ideal $J$ such that~$IJ$ is principal.
\end{theorem}

\begin{proof}
Suppose $R$ is a Dedekind domain and $I$ is a nonzero ideal. Let $a\in I$
be nonzero. Thus $a R \subseteq I$ and $I \mid a R$ as desired.

By Proposition~\ref{easy_invert_prop}, if $I$ and $J$ are nonzero ideals such that
$IJ$ is principal, then~$I$ and $J$ are invertible. So if for 
every nonzero ideal $I$ there is a nonzero ideal~$J$ such that~$IJ$ is principal, then
every nonzero ideal $I$ must be invertible.
Now use Theorem~\ref{thm64}.
\end{proof}

Now we are ready for a list of characterizations of Dedekind domains.

\begin{theorem} \label{dedekind_summary_thm}
Let $R$ be an integral domain. Then  the following are equivalent:
\begin{enumerate}
\item
$R$ is a Dedekind domain. In other words, $R$ is an integrally closed Noetherian domain whose
nonzero prime ideals are maximal.
\item
Every nonzero ideal of $R$ is invertible. (see Theorem~\ref{thm64})
\item
The  fractional ideals of~$R$ form a group under multiplication. (see Cor.~\ref{cor41})
\item
Every nonzero proper ideal is the product of maximal ideals. (Theorem~\ref{thm101})
\item
$R$ is Noetherian and $R_{\mathfrak m}$ is a DVR for each nonzero maximal ideal~$\mathfrak m$.~(Th.~\ref{thm58})
\item
$R$ is Noetherian and $R_{\mathfrak p}$ is a DVR for each nonzero prime ideal~$\mathfrak p$.
\item
$R$ is a divisibility domain. In other words, $R$ has the property that if~$J \subseteq I$,
where $I$ and $J$ are nonzero ideals, then $I \mid J$.v(see Theorem~\ref{thm98})
\item
$R$ is a Noetherian cancellation domain. In other words, $R$ is a Noetherian domain
such that if $I_1 J = I_2 J$ then $I_1 = I_2$ for all nonzero ideals $I_1, I_2, J$. (Theorem~\ref{thm111}).
\item
For every nonzero ideal $I$ of~$R$ there is an nonzero ideal $J$ of~$R$ such that $IJ$ is principal. 
(Theorem~\ref{thm_principal_product})
\end{enumerate}
\end{theorem}

\begin{proof}
From Section~\ref{ch7} we use Theorem~\ref{thm58}, 
Corollary~\ref{cor41}, and Theorem~\ref{thm64}. 
From Section ~\ref{dedekind_vals_ch} we use Theorem~\ref{thm75}.
From Section~\ref{ch12}
we use Theorem~\ref{thm98} and Theorem~\ref{thm111}.
From the current section we use Theorem~\ref{thm101} and Theorem~\ref{thm_principal_product}.
These results give most of the needed implications. The rest are straightforward.
\end{proof}

\begin{remark}
Observe that many of these characterizations are simpler than the traditional 
definition of Dedekind domain (Definition~\ref{dedekind_domain_def}). 
\emph{Why then is the traditional definition still
commonly used}\,?  Perhaps because it is the easiest to verify for rings such as the ring of integers in 
a number field.
\end{remark}

\begin{exercise}\label{ufd_ex}
Let $R$ be a Dedekind domain. Show that $R$ is a PID if and only if~$R$ is a UFD.
Generalize this to any integral domain $R$ with the property that every nonzero prime ideal is maximal.

(Recall that a UFD is an integral domain such that every nonzero non-unit element
$a \in R$ can be written uniquely as the product of irreducible elements. Uniqueness means
that if $a = p_1 \ldots p_k = q_1 \ldots q_l$ are two such products, then $k=l$,
and we can rearrange the order of the terms such that, for each $i$, the elements
$p_i$ and $q_i$ are associates ($p_i$ is a unit times $q_i$). An irreducible
element is an nonzero non-unit element that is not the product of two
nonzero non-unit elements.)

Hint for one direction: Suppose $R$ is a UFD. First show that if $\pi$ is irreducible, then $\pi R$ is a prime ideal.
Given a nonzero prime ideal $\mathfrak p$, let $a \in \mathfrak p$ be nonzero,
and factor $a$ into irreducibles. Use that factorization to show that $\mathfrak p$
is principal. Conclude that all ideals are principal. (To go from prime ideals principal to all ideas principal
is immediate in a Dedekind domain. For more general rings, given a nonzero idea $I$
factor $a \in I$ where $a$ is not zero, and work from there.)
\end{exercise}


\chapter*{Appendix A: The non-local approach to ideal factorization}

One of the main theorems for Dedekind domains, perhaps the main theorem,
is that ideals factor uniquely into prime ideals. 
Initially we assumed this result as background
since this document is in some sense a part two in the theory of 
Dedekind domains, and most introductory accounts give a (non-local) proof of the result

Although we did end up giving an independent proof of this result
built on the local approach (see Theorem~\ref{thm75}
together with Proposition~\ref{unique_prop}),
we now give, for the convenience of the reader,
a somewhat standard non-local proof so the reader can more easily compare the two approaches.\footnote{The
approach in this appendix
is largely similar to the standard approach given by~B. L. van der Waerden in his \emph{Algebra},
and has been adopted by many other textbooks. Van der Waerden attributes this approach to W.~Krull (1928).}
In particular, the proof in this appendix does not
use localization, local rings, or discrete valuation rings.
This proof does use fractional ideals,
and so depends on some of the material from 
Sections~\ref{ch2.5} to~\ref{ch4}.\footnote{Some authors, including Marcus,
prove this result using only ideals, not fractional ideals.
So the use of fractional ideals is not strictly necessary.
 All standard proofs, however,
require something like the result that every element of~$\mathcal R(I)$ is
integral, at least in the case of $I$ an ideal, which we covered in  Section~\ref{ch4} above.
Aside from that, the concepts related to fractional ideals 
presented in Sections~\ref{ch2.5} and~\ref{ch3} are fairly straightforward
and are central to the subject, so I feel comfortable requiring this of the reader
even in this basic non-local proof.}
So the reader should look over these sections, up to Corollary~\ref{new_criterion_cor},
before reading this appendix (skipping any parts of those sections dealing with
localizations or discrete valuation rings).

We will especially draw on Corollary~\ref{new_criterion_cor} from Section~\ref{ch4} 
which immediately implies the following
result which will label our first lemma:

\begin{lemma}\label{lemma101}
Let $\mathfrak p$ be a nonzero prime ideal in a Dedekind domain.
If there is a fractional ideal $I$ not contained in~$R$ such that $I \mathfrak p \subseteq R$,
then $\mathfrak p$ is invertible.
\end{lemma}

Next we use the Noetherian property:

\begin{lemma}
Let $I$ be a nonzero proper ideal of a Noetherian domain~$R$.
Then there are nonzero prime
ideals~$\frak p_1, \ldots, \frak p_k$, where $k\ge 1$, such that
$$
\frak p_1 \cdots \frak p_k \subseteq I.
$$
\end{lemma}

\begin{proof}
Suppose there are nonzero proper ideals where this fails.
Using the Noetherian property, we can find a maximal such ideal $I$.
Note that $I$ is not a prime ideal, so there are $x, y, \in R$ such that $xy \in I$
but $x, y$ are not in~$I$. We can assume the result for~$I+xR$ and $I + yR$.
Finally, observe that $(I + xR)(I + yR) \subseteq I$.
\end{proof}

Our goal will be to show that the inclusion in the above lemma is an equality
when the number of primes $k$ is minimized.

\begin{definition}
A \emph{prime bounding sequence}~$\frak p_1, \ldots, \frak p_k$ of an ideal~$I$
is a sequence of nonzero prime ideals such that
$
\frak p_1 \cdots \frak p_k \subseteq I.
$
If $k$ is as small as possible, then we call the sequence
a \emph{minimal prime bounding sequence} of~$I$.
If $I = R$ we consider the empty sequence to be the minimal prime bounding sequence.
\end{definition}

\begin{lemma} \label{lemma71}
Let $R$ be a Noetherian domain such that every nonzero prime ideal is maximal.
If $I$ is a nonzero ideal and if $\mathfrak p$ is a nonzero prime ideal containing~$I$
then~$\mathfrak p$ appears
in any prime bounding sequence of~$I$.
In particular, $I$ is contained in a finite number of prime ideals.
\end{lemma}

\begin{proof}
If 
$
\frak p_1 \cdots \frak p_k \subseteq I \subseteq \mathfrak p
$ where $\mathfrak p_i$ are nonzero prime ideals, 
then  $\mathfrak p_i \subseteq \mathfrak p$ for some~$i$. Hence $\mathfrak p_i = \mathfrak p$.
\end{proof}

Next we prove invertibility for prime ideals:

\begin{lemma}
Every nonzero prime ideal $\mathfrak p$ in a Dedekind domain $R$ is invertible.
\end{lemma}
 
\begin{proof}
Let~$a\in \mathfrak p$
be a nonzero element and let~$\frak p_1, \ldots, \frak p_k$ be a minimal prime bounding sequence
for $aR$.
By the previous lemma, $\mathfrak p$ is in the sequence. Permute the terms of the sequence
so that $\mathfrak p = \mathfrak p_k$.
By minimality, $\frak p_1 \cdots \frak p_{k-1}$ is not contained in $aR$, so
the fractional ideal
$I = a^{-1} \frak p_1 \cdots \frak p_{k-1}$ 
is not contained in $R$. (if $k=1$, let~$I = a^{-1} R$).
Also $I \mathfrak p \subseteq R$, so $\mathfrak p$ is invertible
by Lemma~\ref{lemma101}.
\end{proof}

\begin{lemma}
Every nonzero ideal~$I$ in a Dedekind domain~$R$ is the product 
of the prime ideals in its minimal prime bounding sequence.
\end{lemma}

\begin{proof}
We prove this by induction on the size $k$ of the sequence. 
The base case~$k=0$ is the empty sequence with $I=R$.
Suppose now that $\frak p_1, \ldots, \frak p_{k+1}$ is a minimal prime bounding sequence
for~$I$.
Let $\mathfrak p$ be a prime ideal containing $I$ which is necessarily in the sequence.
Permute the sequence so that~$\mathfrak p = \mathfrak p_{k+1}$.  Thus
$$
\frak p_1 \cdots \frak p_{k} =  (\frak p_1 \cdots \frak p_{k+1}) \mathfrak p^{-1} 
\subseteq I \mathfrak p^{-1} \subseteq \mathfrak p \mathfrak p^{-1} = R.
$$
Note that $\frak p_1, \ldots, \frak p_{k}$ must be a minimal prime bounding sequence
of $I' = I \mathfrak p^{-1}$ (otherwise, we could form a smaller prime bounding sequence 
for $I = I' \mathfrak p$ than the given one).
So by the induction hypothesis
$$
 I \mathfrak p^{-1}_{k+1} = I' = \mathfrak p_1 \cdots \mathfrak p_{k},
\quad\text{and so} \quad
I =  \mathfrak p_1 \cdots  \mathfrak p_{k+1}.
$$
\end{proof}

We now just need uniqueness:

\begin{lemma}
Suppose that $I$ is a proper nonzero ideal  of a Dedekind domain with prime ideal factorizations:
$$
I =  \mathfrak p_1 \cdots \mathfrak p_k = \mathfrak q_1 \cdots \mathfrak q_l.
$$
Then $\mathfrak p_1, \ldots, \mathfrak p_k$
is a permutation of $ \mathfrak q_1, \ldots, \mathfrak q_l$
\end{lemma}

\begin{proof}
Using properties of prime ideals we can show that $\mathfrak p_k = \mathfrak q_i$
for some index~$i$. After permuting the factors, we can assume that $\mathfrak p_k = \mathfrak q_l$.
We multiply both factorizations by $\mathfrak p_k^{-1}$. Continuing in this way\footnote{
We can make this into a more formal induction by setting $n$ to be the minimum of $k$ and~$l$
and proceed by induction on~$n$. Another approach is to prove the following 
statement using induction on~$n$: for any two sequences of prime ideals whose products
are equal, if a prime $\mathfrak p$ occurs exactly~$n$ times in one sequence, it occurs
exactly $n$ times in the other.}
 we get uniqueness.
\end{proof}

\begin{theorem}
Every nonzero ideal in a Dedekind domain is uniquely the product of nonzero prime ideals.
(Uniqueness is up the order of the terms).
\end{theorem}


\chapter*{Appendix B: The singular case}

In algebraic number theory and algebraic geometry there are important situations where
rings arise that are like Dedekind domains except for not being integrally closed. 
For example, nonmaximal orders in algebraic number fields
have this property. In this appendix we consider such rings.

\begin{definition}
An \emph{integral domain of dimension~$1$}, or a \emph{1-domain} for short,
is an integral domain that is not a field and such that every nonzero prime ideal is maximal.\footnote{This 
notion of dimension~$1$ is based on the notion of Krull dimension.
The Krull dimension of a commutative ring is
is one less than the length of the longest proper chain of prime ideals. In our case,
since every nonzero
prime ideal is maximal, the
longest chain is of the form $0 \subsetneq \mathfrak p$. So the Krull dimension is one.}

If a Noetherian $1$-domain is not integrally closed, and so is not a Dedekind domain, 
then we say that it is \emph{singular}.
Any nonzero prime ideal $\mathfrak p$ of $R$ such that~$R_{\mathfrak p}$ 
fails to be a DVR is called a \emph{singular prime ideal}.
\end{definition}

By the results we have established, especially Theorem~\ref{dedekind_summary_thm}, we have the following
facts about any singular Noetherian 1-domain~$R$:
\begin{enumerate}
\item
The fractional ideals of $R$ do not form a
group: some nonzero ideals of $R$ are not invertible.
\item
There is at least one singular prime ideal~$\mathfrak p$.
\item
There are nonzero ideals $J \subseteq I$ where~$I \mid J$ fails.
\item
There is a nonzero proper ideal of $R$ that is not the product of prime ideals.
Since $R$ is Noetherian, every nonzero proper ideal is the product
of irreducible ideals (Exercise~\ref{exxx}), so this means that
there are irreducible ideals (in the sense of Exercise~\ref{exxx}) that are not
prime. Below we will see that such irreducible ideals must be primary ideals,
and will give examples.
\item
$R$ is not a PID or a UFD (since it is not integrally closed).
\item
Cancellation fails: there are nonzero ideals $I_1 \ne I_2$ and $J$ such that $I_1 J = I_2 J$.
\end{enumerate}

Let $\mathfrak p$ be a singular prime of a Noetherian 1-domain.
By the results developed above (especially in the last part of Section~\ref{ch2}) we have the
following:

\begin{enumerate}
\item The ring $R_{\frak p}$ has ideals that are not principal, including its maximal ideal.
All such non-principal ideals are not invertible.

\item The ring $R_{\frak p}$ is local, but not strongly local. 
In particular, there are nonmaximal ideals not contained in $\left( \mathfrak p R_{\mathfrak p} \right)^2$.
We will give examples below.

\item
If $\frak m = \frak p R_{\frak p}$ is the maximal
ideal of $R_{\frak p}$, then $\frak m/ \frak m^2$ has
dimension greater than one over the field $R/\frak p \cong R_{\frak p}/\frak p R_{\frak p}$.

\item The ring $R_{\frak p}$ is not integrally closed in $K$.
\item
There are elements $x \in K^\times$ such that neither $x$ and $x^{-1}$ are in~$R_{\frak p}$.

\item
There are irreducible ideals of $R_{\frak p}$ (in the sense of Exercise~\ref{exxx}) that are not
prime ideals because factorization into nonzero prime ideals fails.
We will give examples below.

\item
$R_{\mathfrak p}$ is not a UFD (since it is not integrally closed).

\item
There are nonzero ideals  $J \subseteq I$ where~$I \mid J$ fails.

\item
Cancellation fails: there are nonzero ideals $I_1 \ne I_2$ and $J$ such that $I_1 J = I_2 J$.
\end{enumerate}

In addition, the fact that $R_{\mathfrak p}$ is not integrally closed can be seen
using elements (or even generators) of the fractional ideal $(\mathfrak p R_{\mathfrak p})^{-1}$.

\begin{proposition}
Let $R$ be a Noetherian domain with unique nonzero prime ideal~$\mathfrak p$.
Suppose that $R$ is not integrally closed. Then there 
are integral elements~$x \not\in R$ in the fractional ideal~$\mathfrak p^{-1}$.
Moreover, we can find such an $x$ in any set of generators of the fractional
ideal~$\mathfrak p^{-1}$.
\end{proposition}

\begin{proof}
Let $K$ be the fraction field of~$R$.
By Lemma~\ref{lemma38}, there are~$x \in K^\times$ not in~$R$
such that~$x \mathfrak p \subseteq R$.  Such $x$ is in~$\mathfrak p^{-1}$.
As mentioned above $\mathfrak p$ is not invertible.
So, by Proposition~\ref{prop35}, $\mathfrak p^{-1} \subseteq \mathcal R(\mathfrak p)$, and so 
every such $x$ is integral over~$R$.

Given a set of generators of $\mathfrak p^{-1}$, at least one is not in~$R$. Otherwise 
Lemma~\ref{lemma38} would fail.
\end{proof}

Even for singular Noetherian 1-domains, there is a factorization theorem.
To explain how it works, we begin with the theory of localization which contains
the result that $I \mapsto I R_{\frak p}$
is a surjection from the set of ideal of $R$ to the set of ideals of~$R_{\frak p}$.
In fact, if $J$ is an ideal of $R_{\frak p}$, then $J \cap R$ is an ideal of $R$ that will map to $J$
under this mapping:~$(J \cap R) R_{\frak p} = J$. 
In fact, $J \cap R$ is the maximum among ideals  of~$R$ that map
to $J$ under the mapping $I \mapsto I R_{\frak p}$.

Our factorization will make use of
 ideals of $R$ the form $J\cap R$ where $J$ is an ideal of $R_{\frak p}$.
 One desirable
property of such ideals is the property of being $\frak p$-primary:

\begin{definition}
Let $R$ be a 1-domain
and let $I$ be a nonzero ideal of $R$. We say that $I$ is \emph{primary}
if $I$ is contained in exactly one prime ideal.
If $I$ is a primary ideal contained in the prime ideal $\frak p$ then we say
that $I$ is \emph{$\frak p$-primary}.\footnote{The zero idea $0$ of an integral domain is
considered to be $0$-primary, but we are only interested in nonzero primary ideals here. 
There is a more elaborate definition of primary ideal and $\frak p$-primary ideal for general commutative rings
which we will not mention here. We just note that 
the general definitions of these concepts are equivalent to the definitions given here 
in the special case of~1-domains.}
\end{definition}

We now state  a convenient characterization of $\mathfrak p$-primary. It is a special
case of a well-know result from commutative algebra.\footnote{The result
states that the radical of an ideal $I$ is the intersection of prime ideals containing~$I$.}
We give a short
proof here for the convenience of the reader.

\begin{proposition}\label{criterion}
Let $R$ be a $1$-domain.
Suppose $I \subseteq \mathfrak p$ where $I$ is a nonzero ideal and 
where $\mathfrak p$ is a prime ideal. Then $I$ is $\frak p$-primary
if and only if for each $a \in \frak p$ there is a $k\in \bN$ such that $a^k \in I$.
\end{proposition}

\begin{proof}
Suppose that $I$ is $\frak p$-primary, and that $a \in \frak p$. We have what we want when~$a=0$,
so suppose $a$ is not zero. Consider the multiplicative system $S$ consisting of powers $a^k$ where
$k\ge 0$.

We start with the following claim: the ideal $S^{-1} I$ is all of $S^{-1} R$.
If not, then~$S^{-1} I$ would have to be contained in a maximal ideal, which, by the theory of localization,
has form~$S^{-1} \frak q$ where $\frak q$ is a prime ideal of~$R$ not intersecting~$S$.
In particular, $\frak q$ is not~$\frak p$.
Also
$$I \subseteq (S^{-1} \frak q) \cap R = \frak q,$$
contradicting the definition of $\frak p$-primary.
Thus the claim is established.

So $1 \in S^{-1} I$, which means $1 = b/a^k$ for some $b\in I$ and $k\ge 0$. Thus $a^k = b \in I$.

Suppose, conversely, that $I$ is not $\frak p$-primary. Then there is another prime ideal~$\frak q$
such that $I \subseteq \frak q$. Let $a \in \frak p \smallsetminus \frak q$. Then $a^k$ cannot
be in~$\frak q$, and so cannot be in~$I$.
\end{proof}

\begin{example}
Let $R$ be a 1-domain and let $\frak p$ be a nonzero prime ideal of~$R$. 
Then~$\frak p^k$ is a $\frak p$-primary
ideal for all~$k\ge 1$ by the above proposition. 
Also, any ideal $I$ with~$\frak p^k \subseteq I \subseteq \frak p$
is $\frak p$-primary.
\end{example}

\begin{exercise}
Let $R$ be a 1-domain and let $\frak p$ be a nonzero prime ideal of~$R$. 
Show that 
the product of two~$\frak p$-primary ideals is $\frak p$-primary.
\end{exercise}

\begin{proposition}
Let $R$ be a 1-domain and let $\frak p$ be a nonzero prime ideal of~$R$. 
If $J$ is a nonzero proper ideal of~$R_{\frak p}$ then~$J \cap R$ is a $\frak p$-primary ideal
of $R$.
\end{proposition}

\begin{proof}
We have $J \cap R \subseteq (\frak p R_{\frak p}) \cap R= \frak p$, so we can use the
above criterion (Proposition~\ref{criterion}). Suppose $a\in \frak p$, then $a \in \frak p R_{\frak p}$.
Clearly all proper nonzero ideals of $R_{\frak p}$ are~$\mathfrak p R_{\frak p}$-primary, so~$a^k \in J$
for some $k\in \bN$. Hence $a^k \in J \cap R$.
\end{proof}

\begin{proposition}
Let $R$ be a 1-domain and let $\frak p$ be a nonzero prime ideal of~$R$. 
If~$I$ is a $\frak p$-primary ideal of $R$ and if $\frak q$ is a prime ideal not equal to~$\frak p$
then $I R_{\frak q} = R_{\frak q}$.
\end{proposition}

\begin{proof}
Since $I$ is not contained in~$\frak q$, there is an element $a \in I$ not in~$\frak q$.
Observe that $a$ is a unit in~$R_{\frak q}$.
\end{proof}

\begin{proposition}\label{prop114}
Let $R$ be a 1-domain and let $\frak p$ be a nonzero prime ideal of~$R$. 
Then the map $I \mapsto I R_{\frak p}$ is an order preserving bijection from the set of $\frak p$-primary ideals of $R$
to the set of nonzero proper ideals of $R_{\frak p}$.
The inverse map is $J \to J\cap R$.
\end{proposition}

\begin{proof}
Given $J$ a nonzero proper ideal of $R_{\mathfrak p}$, we know that $J \cap R$
is a $\mathfrak p$-primary ideal of~$R$ that maps to $J$. So the map is surjective.

Suppose that~$I$ and $I'$ 
are $\frak p$-primary ideals of $R$ such that $I R_{\frak p} = I' R_{\frak p}$.
For any maximal ideal~$\frak q$ not equal to $\frak p$, we have 
$I R_{\frak q} = R_{\frak q} = I' R_{\frak q}$ by the previous proposition. 
Thus, by Corollary~\ref{equality_test_cor},
$I = I'$. So the map is injective.
\end{proof}

\begin{exercise}
Let $R$ be a Noetherian 1-domain and let $\frak p$ be a nonzero prime ideal of~$R$. 
Show that $I$ is $\frak p$-primary if and only if
there is a $k \ge 1$ such that 
$$
\frak p^k \subseteq I \subseteq \frak p.
$$
\end{exercise}

\begin{exercise}
Let $R$ be a 1-integral domain.
Let $I$ be a $\mathfrak p$-primary ideal where $\mathfrak p$ is a maximal ideal of~$R$.
Show that there is a ring isomorphism
$$
R/I \to R_{\mathfrak p} / I R_{\mathfrak p}.
$$
Hint: $(I R_{\mathfrak p}) \cap R = I$. If $s \in R\smallsetminus {\mathfrak p}$, what
is $s R + I$?
\end{exercise}

We are ready for the main factorization theorem.
Recall that if $R$ is a Noetherian~1-domain 
then the number of prime ideals containing any given nonzero ideal is finite
(see Lemma~\ref{lemma71}).

\begin{theorem}
Let $R$ be a Noetherian 1-domain.
Then every nonzero ideal $I$ can be written as the product of primary ideals.
More precisely, let $\frak p_1, \ldots, \frak p_k$ be the distinct prime ideals containing~$I$.
Then
$$
I = I_1 I_2 \cdots I_k
$$
where $I_i$ is $\frak p_i$-primary.
Moreover, this product is unique.
\end{theorem}

\begin{proof}
For each $\frak p_i$ let $I_i = I R_{{\frak p}_i} \cap R$.
As we have seen, $I_i$ is ${\frak p}_i$-primary.
Now define~$I' = I_1 I_2 \cdots I_k$.
Observe that, for each $i$,
$$
I' R_{{\frak p}_i} = 
(I_1 R_{{\frak p}_i} )\cdots (I_k R_{{\frak p}_i})
=
I_i R_{{\frak p}_i} = I R_{{\frak p}_i}.
$$
For any other nonzero prime ideal $\frak q$ of $R$ we have
$$
I' R_{\frak q} = R_{\frak q}= I R_{\frak q}.
$$
Thus, by Corollary~\ref{equality_test_cor}, $I' = I$.
This establishes existence.

Now suppose $I = I'_1 I'_2 \cdots I'_k$ where $I_i'$ is ${\frak p}_i$-primary.
For each $i$
$$
I R_{{\frak p}_i} = (I_1 R_{{\frak p}_i} )\cdots (I_k R_{{\frak p}_i}) = I_i R_{{\frak p}_i} 
\qquad
I R_{{\frak p}_i} = (I'_1 R_{{\frak p}_i} )\cdots (I'_k R_{{\frak p}_i}) = I'_i R_{{\frak p}_i}.
$$
By the  injectivity of the ideal map on the set of $\frak p_i$-primary ideals, we
deduce from~$ I_i R_{{\frak p}_i} = I'_i R_{{\frak p}_i}$ the desired conclusion: $I_i = I'_i$.
\end{proof}

\begin{exercise}
Let $R$ be a Noetherian 1-domain.
Show that every irreducible ideal must be a primary ideal.
(An irreducible ideal is a nonzero proper ideal that is not the product of two proper ideals.)

Let $I$ be a $\mathfrak p$-primary ideal.
Show  $I$ is an irreducible ideal of~$R$
if and only if~$I R_{\mathfrak p}$ is an irreducible ideal of~$R_{\mathfrak p}$.
Conclude a $\mathfrak p$-primary ideal $I$ of~$R$ is an irreducible non-prime ideal if and only if 
$I R_{\mathfrak p}$ is an irreducible non-prime ideal.
\end{exercise}

Let $R$ be a singular Noetherian 1-domain.
Then there must be at least one irreducible non-prime ideal $I$.
This is because factorization into nonzero prime ideals fails for some nonzero ideals, but
factorization into irreducible ideals holds (Exercise~\ref{exxx}).
The above exercise shows that such an irreducible $I$ must be a $\mathfrak p$-primary ideal
for some prime ideal $\mathfrak p$.
Also, $I R_{\mathfrak p}$ must be an irreducible non-prime ideal by the above exercise,
and so $R_{\mathfrak p}$ is not a DVR. In other words, 
$\mathfrak p$ is singular.
So any  irreducible non-prime ideal is $\mathfrak p$-primary for
some singular prime ideal~$\mathfrak p$.

Conversely, if ${\mathfrak p}$ is singular, $R_{\mathfrak p}$ must have an irreducible
non-prime ideal $J$. There is a unique $\mathfrak p$-primary ideal $I$
of $R$ with $I R_{\mathfrak p} = J$, and by the above exercise this $I$ an irreducible
non-prime ideal.

So when we look for irreducible non-prime ideals,  we can focus on
the local situation at a singular prime. So from now on we will limit our attention to
singular  local 
Noetherian~$1$-domains, i.e., Noetherian domains with exactly one prime ideal that are not DVRs. The next results shows that we can find irreducible
non-prime ideals among the principal ideals:

\begin{proposition}
Let $R$ be a singular local Noetherian $1$-domain.
Let $a \in R$ be a nonzero non-unit element.
Then $a R$ is the product of irreducible non-prime ideals.
\end{proposition}

\begin{proof}
Factor $a R = I_1 \cdots I_k$ where each $I_i$ is irreducible (Exercise~\ref{exxx}).
Since $aR$ is invertible, each $I_i$ must also be invertible.
Thus each $I_i$ is principal (Proposition~\ref{propInv}) and can be written as $a_i R$
for some $a_i \in R$. Since $R$ is singular, its maximal ideal is not principal.
Thus each $a_i R$ is an irreducible non-prime ideal.
\end{proof}

We can identify another source of irreducible non-prime ideals (there may be overlap
between our two categories).

\begin{proposition}
Let $R$ be a singular local Noetherian $1$-domain with maximal ideal~$\mathfrak m$.
If $I$ is an ideal such that $\mathfrak m^2 \subsetneq I \subsetneq \mathfrak m$
then $I$ is an irreducible non-prime ideal.
\end{proposition}

\begin{proof}
Suppose such an $I$ factors as proper ideals: $I = J_1 J_2$. 
Since $J_i \subseteq \mathfrak m$, the product $I = J_1 J_2$
is contained in~$\mathfrak m^2$,
a contradiction.
\end{proof}

Are there ideals $I$ such that  $\mathfrak m^2 \subsetneq I \subsetneq \mathfrak m$? The answer is that there
are a lot of them when $R$ is singular. Assume $R$ is a singular local Noetherian 1-domain
with maximal ideal~$\mathfrak m$.
Consider $\mathfrak m / \mathfrak m^2$ which is a vector space over the
residue field~$R/\mathfrak m$ (Exercise~\ref{ex18}). Since $\mathfrak m$ is a finitely 
generated, the vector space $\mathfrak m / \mathfrak m^2$
is finite dimensional. Note that $\mathfrak m  \ne \mathfrak m^2$
by Exercise~\ref{exA}, so $\mathfrak m / \mathfrak m^2$ has positive dimension.
Since $R$ is singular, the dimension cannot be one (Exercise~\ref{dimension_ex}).
So~$\mathfrak m / \mathfrak m^2$
has dimension at least two. Such vector spaces have multiple proper nonzero subspaces.

There is a one-to-one correspondence between ideals $I$ with $\mathfrak m^2\subseteq I \subseteq  \mathfrak m$
and subspaces of the vector space $\mathfrak m / \mathfrak m^2$. (See the next appendix for
the relationship between ideals of $R$ containing $\mathfrak m^2$ and the ideals
of $R / \mathfrak m^2$. Note that every vector  subspace of $\mathfrak m / \mathfrak m^2$
is an ideal of $R / \mathfrak m^2$.) So we get the following:

\begin{proposition}
There are multiple ideals $I$ such that $\mathfrak m^2\subsetneq I \subsetneq  \mathfrak m$.
These are in one-to-one correspondence with the proper, nonzero subspaces of $\mathfrak m^2\subseteq I \subseteq  \mathfrak m$.
\end{proposition}

So this gives another source of irreducible non-prime ideals in the local singular case, and hence in the
general singular case.


\chapter*{Appendix C: A note on ideals in quotient groups}\label{Appendix_C}

We have frequently used the correspondence between ideals of 
an integral domain~$R$ and ideals of a localization $S^{-1} R$. Prime ideals are
well-behaved under this correspondence. This correspondence is also well-behaved
with respect to products (and other operations). In fact the map $I \mapsto S^{-1} I$ yields
a surjective homomorphism
from the monoid of ideals of $R$ to the monoid of ideals of $S^{-1} R$.
We extended this surjection to fractional ideals.

There is a similar correspondence between ideals of a ring $R$ and
ideals of a quotient ring~$R/I$. However, this situation is trickier since
this correspondence is really two levels of correspondence where each
level has properties that the other does not.
This discrepancy between the two levels actually gives us a method
of finding irreducible ideals that are not prime ideals, which we will use later.
For simplicity, we will stick to commutative rings, and
we will not attempt to extend the correspondence to fractional ideals.

The purpose of this appendix is
introduce the ideal correspondence for quotient rings, which is of key importance
in commutative algebra, and  
 to prepare the groundwork for the next appendix
on prime ideal factorization.

So let $R$ be a commutative ring and let $I$ be an ideal. As mentioned above,
there are two levels to the correspondence between ideals of $R$ and ideals of $R/I$. 
At first we will consider only ideals $J$ of $R$ that contain $I$. In this case, the 
Abelian group $J/I$ is actually a subset of $R/I$: for each $a\in J$, the coset $a + I$  as an element of~$J/I$
is also a coset appearing as an element of~$R/I$. We will often write this coset as~$[a]$.
Checking that $J/I$ is an ideal of $R/I$ is straightforward, as are most of the claims of the following:

\begin{proposition}\label{prop116}
Let $I$ be an ideal of a commutative ring~$R$. Then the 
natural map~$J \mapsto J/I$ is an inclusion preserving bijection
from the set of ideals containing~$I$ to the set of ideals of~$R/I$.
Restricting this bijection to prime ideals yields a bijection
from the set of prime ideals containing $I$ to the set of prime ideal of~$R/I$.

In particular, if $R$ is a local ring with maximal ideal~$\mathfrak m$, and if $I$
is a proper ideal of~$R$, then $R/I$ is a local ring with maximal ideal~$\mathfrak m/ I$.
\end{proposition}

\begin{proof}
The claims are straightforward. For surjectivity, given $\tilde J$ an ideal of $R/I$
consider the following ideal of~$R$:
$$J = \left\{ a \in R \;\; \middle| \;\;  [a] \in \tilde J \;\right\}.$$
\end{proof}

This bijection is compatible with ideal operations:

\begin{proposition}\label{prop117}
Let $I$ be an ideal of a commutative ring~$R$. Then the 
natural bijection~$J \mapsto J/I$ respects the operations
of addition and intersection:
$$
(J_1 + J_2) / I = (J_1/I) + (J_2/I),\qquad
(J_1 \cap J_2) / I = (J_1/I) \cap (J_2/I).
$$
This bijection sends any principal ideal $a R$ containing $I$ to the principal ideal~$[a] (R/I)$.
Finally, if $J_1, J_2,$ and $J_1 J_2$ all contain $I$ then 
$$
(J_1 J_2) / I = (J_1/I)  (J_2/I).
$$
\end{proposition}

\begin{proof}
This is straightforward.
\end{proof}

A major limitation of the bijection $J \mapsto J/I$  is that it is not in general a monoid homomorphism
with respect to products for the simple reason that the domain is not in general closed under
products: just because~$J_1$ and $J_2$ contain $I$ does not guarantee that $J_1 J_2$ will contain~$I$.
(Addition fares better with respect to closure, but note that the additive identity $\{ 0 \}$ is not in general in the domain)

We can fix this problem by expanding this correspondence.
This second level of correspondence is based on the following idea:
Given a homomorphism~$\varphi: R_1 \to R_2$ between rings, if $J$ is an ideal of~$R_1$
we define $\varphi[J]$ to be the image of $J$ under this map.
If $\varphi$ is surjective, then $\varphi[ J ]$ is seen to be an ideal of $R_2$.

\begin{proposition}\label{prop118}
Let $I$ be an ideal of a commutative ring~$R$ and let $\varphi\colon R \to R/I$
be the canonical homomorphism $a \mapsto [a]$. Then the map~$J \mapsto \varphi[J]$
is an order preserving surjective function from the set of ideals of $R$ to the set of ideals of~$R/I$.
This map extends the bijection $J\to J/I$ defined above which was defined only when~$I \subseteq J$.

Given ideals $J_1, J_2$ of $R$, we have 
$$\varphi[J_1 J_2] = \varphi[J_1] \varphi[J_2], \quad \varphi[J_1+ J_2] = \varphi[J_1] + \varphi[J_2].$$
In fact,  $J \mapsto \varphi[J]$ is a surjective homomorphism from the monoid of ideals of $R$ under products
to the monoid of ideals of $R/I$ under products. Similarly for the monoids under addition.
 
In addition, this map sends any principal ideal $a R$ to the principal ideal~$[a] (R/I)$, and
yields a surjection from the monoid of principal ideals of $R$ (under multiplication)
to the monoid of principal ideals of $R/I$.
\end{proposition}

\begin{proof}
This is straightforward. 
\end{proof}
Since the second level correspondence is not in general injective, we should address the kernel
and the issue of injectivity: 

\begin{proposition}\label{prop119}
Let $I$ be an ideal of a commutative ring~$R$ and let $\varphi\colon R \to R/I$
be the canonical homomorphism $a \mapsto [a]$.
Under the map $J \mapsto \varphi[J]$, an ideal $J$ 
maps to zero if and only if $J \subseteq I$.
Two ideals $J_1$ and $J_2$ map to the same ideal in~$R/I$ if and only if $J_1 + I = J_2 + I$.

If $J$ is an ideal of $R$,
then $I + J$ is the unique ideal $J'$ of $R$ containing $I$ such that $\varphi[J] = \varphi[J']$.
\end{proposition}

So there are two levels of the correspondence; the first is bijective and the second is only surjective.
For some purposes the first correspondence works better, for other purposes the second works better.
So both are of use.
For example, the first has the advantage that the correspondence is bijective and sends primes ideals
to prime ideals. The second has the advantage that it is a monoid homomorphism for ideal multiplication
and ideal addition. The second is a surjection for principal ideals, while the
first is an injection. The first works better for intersections of ideals. 

We will now illustrate how having this difference in properties, especially with respect to products, 
can be exploited to produce ideals that are irreducible
but not prime. (Recall that an ideal is irreducible if it is nonzero and proper, and it cannot be written as the product
of two proper ideals. In the previous appendix we gave examples of rings with nonprime
irreducible ideals, but these integral domains were not integrally closed. In this appendix we will
give examples that include integrally closed integral domains).

Recall that in a Dedekind domain every nonzero ideal is the product of prime ideals, so there
is no nonprime irreducible ideals. 
Since every PID is a Dedekind domain, this property also holds for PIDs as well.
Does this happen for most nice integral domains?
Consider the polynomial ring $R = F[X, Y]$ in two variables where $F$ is a field. 
We won't prove it here, but this ring is a UFD and is a fairly well-behaved ring. 
In particular it is integrally closed.
Surprisingly, this ring has nonprime irreducible ideals, and we
will prove this. It will be a bit
of work, but the construction nicely shows off a lot of the techniques we
are interested in.
The proof also works in more generality than this particular example, and will
work in the context of the rings in the following appendix.

Observe that in our example~$R = F[X, Y]$, the 
quotient~$R/\left< X \right>$  is isomorphic to~$F[Y]$ which is a~PID, so 
the ideal~$\left< X \right>$ is prime. 
These are key properties for our construction.
More generally, let $R$ be any integral domain with a nonzero prime ideal $\mathfrak p_1$
such that~$R/\mathfrak p_1$ is a PID, or even a Dedekind domain.
Assume that $\mathfrak p_2$ is a prime ideal such that $\mathfrak p_1 \subsetneq \mathfrak p_2$, 
but
that $\mathfrak p_1$ is not contained in ${\mathfrak p}_2^k$ for some~$k \ge 2$, and
assume $k$ is minimal with this property.
For example, if $R = F[X, Y]$ we could choose~$\mathfrak p_1 =X R$, $\mathfrak p_2 = X R + Y R$.
We check that $\mathfrak p_1$ is not contained in ${\mathfrak p}_2^2$ in this case.
In this example, $\mathfrak p_2$ is not only prime but is maximal
since $R/\mathfrak p_2$ is isomorphic to the field~$F$.

Given this general set-up,
consider the homomorphism $\varphi: R \to R/\mathfrak p_1$. 
Let~$\mathfrak p'_2$ be the nonzero prime ideal of $R/\mathfrak p_1$
corresponding to~$\mathfrak p_2$.
In other words $\mathfrak p'_2$ is $\varphi[\mathfrak p_2]$, or equivalently $\mathfrak p_2/\mathfrak p_1$.
Let $J'$ be $(\mathfrak p'_2)^k$ where~$k\ge 2$ is as above.
By Proposition~\ref{prop116} there is a unique ideal $J$ of $R$ containing~$\mathfrak p_1$
which maps to $J'$, and $J$ is a nonzero proper ideal of~$R$.
Because of the problems of the correspondence of Proposition~\ref{prop116}
mentioned above involving products of ideals, 
$J$ might not be $\mathfrak p_2^k$
as you might expect. On the contrary, $J$ will turn out to be irreducible.

\begin{lemma}
The ideal $J$ constructed above is an irreducible ideal in $R$, but is not a prime ideal.
\end{lemma}

\begin{proof}
First observe that $J$ is not a prime ideal since $J'$ is not a prime ideal
in the Dedekind domain $R/\mathfrak p_1$ (see Proposition~\ref{prop116}).

Suppose $J$ is reducible in the sense that $J = I_1 I_2$ where $I_1, I_2$ are proper ideals. 
Observe $\mathfrak p_1 \subseteq J  = I_1 I_2 \subseteq I_i$.
Let $I'_1$ and $I'_2$ be the respective images in $R/\mathfrak p_1$.
By the homomorphism property (Proposition~\ref{prop118}), $I_1' I_2' = J' = ({\mathfrak p}'_2)^k$.
We are assuming that $R/\mathfrak p_1$ is a Dedekind domain,
so $I_1' =( {\mathfrak p}'_2)^{l_1}$ and $I_2' = ({\mathfrak p}'_2)^{l_2}$
where $l_1 + l_2 = k$ and $l_1, l_2 > 0$ (since $I'_1, I'_2$ are proper ideals).

By the injectivity of the correspondence (Proposition~\ref{prop116}), we have $I_1 = ( {\mathfrak p}_2)^{l_1}$
and $I_2 = ({\mathfrak p}_2)^{l_2}$. So~$J= ( {\mathfrak p}_2)^k$.
In particular, $\mathfrak p_1 \subseteq J = {\mathfrak p}_2^k$ contradicting the choice of~$k$.
\end{proof}

\begin{remark}
So $J$ is not $\mathfrak p_2^k$, but they map to the same ideal in $R/\mathfrak p_1$.
Thus, by Proposition~\ref{prop119} we have
$J = \mathfrak p_1 + \mathfrak p_2^k$ since $J$ contains~$\mathfrak p_1$.
\end{remark}

\begin{example}
In our original example with $R = F[X, Y]$, the irreducible
ideal produced in our construction is $J = X R + (X R + Y R)^2$
which can be written as~$X R + Y^2 R$ or $\left< X, Y^2 \right>$.
\end{example}

We summarize our construction:

\begin{theorem}\label{thm121}
Suppose $R$ is an integral domain with nonzero prime ideals $\mathfrak p_1 \subsetneq \mathfrak p_2$.
Suppose that $\mathfrak p_1$ is not contained in $\mathfrak p_2^k$ for some $k\ge 2$.
Finally, suppose $R/\mathfrak p_1$ is a Dedekind domain.
Then $R$ has an irreducible ideal that is not a prime ideal.
\end{theorem}

\begin{corollary}
If $F$ is a field, then the polynomial ring $F[X, Y]$ has an irreducible ideal that is not a prime ideal.
\end{corollary}

\begin{exercise}
Show that the polynomial ring $F[X_1, \ldots, X_n]$ has an irreducible ideal that is not a prime ideal
where $n\ge 2$ and where $F$ is a field.
\end{exercise}

\begin{exercise} 
Suppose $R$ is a PID with ideal $I$.
Use Proposition~\ref{prop118} to show
that every ideal of~$R/I$ is principal.
\end{exercise}

\begin{exercise} 
Extend the above exercise to Dedekind domains. In other words,
show that if $I$ is a nonzero ideal of a Dedekind domain~$R$, then every ideal of~$R/I$ is principal.

Hint: Let $S$ consist of all $s \in R$ not in any prime ideal containing~$I$.
Is $S$ a multiplicative system? Why is $S^{-1} R$ a PID?
You can also take the following standard result as given:
If $S$ is a multiplicative system of~$R$ disjoint from
every maximal ideal containing~$I$, 
then we have a natural isomorphism
$$
R/I \, \cong \, S^{-1} R / S^{-1} I.
$$
\end{exercise}


\chapter*{Appendix D: A note on prime ideal factorization}

Above we proved that 
if $R$ is an integral domain such that every nonzero proper ideal is the product
of maximal ideals then $R$ is a Dedekind domain (Theorem~\ref{thm101}).
What if we have an integral domain with the weaker property that every nonzero proper ideal
is the product of prime ideals? It turns out that this is enough to give a Dedekind domain.
In other words, we can 
strengthen Theorem~\ref{thm101} by changing the hypothesis where we replace products of maximal ideals
with products of prime ideals, but it will take us some effort even using local methods.\footnote{Paulo Ribenboim attributes
this stronger version of Theorem~\ref{thm101} to Matusita. He mentions this in Section 7.1 of his
\emph{Classical Theory of Algebraic Numbers} (Springer 2001). In this section 
Ribenboim gives a short, but tricky
proof. It is a non-local proof that uses 
the quotient ring correspondence (our Proposition~\ref{prop116}).}
Fortunately some of the heavy lifting was done in a previous appendix in the proof of Theorem~\ref{thm121}.

To prove the stronger version of Theorem~\ref{thm101} we will switch between 
integral domains sharing the key property. So it will be useful to
label the property:

\begin{definition}
An integral domain is said to have the \emph{prime ideal factorization}~(PIF) property
if every nonzero proper ideal is the product of prime ideals.
\end{definition}

Our goal, then, is to show that any integral domain with the PIF property is 
a Dedekind domain. 
As we have seen, it is often easiest to work first with local integral domains, and leverage
the results to arbitrary integral domains. When we work locally in the context
of the PIF property, we get a UFD:

\begin{lemma} \label{UFD_condition}
Any local integral domain $R$ with the PIF property is a UFD.
\end{lemma}

\begin{proof}
Given a nonzero non-unit $a \in R$, the ideal $a R$ factors as the product
of prime ideals: $a R = \mathfrak p_1 \cdots \mathfrak p_k$.
By Proposition~\ref{easy_invert_prop}, each of these prime ideals
is invertible. By Proposition~\ref{propInv}, each of these invertible prime ideals
is actually principal.  This means that~$a = \pi_1 \cdots \pi_k$ where each $\pi_i$
is a prime element. (Define a prime element~$\pi$ to be a nonzero element
such that $\pi R$ is a prime ideal).

So  every nonzero non-unit element factors as the product of prime
elements. 
It is an easy exercise to show that (1) each prime element is irreducible,
and (2) if a nonzero non-unit element has such a prime factorization, then 
it has a unique factorization into irreducible elements.
We conclude that $R$ is a UFD.
(See Exercise~\ref{ufd_ex} for the definition of UFD, irreducible, and the meaning of unique here).
\end{proof}

Using what we know about localization and quotient rings, it is straightforward to prove the following.

\begin{lemma}\label{lemma116}
If $R$ is an integral domain with the PIF property, then so is $S^{-1} R$ for any multiplicative
system $S$ of~$R$.
\end{lemma}

We have a similar result for quotients $R/\mathfrak p$. 

\begin{lemma}\label{lemma125}
If $R$ is an integral domain with the PIF property, then so is $R/\mathfrak p$
for any prime ideal $\mathfrak p$ of $R$. 
\end{lemma}

\begin{proof}
Let $J'$ be a nonzero proper ideal of $R/\mathfrak p$.
Let $J$ be the ideal of $R$ containing~$\mathfrak p$
that maps to~$J'$ under the bijection in Proposition~\ref{prop116}.
Since $J'$ is nonzero, it follows that $\mathfrak p \subsetneq J$. Also $J$ must be 
a proper ideal. So use the PIF property to factor $J$ into prime ideals
of~$R$:
$$
J = \mathfrak p_1 \cdots \mathfrak p_k.
$$
Since $J \subseteq \mathfrak p_i$, it follows that $\mathfrak p \subsetneq \mathfrak p_i$.
This implies that $\mathfrak p_i'$ is a nonzero prime ideal of $R/\mathfrak p$
where $\mathfrak p_i'$ is the image (i.e., $\mathfrak p_i/\mathfrak p$) of $\mathfrak p_i$ under the canonical map
(see Proposition~\ref{prop116}).
By the homomorphism property (Proposition~\ref{prop118}),
$$
J' = \mathfrak p_1'
 \cdots \mathfrak p_k'.
$$
\end{proof}

\begin{lemma}\label{lemmaD_minimal}
If $R$ is a UFD then every nonzero principal prime ideal $\pi R$
is a minimal nonzero prime ideal of $R$.
\end{lemma}

\begin{proof}
Let $\mathfrak p$ be any nonzero prime ideal contained in the prime ideal~$\pi R$, and let~$a\in \mathfrak p$
be a nonzero element. 
Write~$a = \pi_1 \cdots \pi_k$ where  $\pi_i$ are irreducible.
Since~$a \in \mathfrak p$, there is an $i$ such that~$\pi_i \in \mathfrak p$.
So
$$
\pi_i R \subseteq \mathfrak p \subseteq \pi R.
$$
Since $\pi_i$ is irreducible, $\pi_i$ and $\pi$ are associates, so $\pi_i R = \mathfrak p = \pi R$.
\end{proof}

Our intermediate goal is to show that a local integral domain
with the PIF property is a DVR. The next
lemma shows that we just need to show the maximal ideal is principal.

\begin{lemma}\label{lemma126}
Let $R$ be a local integral domain with the PIF property.
If the maximal ideal~of~$R$ is nonzero and principal then $R$ is a DVR.
\end{lemma}

\begin{proof}
Let $\pi R$ be the maximal ideal of $R$.
By Lemma~\ref{UFD_condition},~$R$ is a UFD.
By the above lemma (Lemma~\ref{lemmaD_minimal}),  $\pi R$ 
is the unique nonzero prime ideal of~$R$.
Since $R$ has the PIF property, every nonzero proper ideal is a power of $\pi R$. 
This means that every ideal of $R$ is principal. By Corollary~\ref{cor11}, $R$ is a DVR.
\end{proof}

\begin{proposition}\label{tricky_thm}
Let $R$ be a local integral domain with the PIF property that is not a field.
Then~$R$ is a DVR.
\end{proposition}

\begin{proof}
Let $\mathfrak m$ be the maximal ideal of~$R$. If $\mathfrak m$
is principal, we are done by the previous lemma. So we will
assume $\mathfrak m$ is not principal and derive a contradiction.

By Lemma~\ref{UFD_condition},  $R$ is a UFD, and in a UFD every irreducible element is 
a prime element.
Let $a \in \mathfrak m$ be nonzero and let $\pi_1$ be an irreducible factor of~$a$.
So $\pi_1 R$ is a nonzero prime idea, call it $\mathfrak p_1$. Observe that 
$\mathfrak p_1\subsetneq \mathfrak m$ since $\mathfrak m$ is not principal.

By Lemma~\ref{lemma125}, the quotient $R/\mathfrak p_1$ has the PIF property, 
and so is a UFD by Lemma~\ref{UFD_condition}.
The image $\tilde{ \mathfrak m}$ of~$\mathfrak m$ in $R/\mathfrak p_1$ is a nonzero prime ideal 
since~$\mathfrak p_1\subsetneq \mathfrak m$.
Let $[b] \in \tilde{ \mathfrak m}$ be nonzero where $b\in R$. 
Let $[\pi_2]$ be an irreducible factor of $[b]$, where~$\pi_2 \in R$.
So $[\pi_2] (R/\mathfrak p_1)$ is a nonzero prime ideal of~$R/\mathfrak p_1$.

The ideal $\pi_1 R + \pi_2 R$ maps to $[\pi_2] (R/\mathfrak p_1)$
under the mapping of Proposition~\ref{prop116}. So $\pi_1 R + \pi_2 R$ is a prime
ideal of $R$, call it $\mathfrak p_2$.

Let $R'$ be $R_{\mathfrak p_2}$ and let $\mathfrak p'_2$ be its maximal ideal.
So $\mathfrak p'_2 = \pi_1 R' + \pi_2 R'$.
Let $\mathfrak p'_1$ be~$\pi_1 R'$.
By the correspondence between prime ideals of $R$ contained in $\mathfrak p_2$ and prime ideals of~$R'$,
the ideal $\mathfrak p'_1$ is a prime ideal and
$\mathfrak p_1' \subsetneq \mathfrak p'_2$.
Also note that $R'$ has the PIF property by Lemma~\ref{lemma116}.

Our goal is to apply Theorem~\ref{thm121} to the ring~$R'$ to help derive a contradiction.
So we wish to show that $R'/\mathfrak p_1'$ is a Dedekind domain.
By Lemma~\ref{lemma125} this ring has the PIF property.
The image of $\mathfrak p'_2$ is its unique maximal ideal.
Observe $\mathfrak p'_2$ is principal with generator $[\pi_2]$
and is nonzero.
So, by Lemma~\ref{lemma126}, the ring $R'/\mathfrak p_1'$ is a DVR,
and so is a Dedekind domain.

To use Theorem~\ref{thm121}, we
also wish to show that $\mathfrak p_1'$ is not contained in $(\mathfrak p'_2)^2$.
Suppose otherwise. Then 
$$
\pi_1 R' \subseteq  (\pi_1 R' + \pi_2 R')^2
 = \pi_1^2 R' + \pi_1 \pi_2 R' + \pi_2^2 R',
$$
and so $\pi_1 = a \pi_1^2 + b \pi_1 \pi_2  + c \pi_2^2$ for some $a, b, c\in R'$.
This equation yields the equation~$[0] = [c] [\pi_2]^2$ in $R'/\mathfrak p_1'$.
But $[\pi_2]$ is nonzero, so $[c]$ is zero. Thus $c = d \pi_1$ for some $d \in R'$.
Dividing both sides of the equation by~$\pi_1$ 
yields~$1 = a \pi_1 + b  \pi_2  + d \pi_2^2$.
This shows $1 \in \mathfrak p'_2$, a contradiction.
Thus we have established that 
$\mathfrak p_1'$ is not contained in $(\mathfrak p'_2)^2$.

Now we can use Theorem~\ref{thm121}
to conclude that $R'$ has an irreducible ideal that is not a prime ideal.
This contradicts the fact that $R'$ has the PIF property.
\end{proof}

\begin{lemma}
If $R$ is an integral domain with the PIF property, then every nonzero prime of $R$ is maximal.
\end{lemma}

\begin{proof}
Let $\mathfrak p$ be a nonzero prime ideal, and let $\mathfrak m$ be a maximal ideal containing~$\mathfrak p$.
By Lemma~\ref{lemma116}, the ring $R_{\mathfrak m}$ also has the PIF property.
By Proposition~\ref{tricky_thm}, the ring $R_{\mathfrak m}$ is a DVR.
Now $\mathfrak p R_{\mathfrak m}$ is a nonzero prime ideal, hence
$
\mathfrak p R_{\mathfrak m} = \mathfrak m R_{\mathfrak m}
$
since DVRs have a unique nonzero prime ideal.
This implies $\mathfrak p = \mathfrak m$.
\end{proof}

We are now ready to state and prove the stronger version of~Theorem~\ref{thm101}.

\begin{theorem}
Suppose that $R$ is an integral domain  such that every nonzero proper ideal factors 
as the product of prime ideals. Then $R$ is a Dedekind domain.
\end{theorem}

\begin{proof}
By the previous lemma, we see that every nonzero proper ideal factors 
as the product of maximal ideals. Then $R$ is a Dedekind domain by Theorem~\ref{thm101}.
\end{proof}


\chapter*{Appendix E: Almost Dedekind domains and Pr\"ufer domains}

Dedekind domains are both Noetherian and integrally closed.
In Appendix B we considered a generalization of Dedekind domains
that are not necessarily integrally closed, but are, however, Noetherian.
In this appendix we consider generalizations of Dedekind domains
that are not necessarily Noetherian, but are, however, integrally closed.

We begin with almost Dedekind domains which we introduced in Section~\ref{ch7}.

\begin{definition}
An \emph{almost Dedekind domain} is an integral domain $R$ with the property that 
$R_{\mathfrak m}$ is a DVR for all nonzero maximal ideals of~$R$.\footnote{I was
tempted to call these \emph{locally Dedekind domains}, but the term \emph{almost
Dedekind domain} is now standard.}
\end{definition}

\begin{example}
We won't construct any non-Noetherian almost Dedekind domains here, but 
they are known to exist. We cite two known examples.

Let $K$ be the algebraic number field generated by $p$th roots of unity for each prime~$p$.
Then the ring of integers in~$K$ is a non-Noetherian almost Dedekind domain.
This was the first non-Noetherian integral domain $R$ identified as having
the property of that $R_{\mathfrak m}$ is a DVR for all nonzero maximal ideals of~$R$
(N. Nakano~1953).

Let $K$ be be the algebraic number field generated by the square roots of $p$ for every prime $p$.
Then the ring of integers in~$K$ is a non-Noetherian almost Dedekind domain.
(C.~Hashbarger 2010)
\end{example}

\begin{remark}
The term
\emph{almost Dedekind domain} was coined by Robert~Gilmer (1964)
who pioneered the study of such rings. It is interesting that non-Noetherian
almost Dedekind domains exists since they are locally Noetherian but not themselves
Noetherian.
\end{remark}

\begin{proposition}
An almost Dedekind domain is integrally closed and has the property that every non-zero prime
ideal is maximal.
Thus an almost Dedekind domain is a Dedekind domain if and only if it is Noetherian.
\end{proposition}

\begin{proof}
Since DVRs are integrally closed and has the property that every non-zero prime
ideal is maximal, it follows that almost Dedekind domains must also have these properties.
See Proposition~\ref{prop56} and Proposition~\ref{prop57}. 
\end{proof}

We know that invertible maximal ideals are finitely generated (Section~\ref{ch3}). In Exercise~\ref{fg_invert_ex} we 
learned that the converse is true for almost Dedekind domains. So we have the following

\begin{proposition}\label{prop162}
A fractional ideal of an 
almost Dedekind domain is invertible if and only if it is finitely generated.
\end{proposition}

Recall that in Section~\ref{ch12} we defined a 
cancellation domain to be an integral domain with the property that if~$I_1 J = I_2 J$,
where $I_1, I_2$ and $J$ are nonzero ideals, then $I_1 = I_2$ 
(Definition~\ref{cancellation_domain_def}).

\begin{lemma}
Every almost Dedekind domain is a cancellation domain.
\end{lemma}

\begin{proof}
Suppose $I_1 J = I_2 J$ where
$I_1, I_2, J$ are nonzero ideals of an almost Dedekind domain $R$.
For each maximal ideal $\mathfrak m$ we have 
$$
(I_1 R_{\mathfrak m}) (J R_{\mathfrak m}) = (I_2 R_{\mathfrak m}) (J R_{\mathfrak m}).
$$
Now every nonzero ideal in a DVR has an inverse. Thus
$(I_1 R_{\mathfrak m}) = (I_2 R_{\mathfrak m})$.
This holds for each maximal ideal $\mathfrak m$, so~$I_1 = I_2$ (Corollary~\ref{equality_test_cor}).
\end{proof}

The surprise here is that the converse of the above lemma holds:

\begin{lemma}
If ${\mathfrak m}$ is a nonzero maximal ideal of a cancellation domain~$R$
then~$R_{\mathfrak m}$ is a DVR.
\end{lemma}

\begin{proof}
By Exercise~\ref{cancel_ex3}, $R_{\mathfrak m}$ is a valuation ring with principal maximal
ideal. Let $\pi$ be a generator of the maximal idea of $R_{\mathfrak m}$.
By clearing denominators if necessary we can choose~$\pi$ in~$R$.
By Exercise~\ref{cancel_ex1},
$$
\bigcap_{k=1}^\infty \pi^k R = \{0\}.
$$
We claim that, furthermore, that
$$
\bigcap_{k=1}^\infty \pi^k R_{\mathfrak m} = \{0\}.
$$
To see this let $a/s$ be in this intersection where $a\in R$ and $s \not\in \mathfrak m$.
From the identity~$\left( \pi^k  R_{\mathfrak m} \right) \cap R = \pi^k R$ (see Exercise~\ref{cancel_ex2a}), we get 
that $a \in \pi^k R$ for all $k\ge 1$. Thus $a$ is in the intersection so~$a = 0$.
Since $a/s = 0$, this establishes the claim.

This implies that every nonzero element of  $R_{\mathfrak m}$
is of the form $u \pi^k$ where $u$ is a unit of~$R_{\mathfrak m}$.
By Corollary~\ref{cor9}, $R_{\mathfrak m}$ is a DVR.
\end{proof}

Combining the previous two lemmas gives us the following;

\begin{theorem}\label{cancellation_almost_dedekind_thm}
Let $R$ be an integral domain. Then $R$ is a cancellation domain if and only
if $R$ is an almost Dedekind domain.
\end{theorem}

We address the issue of ideal factorization  in
an almost Dedekind domain in the next two exercises.

\begin{exercise}
Show that if a nonzero ideal of an almost Dedekind domain
factors  as the product of nonzero prime ideals
then that factorization is unique.
\end{exercise}

\begin{exercise}
Let $I$ be a nonzero proper ideal of an almost Dedekind domain.
Show that $I$ factors as the product of nonzero prime ideals if and only 
if $I$ is contained in only a finite number of maximal ideals. 
Hint: construct a plausible factorization
of $I$ based on images in DVRs and use Corollary~\ref{equality_test_cor}
to verify it.
\end{exercise}

\begin{exercise}
Let $R$ be an almost Dedekind domain. Show that $R$ is a Dedekind domain if and only if
every nonzero element of $R$ is contained in only a finite number of maximal ideals.

Hint: Observe that if $I$ is an ideal and if some element $a \in I$ is only contained in a finite number of
maximal ideals, then $I$ is contained in only a finite number of maximal ideals. Now use the previous exercise.
\end{exercise}

By Proposition~\ref{prop162} every nonzero finitely generated ideal of 
an almost Dedekind domain is invertible. As we saw in Section~\ref{ch3}, this is the best result we can
hope for in a non-Noetherian ring.
This helps motivate the next concept:

\begin{definition}
A \emph{Pr\"ufer domain} is an integral domain such that every nonzero finitely generated
ideal is invertible.
\end{definition}

\begin{remark}
Pr\"ufer domain are named for Heinz Pr\"ufer who introduced such rings in~1932.
\end{remark}

The following three propositions are straightforward given our earlier results.

\begin{proposition}
Every almost Dedekind domain is a Pr\"ufer domain.
\end{proposition}

\begin{proposition}
A Pr\"ufer domain is a Dedekind domain if and only if it is Noetherian.
\end{proposition}

\begin{proposition}
A fractional ideal of a Pr\"ufer domain is invertible if and only if it is finitely generated.
\end{proposition}

Pr\"ufer domains have a cancellation law that is weaker than almost Dedekind domains.
This is expressed with the next concept.

\begin{definition}
A \emph{weak cancellation domain} is an integral domain with the property that
for all  nonzero ideals $I_1$ and $I_2$ and for all  finitely generated nonzero ideals $J$
 if $I_1 J = I_2 J$ then $I_1 = I_2$.
\end{definition}

From the definition of Pr\"ufer domain we have the following:

\begin{proposition}\label{prop169}
Every Pr\"ufer domain is a weak cancellation domain.
\end{proposition}

Several of the lemmas we proved for cancellation domains (Lemma~\ref{cancel_lemma1}
to Lemma~\ref{cancel_lemma6}) generalize almost immediately to the current setting:

\begin{lemma}
Suppose  $I_1, I_2$, and $J$ are 
fractional ideals of a weak cancellation domain where $J$ is finitely generated.
If $I_1 J = I_2 J$ then  $I_1 = I_2$.
\end{lemma}

\begin{lemma}
Suppose  $I_1, I_2$, and $J$ are 
fractional ideals of a weak cancellation domain where $J$ is finitely generated.
If $I_1 J \subseteq I_2 J$ then  $I_1 \subseteq I_2$.
\end{lemma}

\begin{lemma}
Suppose $x \in K^\times$ where $K$ is the field of fractions of a weak cancellation domain~$R$.
Then $xR \subseteq  x^2 R + R$.
\end{lemma}

\begin{lemma}\label{lemma155}
Suppose $x \in K^\times$ where $K$ is the field of fractions of a weak cancellation domain~$R$.
Then $a x^2 + x + b = 0$ for some $a, b \in R$.
\end{lemma}

\begin{lemma}\label{lemma174}
Every weak cancellation domain $R$  is integrally closed.
\end{lemma}

\begin{lemma}\label{lemma156}
Let $R$ be a weak cancellation domain~$R$ with field of fractions~$K$.
Let~$\mathfrak p$ be a prime ideal of~$R$.
Then for every $x \in K^\times$ either $x \in R_{\mathfrak p}$ or~$x^{-1} \in R_{\mathfrak p}$.
In other words, $R_{\mathfrak p}$ is a valuation ring.
\end{lemma}

So weak cancellation domains can differ from cancellation domains (i.e., almost Dedekind domains)
in that the local rings are valuation rings but not necessarily \emph{discrete} valuation rings.

\begin{lemma}\label{lemma157}
Suppose that $R$ is an integral domain and that 
$R_{\mathfrak m}$ is a valuation ring for every maximal ideal~$\mathfrak m$ of~$R$.
Then every nonzero finitely generated ideal $I$ is invertible. In other words, 
$R$ is a Pr\"ufer domain.
\end{lemma}

\begin{proof}
Assume $I$ is nonzero and finitely generated.
Then $I R_{\mathfrak m}$  is nonzero and finitely generated for each maximal ideal~$\mathfrak m$.
By assumption,~$R_{\mathfrak m}$ is a valuation ring, and so every finitely
generated ideal is principal (Exercise~\ref{valuation_ring_ex}). 
So $I R_{\mathfrak m}$ is principal, and hence invertible by Proposition~\ref{propInv}.
In particular, if $J = I I^{-1}$ then 
$$J R_{\mathfrak m} = (I R_{\mathfrak m})(I R_{\mathfrak m})^{-1} = R_{\mathfrak m}$$
(see Proposition~\ref{prop49}). This holds for each maximal ideal $\mathfrak m$,
so $J = R$ (Corollary~\ref{equality_test_cor}). Thus $I$ is invertible.
\end{proof}

\begin{theorem}
Let $R$ be an integral domain. Then $R$ is a Pr\"ufer domain
if and only if it is a weak cancellation domain.
\end{theorem}

\begin{proof}
We have already established one direction (Proposition~\ref{prop169}), so let $R$ be a weak cancellation domain.
By Lemma~\ref{lemma156},~$R_{\mathfrak m}$ is a valuation ring
for every nonzero maximal ideal~$\mathfrak m$.
So by Lemma~\ref{lemma157},~$R$ is a Pr\"ufer domain
\end{proof}

\begin{theorem}
Let $R$ be an integral domain. Then the following are equivalent
\begin{enumerate}
\item
$R$ is a Pr\"ufer domain.
\item
$R_{\mathfrak p}$ is a valuation ring for all prime ideals $\mathfrak p$.
\item
$R_{\mathfrak m}$ is a valuation ring for all maximal ideals $\mathfrak m$.
\end{enumerate}
\end{theorem}

\begin{proof}  Note that if $\mathfrak p$ is the zero ideal then $R_{\mathfrak p}$ is a valuation ring
since it is a field.

\noindent
$(1) \Rightarrow (2)$. This follows from Lemma~\ref{lemma156}.

\noindent
$(2) \Rightarrow (3)$. All maximal ideals are prime ideals.

\noindent
$(3) \Rightarrow (1)$. This follows from Lemma~\ref{lemma157}
\end{proof}

We restate Lemma~\ref{lemma174} in terms of Pr\"ufer domains.

\begin{theorem}\label{thm179}
Every Pr\"ufer domain is integrally closed.
\end{theorem}

\begin{example}
The following claim (which we state without proof) shows the importance of Pr\"ufer domains:

Suppose that $R$ is a Pr\"ufer domain with fraction field~$K$, and suppose that $L$ is an algebraic
extension of~$K$, then the integral closure of $R$ in~$L$ is also a Pr\"ufer domain.

For example, the ring of algebraic integers in the algebraic closure of $\bQ$, or in any algebraic
extension of~$\bQ$, is a Pr\"ufer domain.
\end{example}

\begin{exercise}
Recall that if $I$ and $J$ are ideals in a Dedekind domain then
$$ (I + J) (I \cap J)=I J$$
(see Exercise~\ref{ex28}). 
Suppose that $R$ is an integral domain where this identity holds.
Show that if $I$ and $J$ are invertible ideals then so are $I+J$ and $I\cap J$.
Conclude that all finitely generated fractional ideals must be invertible.
Conclude further that $R$ is a Pr\"ufer domain.
\end{exercise}

\begin{exercise}
(1) Why is every valuation ring a Pr\"ufer domain? Because of this fact,  every valuation ring is integrally closed 
by Theorem~\ref{thm179}.
(2) Give a simple direct proof that every valuation ring is integrally closed (i.e., a proof that does not use 
the notions of Pr\"ufer domain or weak cancellation domain).
\end{exercise}

Finally we mention (but do not develop) the ideal of a \emph{Krull domain}.
These were introduced by W.~Krull in 1931.
Suppose~$\{ v_i \}_{i\in I}$ is a family of discrete valuations of a field~$K$
such that, for each~$x \in K^\times$ there are only a finite number of $i\in I$
such that $v_i(x) \ne 0$.
Then the corresponding intersection intersection of DVRs
$$
R = \bigcap_{i\in I} \mathcal O_{v_i}
$$
is called the \emph{Krull domain} associated to $\{ v_i \}_{i\in I}$.
An integral domain is called a \emph{Krull domain} if it can be represented
as such an intersection for some such  $\{ v_i \}_{i\in I}$.

\begin{example}
Every Dedekind domain $R$ is a Krull domain. Here we take $\{ v_i \}_{i\in I}$
to be the collection of valuations associated with the set of nonzero prime ideals of $R$.

If $R$ happens to be a field $K$, then this gives an empty family. The empty intersection
in this case is  defined to be just $K$ itself.
\end{example}

\begin{example}
We won't justify the following, but these examples shows the scope 
of the notion of a Krull domain: 

If $R$ is Noetherian domain then it is a Krull domain 
if and only if $R$ is integrally closed. However, there are non Noetherian Krull domains. Even Noetherian Krull domains do not have to have
the property that every nonzero prime ideal is maximal. If $R$ is a Krull domain,
then so is $R[X]$. Every UFD is a Krull domain.
\end{example}

\begin{proposition}
Let $R$ be a Krull domain. Then $R$ is integrally closed.
\end{proposition}

\begin{proof}
Observe that $R$ is the intersection of integrally closed domains, so is integrally closed.
\end{proof}


\chapter*{Appendix F: Another route to unique factorization}

We know by Theorem~\ref{thm_principal_product} that an integral domain
is a Dedekind domain if and only if has the following property:
for each nonzero ideal $I$ there is a nonzero ideal~$J$ such that~$IJ$ is principal.
Call this the \emph{principal-complement property}.
This property is sometimes used as a route to unique factorization of ideals, and was
the main route before the 1920s.

In this appendix we explore a more concrete approach to Theorem~\ref{thm_principal_product}.
More specifically, we show how to derive unique factorization of ideals for
rings such as the ring of integers~$\mathcal O_K$
in an algebraic number field~$K$, assuming that we have somehow established the
principal-complement property for $\mathcal O_K$ (it doesn't matter how).
Our proof of the unique factorization of ideal (Theorem~\ref{thmF}) will be low-tech and
not require the concepts 
of Noetherian, fractional ideal, integrally closed, discrete valuation ring, or localization.
It does use, however,
the concepts of prime and maximal ideal, and
the correspondence of ideals for quotient rings (see Proposition~\ref{prop116} in Appendix~C),
concepts that are fairly ubiquitous in an introductory abstract algebra sequence.
(Of course, to establish the principal-complement property in the first place you will certainly
need some of the more advanced concepts --- certainly the integrally closed property since
this is a defining property of~$\mathcal O_K$.)

As mentioned above, this approach is important from a historical perspective.
In fact, before the abstract concept of a Dedekind domain arose with Emmy Noether,
proofs of the unique factorization of ideals in~$\mathcal O_K$, for instance Hurwitz's proof,
establish the principal-complement property for $\mathcal O_K$ as a step to unique factorization of ideals.
Such proofs also exploit the fact that $\mathcal O_K/I$ is finite for nonzero ideals~$I$.\footnote{See
Hilbert's famous \emph{Zahlbericht} of 1897 for an example of this approach. Hilbert states he is following Hurwitz.}

In this appendix we will explore this route to unique factorization of ideals in a ring~$R$
that has similarities to~$\mathcal O_K$. So up through Theorem~\ref{thmF} we assume
the following:
\begin{enumerate}
\item
$R$ is an integral domain such that, for each nonzero $I$, the ring $R/I$ is finite.
\item
For every nonzero ideal $I$ of $R$ there is a nonzero ideal $J$ such that $IJ$ is principal.
\end{enumerate}

Now we derive some consequences of these two assumptions:

\begin{proposition}\label{propF1}
Every nonzero prime ideal~$\mathfrak p$ of $R$ is a maximal ideal.
\end{proposition}

\begin{proof}
By assumption on~$R$ the quotient ring $R/\mathfrak p$ is a finite. 
However, every finite integral domain is a field (we leave this as an exercise).
Thus $\mathfrak p$ is maximal.
\end{proof}

\begin{proposition}\label{propF2}
Every proper nonzero ideal~$I$ of $R$ is contained in a prime ideal.
\end{proposition}

\begin{proof}
By assumption on~$R$, the quotient ring $R/I$ is finite.
Thus $R/I$ must have a maximal ideal, which must be a prime ideal.
By the correspondence of ideals, this gives a prime ideal of $R$ containing~$I$.
\end{proof}

\begin{proposition}\label{propF3}
The cancellation law holds in~$R$: if $I_1, I_2, J$ are nonzero ideals such 
that~$I_1 J = I_2 J$ then $I_1 = I_2$.
\end{proposition}

\begin{proof}
If $J$ is a nonzero principal ideal, then this is straightforward.
In general, let~$J'$ be a nonzero ideal such that $J J' = a R$.
Multiply both sides of $I_1 J = I_2 J$ by~$J'$ to get $I_1 J J' = I_2 J J'$.
Thus $I_1 (Ra)=  I_2 (Ra)$.  Hence $I_1 = I_2$.
\end{proof}

\begin{proposition}\label{propF4}
Let $I$ and $J$ be nonzero ideals of~$R$.
Then $J \subseteq I$ if and only if~$I \mid J$.
\end{proposition}

\begin{proof}
One direction is straightforward, so assume that $J \subseteq I$.
First we consider the case where $I = a R$.
If $I' = \{ b/a \mid b \in J\}$ then it is straightforward to show that~$I'$ is an ideal such that~$J = (a R) I'$.
In general, let $I''$ be such that $I I''$ is of the form $a R$.
Then $J I'' \subseteq I I'' = a R$. As before, there is an $I'$ such that
$$
J I'' = (a R) I' = (I I'') I' = (I I') I''.
$$
By cancellation, $J = I I'$.
\end{proof}

\begin{proposition}\label{propF5}
Let $\mathfrak p$ be a nonzero prime ideal of $R$. If $I$ and $J$ are nonzero ideals of $R$ such that $\mathfrak p \mid I J$
then $\mathfrak p \mid I$ or $\mathfrak p \mid J$.
\end{proposition}

\begin{proposition}\label{propF6}
Let $\mathfrak p_1, \mathfrak p_2$ be nonzero prime ideals of $R$. If $\mathfrak p_1 \mid \mathfrak p_2$ then $\mathfrak p_1 = \mathfrak p_2$.
\end{proposition}

\begin{lemma}
If $I \mid J$ but $I \neq J$ then $R/I$ has fewer elements than $R/J$.
\end{lemma}

\begin{proof}
This follows from the isomorphism $(R/J)/(I/J) \cong R/I$.
\end{proof}

\begin{theorem}\label{thmF}
Every nonzero proper ideal $I$ of $R$ is uniquely the product of nonzero prime ideals of~$R$.
(Uniqueness is up to order).
\end{theorem}

\begin{proof}
Let $\mathfrak p_1$ be a prime ideal containing~$I$ (Proposition~\ref{propF2}).
Observe that $\mathfrak p_1 \mid I$ so $I = \mathfrak p_1 I_1$ for some
nonzero ideal $I_1$ (Proposition~\ref{propF4}). If $I_1 = R$ we are done.
Otherwise, continue by writing $I_1 = \mathfrak p_2 I_2$ with $\mathfrak p_2$ a prime ideal
and $I_2$ a nonzero ideal.  We continue in this way.
If, at any point, $I_i = R$ we stop, and have the existence of a prime factorization of~$I$.

We still need to argue that this process stops. Observe that if $I_i$ is a proper ideal
then $I_i \ne I_{i+1}$. Otherwise $I_i = \mathfrak p_{i+1} I_{i+1} =  \mathfrak p_{i+1} I_i$, 
which implies $\mathfrak p_{i+1} = R$
by the cancellation law. This contradicts the definition of prime ideal.
So, by the above lemma, $R / I_{i+1}$ has fewer elements
than~$R / I_{i}$. Since each $R/I_{i}$ is finite we must eventually have $I_{i} = R$.

Finally, we prove uniqueness in the usual way (using Proposition~\ref{propF5}, Proposition~\ref{propF6},
and the cancellation law).
\end{proof}

The above approach quite simple and appealing, so much so that 
it is natural to suspect that the primary difficulty in pursuing this route is showing 
the principal-complement property. How do we establish this property in a relatively
concrete manner? There are two 
approaches that I am aware of. One is related to Gauss's lemma, and was used by Hurwitz
in one of his proofs.\footnote{See E.~Hecke, \emph{Vorlesungen  \"uber die Theorie
der algebraischen Zahlen} for a proof which Hecke attributes to Hurwitz, with simplifications 
he attributes to Steinitz,
that uses a basic form of Gauss's lemma (see also Lemma 2 of Hilbert's \emph{Zahlbericht}
for the basic form of Gauss's lemma needed). Note: what I am calling ``Gauss's lemma'' 
(See Section~\ref{ch_gauss_lemma} above)
is a generalization of the classical Gauss's lemma for $\bZ[X]$ 
and this generalization is related to Kronecker's theory of forms
developed the 19th century.
}
We will not pursue the Gauss's lemma approach here.\footnote{I can mention that it
depends on proving a special case of Gauss's lemma for integrally closed integral domains, and then setting up a polynomial
$f$ such that $I$ is the content of~$f$. One then  finds another polynomial $g$ such
that the content of $f g$ is principal. The desired $J$ is then the content of $g$.}
 The other is related to the finiteness of the class group of~$\mathcal O_K$.\footnote{A proof based
 on proving the class number is finite is given in \emph{A Classical Introduction to Modern Number Theory},
 by K.~Ireland and M~Rosen. Their approach is also based on results of Hurwitz.}
We will not discuss how one would establish this finiteness claim, but will instead show how this finiteness claim
can be used to establish the principal-complement property.  We will work in a general integral domain $R$,
but of course in practice we have in mind the ring of integers~$\mathcal O_K$ in an algebraic number field~$K$.
(Dedekind domains arising in algebraic geometry typically have infinite class group.)
In what follows we will make use of fractional ideals and the notions of Noetherian
and integrally closed.

\begin{definition} \label{class_monoid_def1}
Two  fractional ideals $I$ and $J$ in an integral domain~$R$ are \emph{equivalent
modulo principal ideals} if $I = (x R) J$ for some principal fractional ideal~$xR$.
\end{definition}

\begin{remark}
If you prefer to just use integral ideals, then $I$ and $J$ are equivalent if and only if
$a I = b J$ for some $a, b \in R$ nonzero.
\end{remark}

\begin{proposition}
The above relation is an equivalence relation on the set of 
fractional ideals. Every class contains an integral ideal.
\end{proposition}

\begin{lemma}
Suppose $I_1$ and $I_2$ are fractional ideals that are equivalent modulo principal ideals. 
Then $I_1 J$ is equivalent to $I_2 J$
for all fractional ideals $J$.
\end{lemma}

\begin{definition} \label{class_monoid_def2}
Let $R$ be an integral domain. Let $[I]$ and $[J]$ be equivalence classes under the relation 
of Definition~\ref{class_monoid_def1}. Then define the product $[I] [J]$ as $[IJ]$.
This is well-defined by the above lemma. 
We call the set of such equivalence classes under products the \emph{class monoid of~$R$}.
\end{definition}

\begin{proposition}
The class monoid of an integral domain~$R$ is a commutative monoid with unit $[R]$.
\end{proposition}

\begin{theorem}
Let $R$ be an integrally closed Noetherian ring
whose class monoid~$R$ is finite. Then for every
ideal $I$ there is an ideal $J$ such that~$I J$ is principal.
\end{theorem}


\begin{proof}
By finiteness,
$[I]^m = [I]^{n}$ for integers $m, n$ with $0 < m < n$. 
This means that~$I^m = I^{n} (x R)$ for some principal fractional ideal $x R$.
We can write this as
$$
 I^m (a R) = I^m I^k (b R)
$$
where $k = n - m$, and $a, b \in R$ nonzero such that $x = a/b$.
By Exercise~\ref{ex15} we cancel to get $aR = I^k (R b)$. 
Let $J$ be~$I^{k-1}(R b)$.
\end{proof}

\begin{remark}
This approach is fairly concrete.
This proof uses Exercise~\ref{ex15}, which in turn
uses standard properties resulting from~$R$ being integrally closed. Exercise~\ref{ex15}  also
depends on Exercise~\ref{exx} which can be established using linear algebra.
The proof uses fractional ideals for convenience, but it is easily adapted to an approach that uses integral ideals only.
The notion of Noetherian can be replaced by showing directly in $\mathcal O_K$ 
that ideals (and fractional ideals) have a finite $\bZ$-basis.

The above result also establishes invertibility of each $[I]$, so the class monoid is seen to be a group.
Observe that for a Dedekind domain, the class monoid is just the class group
of fractional ideals modulo fractional principal ideals.
\end{remark}

\begin{exercise}
Show that two fractional ideals $I$ and $J$ of an integral domain $R$ are isomorphic as $R$-modules
if and only if they are equivalent modulo principal ideals (Definition~\ref{class_monoid_def1}).
So in some sense the class monoid classifies isomorphism types of fractional ideals.
As a first step, show that any isomorphism between~$R$-submodules of the fraction field~$K$
is of the form $x \mapsto c x$ where $c \in K^\times$.

Hint: if $\varphi: I\to J$ is an isomorphism between $R$-submodules of~$K$,
and if one nonzero value is given,  $x_2 = \varphi(x_1)$ say, show that $\varphi(x) = c x$
where $c = x_2/x_1$.
\end{exercise}


\end{document}